\documentclass[12pt]{amsart}
\usepackage{amssymb,mathrsfs, upgreek}
\usepackage[matrix,arrow,curve]{xy}
\usepackage{hyperref}
\usepackage{tabularx,makecell}
\newcolumntype{Y}{>{\centering\arraybackslash}X}
\usepackage{setspace}
\newtheorem{theorem}[equation]{Theorem}
\newtheorem*{theorem*}{Theorem}
\newtheorem{lemma}[equation]{Lemma}
\newtheorem{corollary}[equation]{Corollary}

\newtheorem{proposition}[equation]{Proposition}

\theoremstyle{definition}
\newtheorem{example}[equation]{Example}
\newtheorem{definition}[equation]{Definition}
\newtheorem*{definition*}{Definition}

\theoremstyle{remark}
\newtheorem{remark}[equation]{Remark}
\renewcommand{\rho}{\uprho}

\newcommand{\xref}[1]{\textup{\ref{#1}}}

\makeatletter%
\@addtoreset{equation}{subsection}%
\makeatother%

\renewcommand\labelenumi{(\arabic{enumi})}

\def \O {\mathscr{O}}

\def \P {\mathbb{P}}

\def \Q {\mathbb{Q}}
\def \C {\mathbb{C}}
\def \Z {\mathbb{Z}}
\def \F {\mathbf{F}}
\def \SS {\mathfrak{S}}
\def \A {\mathfrak{A}}
\def \Gr {\mathrm{Gr}}
\def \SO {\mathrm{SO}}
\def \PSO {\mathrm{PSO}}
\def \mumu {\boldsymbol{\mu}}

\newcommand{\Cr}{\operatorname{Cr}}
\def \Supp {\mathrm{Supp}\,}

\def \Bs {\mathrm{Bs}\,}
\def \Bir {\mathrm{Bir}}
\def \Aut {\mathrm{Aut}}
\def \Cl {\mathrm{Cl}\,}
\def \Pic {\mathrm{Pic}\,}
\def \Sym {\mathrm{Sym}\,}

\newcommand{\dd}{\operatorname{d}}
\newcommand{\Id}{\operatorname{Id}}
\newcommand{\diag}{\operatorname{diag}}
\newcommand{\red}{\operatorname{red}}
\newcommand{\g}{\operatorname{g}}

\def \Sing {\mathrm{Sing}\,}
\def \GL {\mathrm{GL}}
\def \SL {\mathrm{SL}}
\def \PGL {\mathrm{PGL}}
\def \PSL {\mathrm{PSL}}
\def \PSp {\mathrm{PSp}}
\def \Sp {\mathrm{Sp}}

\def \Jac {\mathrm{Jac}}
\newcommand{\Fix}{\operatorname{Fix}}
\newcommand{\Cox}{\operatorname{Cox}}
\newcommand{\Ker}{\mathop{\mathsf{Ker}}}

\newcommand{\CO}{{\mathscr{O}}}

\def \ge {\geqslant}
\def \le {\leqslant}

\renewcommand\labelenumi{(\roman{enumi})}
\renewcommand\theenumi{(\roman{enumi})}

\title{Jordan constant for Cremona group of rank~$3$}

\author{Yuri Prokhorov}
\author{Constantin Shramov}

\thanks{This work was supported by the Russian Science Foundation under grant 14-50-00005.}

\address{
Steklov Mathematical Institute of Russian Academy of Sciences, 8 Gubkina st., Moscow, Russia, 119991
}

\email{prokhoro@mi.ras.ru, costya.shramov@gmail.com}

\begin{document}

\begin{abstract}
We give explicit bounds for Jordan constants of groups of birational automorphisms of rationally
connected threefolds over fields of zero characteristic, in particular, for
Cremona groups of ranks~$2$ and~$3$.

\bigskip
\noindent
2010 MATH. SUBJ. CLASS. 14E07, 14J30

\end{abstract}

\maketitle
\tableofcontents

\section{Introduction}
\label{section:intro}

\subsection{Jordan property}

The \textit{Cremona group of rank $n$} is the group $\Cr_n(\Bbbk)$ of birational transformations of
the projective space $\mathbb P^n$ over a field $\Bbbk$.
It has been actively studied from various points of view for many years
(see~\cite{Hudson1927}, \cite{CantatLamy}, \cite{Deserti2012}, \cite{Dolgachev-Iskovskikh}, \cite{Serre-2008-2009}, \cite{Cantat2016}, and references therein).
One of the approaches to this huge group is to
try to understand its finite subgroups.
It appeared that it is possible to obtain a complete
classification of finite subgroups of $\Cr_2(\Bbbk)$ over an algebraically closed
field $\Bbbk$ of characteristic $0$ (see~
\cite{Bayle-Beauville-2000}, \cite{Beauville2004}, \cite{Blanc2009}, \cite{Dolgachev-Iskovskikh},
\cite{Tsygankov2011e}, and~\cite{Prokhorov-stable-conjugacy-II}),
and to obtain partial classification results
for~\mbox{$\Cr_3(\Bbbk)$}
(see~\cite{Prokhorov2009e},
\cite{Prokhorov2011a},
\cite{Prokhorov-2-elementary},
\cite{Prokhorov-2013-im}, and~\cite{Prokhorov-planes}).
Some results are also known for non algebraically closed fields,
see e.g. \cite{Serre2009}, \cite{Dolgachev-Iskovskikh-2009}, and~\cite{Yasinsky2016}.
In general, it is partially known and partially expected that
the collection of finite subgroups of a Cremona group
shares certain features with the collection of finite subgroups
of a group~\mbox{$\GL_m(\Bbbk)$}.

\begin{theorem}[C.\,Jordan, {see e.\,g.~\cite[Theorem 36.13]{Curtis-Reiner-1962}}]
\label{theorem:Jordan}
There is a constant~\mbox{$I=I(n)$} such that for any
finite subgroup $G\subset\GL_n(\C)$
there exists
a normal abelian subgroup~\mbox{$A\subset G$} of index at most $I$.
\end{theorem}

This leads to the following definition
(cf.~\cite[Definition~2.1]{Popov2011}).

\begin{definition}
\label{definition:Jordan}
A group $\Gamma$ is called \emph{Jordan}
(alternatively, we say
that~$\Gamma$ \emph{has
Jordan property})
if there is a constant $J$ such that
for any finite subgroup $G\subset\Gamma$ there exists
a normal abelian subgroup~\mbox{$A\subset G$} of index at most $J$.
\end{definition}

Theorem~\xref{theorem:Jordan} implies that
all linear algebraic groups
over an arbitrary field $\Bbbk$ of characteristic $0$
are Jordan.
Jordan property was also studied recently for groups
of birational automorphisms of algebraic varieties.
The starting point here was the following result of
J.-P.\,Serre.

\begin{theorem}[J.-P.\,Serre
{\cite[Theorem~5.3]{Serre2009}, \cite[Th\'eor\`eme~3.1]{Serre-2008-2009}}]
The Cremona group~\mbox{$\Cr_2(\Bbbk)$} over a field $\Bbbk$ of characteristic $0$ is Jordan.
\end{theorem}

\begin{remark}
Note that the assumption about characteristic is indispensable.
Indeed,
the group $\Cr_2(\Bbbk)$ contains $\PGL_2(\Bbbk)$, so that
if the characteristic of the field
$\Bbbk$ equals~\mbox{$p>0$} and $\Bbbk$ is algebraically closed,
then $\Cr_2(\Bbbk)$ contains
a series of simple subgroups~\mbox{$\PSL_2(\F_{p^k})$} of increasing order.
\end{remark}

It also appeared that there are surfaces with
non-Jordan groups of birational selfmaps (see~\cite{Zarhin10}).
V.\,Popov managed to give a complete
classification of surfaces with Jordan groups
of birational automorphisms.

\begin{theorem}[V.\,Popov {\cite[Theorem~2.32]{Popov2011}}]
Let $S$ be a surface over a field $\Bbbk$ of characteristic $0$.
Then the group
$\Bir(S)$ of birational automorphisms of~$S$
is Jordan if and only if~$S$ is not
birational to~\mbox{$E\times\P^1$}, where $E$ is an
elliptic curve.
\end{theorem}

In dimension $3$ Jordan property is known for
groups of birational automorphisms of
rationally connected varieties (see e.\,g.~\cite[\S\,IV.3]{Kollar-1996-RC}
for definition and basic background).

\begin{theorem}[{see~\cite[Theorem~1.8]{ProkhorovShramov-RC}}]
\label{theorem:RC-Jordan}
Fix a field $\Bbbk$ of characteristic~$0$.
Then there is a constant~$J$ such that
for any rationally connected
variety $X$ of dimension~$3$ defined over~$\Bbbk$
and any finite subgroup $G\subset\Bir(X)$ there exists
a normal abelian subgroup~\mbox{$A\subset G$} of index at most $J$.
In particular, for any rationally connected threefold $X$ the group~\mbox{$\Bir(X)$} is Jordan.
\end{theorem}

Actually, by~\cite[Theorem~1.8]{ProkhorovShramov-RC}
the assertion of Theorem~\xref{theorem:RC-Jordan} holds in
arbitrary dimension modulo boundedness of terminal Fano varieties
(see e.\,g.~\cite{Borisov-1996} or~\cite[Conjecture~1.7]{ProkhorovShramov-RC});
the latter boundedness was recently proved in~\cite[Theorem~1.1]{Birkar}.
For other results (in particular, for birational automorphisms of non rationally connected varieties,
and for automorphisms of varieties of different types) see~\cite{Prokhorov-Shramov-J},
\cite{Popov2011}, \cite{Popov-Jordan},
\cite{Zarhin2015},
\cite{BandmanZarhin2015},
\cite{BandmanZarhin2015a}, \cite{MengZhang},
\cite{Popov-Diff},
\cite{Zimmermann2014}, \cite{TurullRiera2015}, \cite{Riera2016},
and~\cite{Yasinsky2017}.

\subsection{Jordan constants}

Given a Jordan group $\Gamma$, one may get interested
in the minimal value of the constant involved in
Definition~\xref{definition:Jordan}, and in the values of
other relevant constants.

\begin{definition}
\label{definition:Jordan-constant}
Let $\Gamma$ be a Jordan group.
The \emph{Jordan constant}
$J(\Gamma)$ of the group $\Gamma$
is the minimal number $J$ such that
for any finite subgroup $G\subset\Gamma$ there exists
a normal abelian subgroup
$A\subset G$ of index at most $J$.
The \emph{weak Jordan constant}
$\bar{J}(\Gamma)$ of the group~$\Gamma$
is the minimal number $\bar{J}$ such that
for any finite subgroup $G\subset\Gamma$ there exists
a (not necessarily normal) abelian subgroup
$A\subset G$ of index at most $\bar{J}$.
\end{definition}

\begin{remark}\label{remark:Pyber}
It is more traditional to study Jordan constants than weak Jordan
constants of Jordan groups, although there is no big difference between
them. Indeed, one has~\mbox{$\bar{J}(\Gamma)\le J(\Gamma)$} for any Jordan
group $\Gamma$ for obvious reasons.
Moreover, if $G$ is a finite group and $A$ is an abelian
subgroup of $G$, then by~\cite[Theorem~1.41]{Isaacs2008}
(see also~\cite{ChermakDelgado}) one can find
a normal abelian subgroup $N$ of $G$ such that
\begin{equation*}
[G:N]\le [G:A]^2.
\end{equation*}
Therefore, if $\Gamma$ is a Jordan group, one always
has $J(\Gamma)\le\bar{J}(\Gamma)^2$. On the other hand, the advantage of
the weak Jordan constant is that it allows easy estimates
using subgroups of the initial group. Namely, if $\Gamma_1$ is a subgroup
of finite index in a group $\Gamma_2$, and $\Gamma_1$ is Jordan,
then $\Gamma_2$ is Jordan with
\begin{equation*}
\bar{J}(\Gamma_2)\le [\Gamma_2:\Gamma_1]\cdot\bar{J}(\Gamma_1).
\end{equation*}
Also, if $\Delta_1$ and $\Delta_2$ are Jordan groups,
the group $\Delta_1\times\Delta_2$ is Jordan with
\begin{equation*}
\bar{J}(\Delta_1\times\Delta_2)= \bar{J}(\Delta_1)\times\bar{J}(\Delta_2).
\end{equation*}
In particular, if $\Gamma$ is a subgroup of $\Delta\times A$, where $\Delta$ is a Jordan group
and $A$ is an abelian group, then $\Gamma$ is Jordan with~\mbox{$\bar{J}(\Gamma)\le\bar{J}(\Delta)$}.
\end{remark}

Jordan constants are known for example for the groups $\GL_n(\C)$
(see~\cite{Collins2007}).
In~\cite{Serre2009} J.-P.\,Serre gave an explicit
bound for the Jordan constant
of the Cremona group $\Cr_2(\Bbbk)$
(see Remark~\xref{remark:dim-2-cool} below).
Our first result also concerns the group~\mbox{$\Cr_2(\Bbbk)$}.

\begin{proposition}\label{proposition:Cr-2}
Suppose that the field $\Bbbk$ has characteristic~$0$. Then one has
\begin{equation*}
\bar{J}\big(\Cr_2(\Bbbk)\big)\le 288,\quad J\big(\Cr_2(\Bbbk)\big)\le 82944.
\end{equation*}
The first of these bounds becomes an equality if $\Bbbk$ is algebraically closed.
\end{proposition}

The main goal of this paper is to present a bound
for Jordan constants of the groups of birational automorphisms
of rationally connected threefolds, in particular, for the
group~\mbox{$\Cr_3(\Bbbk)=\Bir(\P^3)$}.

\begin{theorem}\label{theorem:constant}
Let $X$ be a rationally connected threefold over a field
$\Bbbk$ of characteristic~$0$. Then one has
\begin{equation*}
\bar{J}\big(\Bir(X)\big)\le 10\,368,\quad
J\big(\Bir(X)\big)\le 107\,495\,424.
\end{equation*}
If moreover $X$ is rational and $\Bbbk$ is algebraically closed,
then the first of these bounds becomes an equality.
\end{theorem}

It is known (see~\cite[Theorem~1.10]{ProkhorovShramov-RC})
that if $X$ is a rationally connected threefold over a field
of characteristic~$0$, then there is a constant $L$ such that for
any prime $p>L$ any finite $p$-group
$G\subset\Bir(X)$ is abelian.
An immediate consequence of Theorem~\xref{theorem:constant}
is an explicit bound for the latter constant $L$.

\begin{corollary}\label{corollary:prime}
Let $X$ be a rationally connected threefold over a field
$\Bbbk$ of characteristic~$0$, and let $p>10\,368$ be a prime.
Let $G\subset\Bir(X)$ be a finite $p$-group.
Then $G$ is abelian.
\end{corollary}

We believe that one can significantly improve the bound given by
Corollary~\xref{corollary:prime}.

\begin{remark}\label{remark:dim-2-cool}
J.-P.\,Serre showed (see the remark made after
Theorem~5.3 in~\cite{Serre2009})
that any finite subgroup $G$ of the Cremona group $\Cr_2(\Bbbk)$
over a field $\Bbbk$ of characteristic~$0$
has a normal abelian subgroup $A\subset G$ such that
the index $[G:A]$ \emph{divides} the number~\mbox{$2^{10}\cdot 3^4\cdot 5^2\cdot 7$}.
The result of Theorem~\xref{theorem:constant}
is not that precise: we cannot say much about the primes that divide
the index~\mbox{$[G:A]$} in our case.
This is explained by the fact that to obtain the bound we
have to deal with terminal singularities
on threefolds as compared to smooth surfaces.
See Remark~\xref{remark-P1-P1-P1}
for our expectations on possible improvements of the bounds given by Proposition~\xref{proposition:Cr-2} and
Theorem~\xref{theorem:constant}, and Remark~\xref{remark:dim-4-fail}
for a further disclaimer in higher dimensions.
\end{remark}

\smallskip
The plan of the paper is as follows.
In~\S\xref{section:GL} we compute weak Jordan constants for some
linear groups.
In~\S\xref{section:dim-2} we compute certain relevant constants for
rational surfaces, and in particular prove Proposition~\ref{proposition:Cr-2}.
In~\S\xref{section:terminal} we study groups
of automorphisms of three-dimensional terminal singularities
and estimate their weak Jordan constants;
then we use these estimates to bound weak Jordan constants for
groups of automorphisms of non-Gorenstein terminal Fano threefolds.
In~\S\xref{section:Mfs} we estimate weak Jordan constants
for groups acting on three-dimensional $G$-Mori fiber spaces.
In~\S\xref{section:Fano} and~\S\xref{section:smooth-Fano} we bound weak Jordan constants for
groups of automorphisms of Gorenstein terminal (and in particular smooth) Fano threefolds.
Finally, in~\S\xref{section:proof} we summarize the above partial results
and complete the proof of Theorem~\xref{theorem:constant}, and also make concluding remarks.
In Appendix~\xref{section:appendix}
we collect some information about automorphism groups of two particular classes
of smooth Fano varieties: complete intersections of quadrics, and
complete intersections in weighted projective spaces; these results are well known to experts,
but we decided to include them because we do not know proper references.

\smallskip
\textbf{Acknowledgments.} We would like to thank J.-P.\,Serre
who attracted our attention
to the questions discussed in this paper.
We are also grateful to
A.\,Anan'in, J.\,Hausen, A.\,Kuznetsov, A.\,Laface, L.\,Pyber, V.\,Popov, L.\,Rybnikov,
and H.\,Suess for useful discussions, and to the referee for his helpful comments.

\smallskip
\textbf{Notation and conventions.}
In what follows
we denote by $\mumu_m$ a cyclic group of order $m$.
We denote by $m.\Gamma$ a central extension
of a group $\Gamma$ by a group isomorphic to $\mumu_m$.
Starting from this point we always work over an \emph{algebraically closed}
field of characteristic~$0$.

\section{Linear groups}
\label{section:GL}

Now we are going to find weak Jordan constants~\mbox{$\bar{J}\big(\GL_n(\Bbbk)\big)$}
for small values of $n$.
Note that the values of Jordan constants~\mbox{$J\big(\GL_n(\Bbbk)\big)$}
were computed in~\cite{Collins2007} for any~$n$.

\subsection{Preliminaries}
\label{subsection:linear-prelim}

The following remark is elementary but rather useful.

\begin{remark}\label{remark:surjection}
Suppose that $\Gamma_1$ is a Jordan group, and there is
a surjective homomorphism~\mbox{$\Gamma_1\to\Gamma_2$} with finite kernel. Then
$\Gamma_2$ is also Jordan. Moreover, one has~\mbox{$\bar{J}(\Gamma_1)\ge \bar{J}(\Gamma_2)$}
and~\mbox{$J(\Gamma_1)\ge J(\Gamma_2)$}.
In particular, for any $n$ the group
$\PGL_n(\Bbbk)$ is Jordan with
\begin{equation*}
\bar{J}\big(\PGL_n(\Bbbk)\big)\le\bar{J}\big(\SL_n(\Bbbk)\big)=\bar{J}\big(\GL_n(\Bbbk)\big),
\quad
J\big(\PGL_n(\Bbbk)\big)\le J\big(\SL_n(\Bbbk)\big)=J\big(\GL_n(\Bbbk)\big).
\end{equation*}
\end{remark}

We will also need the following well-known observation.
Let $U$ be an arbitrary variety and~$P$ be a point of~$U$.
Denote by $\Aut_P(U)$ the stabilizer of $P$ in $\Aut(U)$.
Let~\mbox{$T_P(U)$} be the Zariski tangent space
to the variety~$U$ at the point~$P$.

\begin{lemma}[{see e.\,g. \cite[Lemma 2.4]{Bialynicki-Birula1973}, \cite[Lemma 4]{Popov-Jordan}}]
\label{lemma:Aut-P}
Suppose that $U$ is an irreducible variety.
For any finite group~\mbox{$G\subset\Aut_P(U)$} the
natural representation~\mbox{$G\to\GL\big(T_P(U)\big)$}
is faithful.
In particular, one has
\begin{equation*}
\bar{J}\big(\Aut_P(U)\big)\le \bar{J}\Big(\GL\big(T_P(U)\big)\Big),
\quad
J\big(\Aut_P(U)\big)\le J\Big(\GL\big(T_P(U)\big)\Big).
\end{equation*}
\end{lemma}

\begin{remark}
One does not necessarily have an embedding~\mbox{$\Gamma\hookrightarrow\GL\big(T_P(U)\big)$}
for a non-reductive subgroup $\Gamma\subset\Aut_P(U)$.
This is not the case already for $U\cong\mathbb{A}^2$
and~\mbox{$\Gamma=\Aut_P(U)$}.
\end{remark}

\subsection{Dimension $2$}
\label{subsection:GL2}

The following easy result will be used both to find the weak Jordan constant
of the group $\GL_2(\Bbbk)$, and also later in the proof of Corollary~\xref{corollary:genus-6}.

\begin{lemma}\label{lemma:2-PGL2}
Let $G$ be a group that fits into an exact sequence
\begin{equation*}
1\to \Gamma\longrightarrow G\stackrel{\phi}\longrightarrow\PGL_2(\Bbbk),
\end{equation*}
where $\Gamma\cong\mumu_2$.
Then $G$ is Jordan with~\mbox{$\bar{J}(G)\le 12$}.
\end{lemma}
\begin{proof}
Note that $\Gamma$ is contained in the center of the group~$G$.
We may assume that $G$ is finite.
By the well-known classification of finite subgroups
in $\PGL_2(\Bbbk)$, we know that the group $\bar{G}=\phi(G)$ is
either cyclic, or dihedral, or isomorphic to one of the groups $\A_4$,
$\SS_4$, or~$\A_5$.

If $\bar{G}$ is cyclic, then the group $G$ is abelian.

If $\bar{G}$ is dihedral, then the group $G$ contains an abelian subgroup of index~$2$.

If $\bar{G}\cong\A_4$, then $\bar{G}$ contains a cyclic subgroup of order~$3$,
so that $\bar{J}(G)\le 4$; the inequality here is due to the fact that
in the case when $G\cong\mumu_2\times\A_4$ one has $\bar{J}(G)=3$, but for a non-trivial
central extension $G\cong 2.\A_4$ one has $\bar{J}(G)=4$.

If $\bar{G}\cong\SS_4$, then $\bar{G}$ contains a cyclic subgroup of order~$4$, and $\bar{J}(G)=6$.

Finally, if $\bar{G}\cong\A_5$, then $\bar{G}$ contains a cyclic subgroup of order~$5$, and $\bar{J}(G)=12$.
\end{proof}

As an easy application of Lemma~\xref{lemma:2-PGL2}, we can
find the weak Jordan constants of the groups $\GL_2(\Bbbk)$ and~\mbox{$\Aut(\P^1)\cong\PGL_2(\Bbbk)$}.

\begin{corollary}
\label{corollary:GL2}
One has
\begin{equation*}
\bar{J}\big(\GL_2(\Bbbk)\big)=\bar{J}\big(\PGL_2(\Bbbk)\big)=12.
\end{equation*}
\end{corollary}

\begin{proof}
Let $V$ be a
two--dimensional
vector space over $\Bbbk$, and let $G\subset\GL(V)$ be a finite subgroup.
It is enough to study
the weak Jordan constant
$\bar{J}(G)$. Moreover, for this we may assume that
$G\subset\SL(V)\cong\SL_2(\Bbbk)$, and that $G$ contains the scalar matrix
acting by~$-1$ on~$V$. Therefore, the bound $\bar{J}(G)\le 12$
follows from Lemma~\xref{lemma:2-PGL2}, so that~\mbox{$\bar{J}(\GL_2(\Bbbk))\le 12$}.
The inequality
\begin{equation*}
\bar{J}\big(\PGL_2(\Bbbk)\big)\le\bar{J}\big(\GL_2(\Bbbk)\big)
\end{equation*}
holds by Remark~\xref{remark:surjection}.
The value $\bar{J}(\PGL_2(\Bbbk))=12$ is given by the group~\mbox{$\A_5\subset\PGL_2(\Bbbk)$},
and the value $\bar{J}(\GL_2(\Bbbk))=12$ is given by the group~\mbox{$2.\A_5\subset\GL_2(\Bbbk)$}.
\end{proof}

\begin{remark}\label{remark:elliptic}
Suppose that $C$ is an irreducible curve
such that the normalization
$\hat{C}$ of~$C$ has genus $g$.
Since the action of the group $\Aut(X)$ lifts to $\hat{C}$, one has
$$
\bar{J}\big(\Aut(C)\big)\le \bar{J}\big(\Aut(\hat{C})\big).
$$
On the other hand, it is well known that
$\bar{J}\big(\Aut(\hat{C})\big)\le 6$
if $g=1$, and the Hurwitz bound implies that
\begin{equation*}
\bar{J}\big(\Aut(\hat{C})\big)\le |\Aut(\hat{C})|\le 84(g-1)
\end{equation*}
if $g\ge 2$.
\end{remark}

We can use a classification of finite subgroups
in $\PGL_2(\Bbbk)$ to find the weak Jordan
constant of the automorphism group
of a line and a smooth two-dimensional quadric.
More precisely, we have the following result.

\begin{lemma}\label{lemma:weak-constant-quadric}
The following assertions hold.
\begin{enumerate}
\item \label{lemma:weak-constant-quadric-i}
Let $G\subset\Aut(\P^1)$ be a finite group.
Then there exists an abelian subgroup $A\subset G$ of index
at most $12$ acting on $\P^1$ with a fixed point.
\item\label{lemma:weak-constant-quadric-ii}
Let $G\subset\Aut\big(\P^1\times\P^1\big)$ be a finite group.
Then there exists an abelian subgroup~\mbox{$A\subset G$} of index
at most $288$ that acts on $\P^1\times\P^1$ with a fixed point,
and does not interchange the rulings of $\P^1\times\P^1$.

\item\label{lemma:weak-constant-quadric-iii}
One has
\begin{equation*}
\bar{J}\Big(\Aut\big(\P^1\times\P^1\big)\Big)=288.
\end{equation*}
\end{enumerate}
\end{lemma}

\begin{proof}
Assertion (i) follows from the classification of finite subgroups of $\PGL_2(\Bbbk)$.
Observe that
\begin{equation*}
\Aut\big(\P^1\times\P^1\big)\cong
\big(\PGL_2(\Bbbk)\times\PGL_2(\Bbbk)\big)\rtimes\mumu_2.
\end{equation*}
Therefore, assertion (i) implies assertion (ii).
In particular, we get the bound
\begin{equation*}
\bar{J}\Big(\Aut\big(\P^1\times\P^1\big)\Big)\le 288.
\end{equation*}
The required equality is given by the group
\begin{equation*}
\big(\A_5\times\A_5\big)\rtimes\mumu_2\subset\Aut\big(\P^1\times\P^1\big).
\end{equation*}
This proves assertion~(iii).
\end{proof}

\subsection{Dimension $3$}
\label{subsection:GL3}

\begin{lemma}\label{lemma:weak-GL3}
One has
\[
\bar{J}\big(\PGL_3(\Bbbk)\big)=40,
\qquad \bar{J}\big(\GL_3(\Bbbk)\big)=72.
\]
\end{lemma}
\begin{proof}
Let $V$ be a three-dimensional vector
space over $\Bbbk$, and let $G\subset\GL(V)$ be a finite subgroup.
It is enough to study the weak Jordan constant
$\bar{J}(G)$. Moreover, for this we may assume that
$G\subset\SL(V)\cong\SL_3(\Bbbk)$.
Recall that there are the following possibilities
for the group $G$ (see~\cite[Chapter~XII]{MBD1916}
or~\cite[\S8.5]{Feit}):
\begin{enumerate}
\item the $G$-representation $V$ is reducible;
\item there is a transitive homomorphism $h\colon G\to\SS_{3}$ such that
$V$ splits into a sum of three one-dimensional representations of
the subgroup $H=\Ker(h)$;
\item the group $G$ is generated by some subgroup
of scalar matrices in $\SL_3(\Bbbk)$ and a group $\hat{G}$ that
is one of the groups $\A_5$ or $\PSL_2(\F_7)$;
\item one has $G\cong 3.\A_6$;
\item one has $G\cong\mathcal{H}_3\rtimes \Sigma$,
where $\mathcal{H}_3$ is the Heisenberg group of order $27$, and
$\Sigma$ is some subgroup of $\SL_2(\F_3)$.
\end{enumerate}
Let us denote by $\bar{G}$ the image of $G$ in the group $\PGL_3(\Bbbk)$.
One always has~\mbox{$\bar{J}(\bar{G})\le\bar{J}(G)$}.

In case~(i) there is an embedding $G\hookrightarrow A\times\Gamma$,
where $A$ is a finite abelian group and $\Gamma$ is a finite subgroup
of $\GL_2(\Bbbk)$. Thus
\begin{equation*}
\bar{J}(\bar{G})\le\bar{J}(G)=\bar{J}(\Gamma)\le
\bar{J}\big(\GL_2(\Bbbk)\big)=\bar{J}\big(\PGL_2(\Bbbk)\big)=12
\end{equation*}
by Corollary~\xref{corollary:GL2}.

In case~(ii) the group $H$ is an abelian subgroup
of $G$, so that
\begin{equation*}
\bar{J}(\bar{G})\le\bar{J}(G)\le [G:H]\le |\SS_3|=6.
\end{equation*}

In case~(iii) it is easy to check that
$\bar{J}(\bar{G})\le\bar{J}(G)=\bar{J}(\hat{G})\le 24$.

In case (iv) one has $G\cong 3.\A_6$ and $\bar{G}\cong\A_6$. The
abelian subgroup of maximal order in $\bar{G}$
is a Sylow $3$-subgroup,
so that $\bar{J}(\bar{G})=40$.
The abelian subgroup of maximal order in $G$ is~$\mumu_{15}$
that is a preimage of a Sylow $5$-subgroup with respect
to the natural projection $G\to\bar{G}$.
This gives $\bar{J}(G)=72$.

In case~(v) one has
\begin{equation*}
\bar{J}(\bar{G})\le\bar{J}(G)\le
\bar{J}\big(\mathcal{H}_3\rtimes\SL_2(\F_3)\big)
=24.
\end{equation*}

Therefore, we see that $\bar{J}(\PGL_3(\Bbbk))=40$
and $\bar{J}\big(\GL_3(\Bbbk)\big)=72$.
\end{proof}

\begin{lemma}\label{lemma:PGL3-small}
Let $\bar{G}\subset\PGL_3(\Bbbk)$
be a finite subgroup of order
$|\bar{G}|>360$.
Then~\mbox{$\bar{J}(\bar{G})\le 12$}.
\end{lemma}
\begin{proof}
Let $G\subset\SL_3(\Bbbk)$ be a preimage
of $\bar{G}$ with respect to the natural
projection
\begin{equation*}
\SL_3(\Bbbk)\to\PGL_3(\Bbbk).
\end{equation*}
Then one has $|\bar{G}|=|G|/3$, and
$\bar{J}(\bar{G})\le\bar{J}(G)$.

Let us use the notation introduced in the proof
of Lemma~\xref{lemma:weak-GL3}.
If $G$ is a group of type~(i) or~(ii), then
$\bar{J}(G)\le 12$. If $G$ is a group of type~(iii) or~(iv), then
\begin{equation*}
|G|\le |3.\A_6|=1080,
\end{equation*}
and $|\bar{G}|\le 360$.
Finally, if $G$ is a group of type~(v), then
\begin{equation*}
|G|\le |\mathcal{H}_3\rtimes\SL_2(\F_3)|=648,
\end{equation*}
and $|\bar{G}|\le 216$.
\end{proof}

\begin{lemma}\label{lemma:three-generators}
Let $B$ be a (non-trivial) finite abelian subgroup of $\PGL_3(\Bbbk)$. Then
$B$ is generated by at most three elements.
\end{lemma}
\begin{proof}
Recall that a finite abelian subgroup of $\GL_n(\Bbbk)$ is generated by at most $n$ elements.
Let $\tilde{B}\subset\SL_3(\Bbbk)$ be the preimage of $B$ with respect to the natural projection~\mbox{$\SL_3(\Bbbk)\to\PGL_3(\Bbbk)$}.
Let $\tilde A\subset \tilde B$ be a maximal abelian subgroup and let $A\subset B$
be its image. Then $A$ has an isolated fixed point on $\P^2$, and
the number of its isolated fixed points is at most~$3$.
Therefore,
the group $B$ has an orbit of length at most $3$ on $\P^2$. Let~$P$ be a point of
such orbit, and let $B'\subset B$ be the stabilizer of $P$.
By Lemma~\xref{lemma:Aut-P} there is a faithful representation of the group
$B'$ in the Zariski tangent space $T_P(\P^2)\cong\Bbbk^2$, so that~$B'$ is generated by at most
two elements.
The group $B$ is generated by its subgroup~$B'$ and an arbitrary element from
$B\setminus B'$, if any.
\end{proof}

The following fact is a refinement of~\cite[Lemma~2.8]{Prokhorov-Shramov-J}
(cf.~\cite[Remark~2.4]{Prokhorov-Shramov-J}).

\begin{lemma}\label{lemma:4-PGL3}
Let $G$ be a group that fits into an exact sequence
\begin{equation*}
1\to \Gamma\longrightarrow G\stackrel{\phi}\longrightarrow\PGL_3(\Bbbk),
\end{equation*}
where $\Gamma\cong\mumu_2^m$ with $m\le 2$.
Then $G$ is Jordan with
\begin{equation*}
\bar{J}(G)\le 2304.
\end{equation*}
\end{lemma}
\begin{proof}
We may assume that $G$ is finite.
If the order of the group $\phi(G)\subset\PGL_3(\Bbbk)$
is at most $360$, then one has
\begin{equation*}
\bar{J}(G)\le [G:\Gamma]=|\phi(G)|\le 360.
\end{equation*}
Therefore, we may assume that $|\phi(G)|>360$.
By Lemma~\xref{lemma:PGL3-small}
we can find an abelian subgroup $B$ in $\phi(G)$ of index
$[\phi(G):B]\le 12$. Put $\tilde{G}=\phi^{-1}(B)$.
Then
\begin{equation*}
[G:\tilde{G}]=[\phi(G):B]\le 12,
\end{equation*}
so that by Remark~\xref{remark:Pyber} we are left
with the task to bound $\bar{J}(\tilde{G})$.

We have an exact sequence of groups
\begin{equation*}
1\to \Gamma\to \tilde{G}\to B\to 1.
\end{equation*}
For an element $g\in\tilde{G}$ denote by $Z(g)$ the centralizer
of $g$ in $\tilde{G}$. Since $B$ is an abelian quotient of $\tilde{G}$,
we see that the commutator subgroup of $\tilde{G}$ has order at most $|\Gamma|$,
so that for any $g\in\tilde{G}$ one has $[G:Z(g)]\le |\Gamma|$.

Since $B$ is an abelian subgroup of $\PGL_3(\Bbbk)$, it is generated
by at most three elements by Lemma~\xref{lemma:three-generators}.
Choose three generators of $B$, and let $g_1$, $g_2$ and $g_3$ be elements
of $\tilde{G}$ that project to these three generators.
Put
\begin{equation*}
I=Z(g_1)\cap Z(g_2)\cap Z(g_3).
\end{equation*}
Then the index
\begin{equation*}
[\tilde{G}:I]\le |\Gamma|^3\le 64.
\end{equation*}

Let $C$ be the centralizer of $\Gamma$ in $\tilde{G}$.
Since $\Gamma$ is a normal subgroup of $\tilde{G}$, we see that $C$
is a normal subgroup of $\tilde{G}$ as well.
Moreover, since $\Gamma\subset C$, we have an inclusion $\tilde{G}/C\subset B$,
so that $\tilde{G}/C$ is an abelian group generated by
three elements.
Also, one has an inclusion
\begin{equation*}
\tilde{G}/C\subset\Aut(\Gamma)\subset\GL_2(\F_2)\cong\SS_3.
\end{equation*}
Therefore, we conclude that $|\tilde{G}/C|\le 3$.

Let $Z$ be the center of $\tilde{G}$. Then $Z$ contains the intersection $C\cap I$, so that
\begin{equation*}
\bar{J}(\tilde{G})\le J(\tilde{G})\le [\tilde{G}:Z]\le [\tilde{G}:C\cap I]\le [\tilde{G}:C]\cdot [\tilde{G}:I]\le 3\cdot 64=192,
\end{equation*}
and thus
\begin{equation*}
\bar{J}(G)\le [G:\tilde{G}]\cdot\bar{J}(\tilde{G})\le 2304.\qedhere
\end{equation*}
\end{proof}

\subsection{Dimension $4$}
\label{subsection:GL4}

\begin{lemma}\label{lemma:weak-GL4}
One has
\begin{equation*}
\bar{J}\big(\PGL_4(\Bbbk)\big)=\bar{J}\big(\GL_4(\Bbbk)\big)= 960.
\end{equation*}
\end{lemma}
\begin{proof}
Let $V$ be a four-dimensional vector
space over $\Bbbk$, and let $G\subset\GL(V)$
be a finite subgroup. It is enough to study the weak Jordan constant
$\bar{J}(G)$. Moreover, for this we may assume that
$G\subset\SL(V)\cong\SL_4(\Bbbk)$.
Then there are the following possibilities
for the group $G$ (see~\cite[Chapter~VII]{Blichfeldt}
or~\cite[\S8.5]{Feit}):
\begin{enumerate}
\item the $G$-representation $V$ is reducible;
\item there is a transitive homomorphism $h\colon G\to\SS_{k}$
such that
$V$ splits into a sum of $k$ representations of
the subgroup $H=\Ker(h)$ of dimension $4/k$ for
some $k\in\{2,4\}$;
\item
the group $G$ contains a subgroup
$H$ of index at most $2$, such that $H$ is a quotient
by a certain central subgroup of a group $\Gamma_1\times\Gamma_2$,
where $\Gamma_1$ and $\Gamma_2$ are finite subgroups
of $\GL_2(\Bbbk)$;
\item
the group $G$ is generated by some subgroup
of scalar matrices in $\SL_4(\Bbbk)$ and a group $\hat{G}$ that is
one of the
groups $\A_5$, $\SS_5$, $2.\A_5$, $2.\SS_5$, or $\SL_2(\F_7)$;
\item
the group $G$ is generated by some subgroup
of scalar matrices in $\SL_4(\Bbbk)$ and a group $\hat{G}$ that is
one of the
groups $2.\A_6$, $2.\SS_6$, $2.\A_7$,
or $\Sp_4(\F_3)$;
\item
the group $G$ contains an extra-special group~$\mathcal{H}_4$
of order $32$ and
is contained in the normalizer
of $\mathcal{H}_4$ in $\SL(V)$.
\end{enumerate}

In case~(i) there is an embedding
$G\hookrightarrow \Gamma_1\times\Gamma_2$,
where $\Gamma_i$ is a finite subgroup
of $\GL_{n_i}(\Bbbk)$ for $i\in\{1,2\}$, and
$n_1\le n_2$ are positive integers such that $n_1+n_2=4$.
One has
\begin{equation*}
\bar{J}(G)\le\bar{J}(\Gamma_1\times\Gamma_2)\le
\bar{J}\big(\GL_{n_1}(\Bbbk)\big)\cdot
\bar{J}\big(\GL_{n_1}(\Bbbk)\big).
\end{equation*}
If
$(n_1,n_2)=(1,3)$, this gives
$\bar{J}(G)\le 72$
by Lemma~\xref{lemma:weak-GL3}.
If
$(n_1,n_2)=(2,2)$, this gives
\begin{equation*}
\bar{J}(G)\le 12\cdot 12=144
\end{equation*}
by Corollary~\xref{corollary:GL2}.

In case~(ii) the group $H$ is a subgroup
of $G$ of index
\begin{equation*}
[G:H]\le |\SS_k|=k!
\end{equation*}
Moreover, there is an embedding
$H\hookrightarrow \Gamma_1\times\ldots\times\Gamma_{k}$,
where $\Gamma_i$ are finite subgroups
of~\mbox{$\GL_{4/k}(\Bbbk)$}.
Thus
\begin{equation*}
\bar{J}(G)\le
[G:H]\cdot\bar{J}(H)\le
k!\cdot\bar{J}(\Gamma_1)\cdot\ldots\cdot\bar{J}(\Gamma_k)\le
k!\cdot\bar{J}\big(\GL_{4/k}(\Bbbk)\big)^k.
\end{equation*}
If $k=2$, this gives
$\bar{J}(G)\le 288$
by Corollary~\xref{corollary:GL2}.
If $k=4$, this gives
$\bar{J}(G)\le 24$.

In case~(iii) we obtain the bound
$\bar{J}(G)\le 288$ in a similar way.

In case (iv) one has
\begin{equation*}
\bar{J}(G)=\bar{J}\big(\hat{G}\big)\le |\hat{G}|\le 336.
\end{equation*}

In case (v) one has
\begin{equation*}
\bar{J}(G)=\bar{J}\big(\hat{G}\big)\le\bar{J}\big(\Sp_4(\F_3)\big)=960.
\end{equation*}

In case (vi) one has $\bar{J}(G)\le\bar{J}(N)$,
where $N$ is the normalizer of $\mathcal{H}_4$ in $\SL(V)$.
The group $N$ fits into the exact sequence
\begin{equation*}
1\to\tilde{\mathcal{H}}_4\to N\to \SS_6\to 1,
\end{equation*}
where $\tilde{\mathcal{H}}_4$ is a group generated by
$\mathcal{H}_4$ and a scalar matrix
\begin{equation*}
\sqrt{-1}\cdot\mathrm{Id}\in\SL(V).
\end{equation*}
Recall that
\begin{equation*}
\mathcal{H}_4\cong Q_8\times Q_8/\mumu_2,
\end{equation*}
where $Q_8$ is a quaternion group of order~$8$.
Being viewed as a subgroup of $\SL_2(\Bbbk)$, the group
$Q_8$ is normalized by a binary octahedral group
$2.\SS_4$. Thus the group $N$ contains a subgroup
\begin{equation*}
R\cong 2.\SS_4\times 2.\SS_4/\mumu_2,
\end{equation*}
and also a subgroup $\tilde{R}$ generated by $R$ and $\sqrt{-1}\cdot\mathrm{Id}$.
One has
\begin{equation*}
\bar{J}\big(\tilde{R}\big)=\bar{J}(R)=\bar{J}(2.\SS_4\times 2.\SS_4)=\bar{J}(2.\SS_4)^2=36.
\end{equation*}
On the other hand, we compute
the index $[N:\tilde{R}]=20$. This gives
\begin{equation*}
\bar{J}(N)\le [N:\tilde{R}]\cdot\bar{J}\big(\tilde{R}\big)=20\cdot 36=720.
\end{equation*}

Therefore, we see that
$\bar{J}(G)\le 960$, and thus $\bar{J}\big(\GL_4(\Bbbk)\big)\le 960$.
The inequality
\begin{equation*}
\bar{J}\big(\PGL_4(\Bbbk)\big)\le\bar{J}\big(\GL_4(\Bbbk)\big)
\end{equation*}
holds by Remark~\xref{remark:surjection}.
The value $\bar{J}(\PGL_4(\Bbbk))=960$ is given by the
group~\mbox{$\PSp_4(\F_3)\subset\PGL_4(\Bbbk)$}
whose abelian subgroup of maximal order is~$\mumu_3^3$
(cf.~\mbox{\cite[Table~2]{Vdovin}}).
The value $\bar{J}(\GL_4(\Bbbk))=960$ is given by the group~\mbox{$\Sp_4(\F_3)\subset\GL_4(\Bbbk)$}
whose abelian subgroup of maximal order is~$\mumu_2\times\mumu_3^3$
that is a preimage of a subgroup~\mbox{$\mumu_3^3\subset\PSp_4(\F_3)$} with respect
to the natural projection~\mbox{$\Sp_4(\F_3)\to\PSp_4(\F_3)$}.
\end{proof}

\begin{remark}
The group $2.\SS_5$ listed in case~(iv) of Lemma~\xref{lemma:weak-GL4}
is omitted in the list
given in~\cite[\S8.5]{Feit}. It is still listed by some other classical
surveys, see e.g.~\cite[\S119]{Blichfeldt}.
\end{remark}

Recall that for a given group $G$ with a representation in a
vector space $V$ a \emph{semi-invariant} of $G$ of degree $n$
is an eigen-vector of $G$ in
$\Sym^n V^{\vee}$.

\begin{lemma}\label{lemma:GL4-quadric}
Let $V$ be a four-dimensional vector
space over $\Bbbk$, and let $G\subset\GL(V)$
be a finite subgroup.
If $G$ has a semi-invariant of degree $2$, then
$\bar{J}(G)\le 288$.
\end{lemma}
\begin{proof}
Let $q$ be a semi-invariant of $G$ of degree $2$.
We consider the possibilities for the rank of the quadratic form $q$
case by case.

Suppose that $V$ has a one-dimensional subrepresentation of $G$.
Then~\mbox{$G\subset\Bbbk^*\times\GL_3(\Bbbk)$}, so that
$\bar{J}(G)\le 72$ by Lemma~\xref{lemma:weak-GL3}.
Therefore we may assume that the rank of $q$ is not equal to $1$ or $3$.

Suppose that the rank of $q$ is $2$, so that $q$ is a product of two linear forms.
Then there is a subgroup $G_1\subset G$ of index at most $2$
such that these linear forms are semi-invariant with respect to~$G_1$.
Hence $V$ splits as a sum of a two-dimensional and two one-dimensional
representations of $G_1$.
This implies that~\mbox{$G_1\subset\Bbbk^*\times\Bbbk^*\times\GL_2(\Bbbk)$}, so that
\begin{equation*}
\bar{J}(G)\le 2\cdot\bar{J}(G_1)\le 2\cdot\bar{J}\big(\GL_2(\Bbbk)\big)=24
\end{equation*}
by Corollary~\xref{corollary:GL2}.

Finally, suppose that the rank of $q$ is $4$, so that
the quadric $Q\subset \P(V)\cong\P^3$ given by the equation $q=0$ is smooth,
i.e. $Q\cong\P^1\times\P^1$.
By Lemma~\xref{lemma:weak-constant-quadric} there is a subgroup~\mbox{$H\subset G$}
of index $[G:H]\le 288$ that acts on $Q$ with a fixed point $P$ and does not
interchange the lines $L_1$ and $L_2$ passing through $P$ on $Q$.
As the representation of $H$, the vector space~$V$ splits as a sum of the one-dimensional
representation corresponding to the point~$P$, two one-dimensional
representations arising from the lines $L_1$ and $L_2$, and
one more one-dimensional representation.
Therefore, $H$ is an abelian group (note that Lemma~\xref{lemma:weak-constant-quadric}
asserts only that the image of $H$ in $\PGL_4(\Bbbk)$ is abelian).
This shows that~\mbox{$\bar{J}(G)\le 288$} and completes the proof of the lemma.
\end{proof}

\subsection{Dimension $5$}
\label{subsection:GL5}

\begin{lemma}\label{lemma:weak-GL5}
One has
\begin{equation*}
\bar{J}\big(\PGL_5(\Bbbk)\big)=\bar{J}\big(\GL_5(\Bbbk)\big)= 960.
\end{equation*}
\end{lemma}
\begin{proof}
Let $V$ be a five-dimensional vector
space over $\Bbbk$, and let $G\subset\GL(V)$
be a finite subgroup.
It is enough to study the weak Jordan constant
$\bar{J}(G)$. Moreover, for this we may assume that
$G\subset\SL(V)\cong\SL_5(\Bbbk)$.
Recall that there are the following possibilities
for the group $G$ (see~\cite{Brauer67} or~\cite[\S8.5]{Feit}):
\begin{enumerate}
\item the $G$-representation $V$ is reducible;
\item there is a transitive homomorphism $h\colon G\to\SS_{5}$
such that
$V$ splits into a sum of five one-dimensional representations of
the subgroup $H=\Ker(h)$;
\item
the group $G$ is generated by some subgroup
of scalar matrices in $\SL_5(\Bbbk)$ and a group
$\hat{G}$ that is
one of the groups
$\A_5$, $\SS_5$, $\A_6$, $\SS_6$, $\PSL_2(\F_{11})$, or $\PSp_4(\F_3)$;
\item one has $G\cong\mathcal{H}_5\rtimes \Sigma$,
where $\mathcal{H}_5$ is the Heisenberg group of order $125$, and
$\Sigma$ is some subgroup of $\SL_2(\F_5)$.
\end{enumerate}

In case~(i) there is an embedding
$G\hookrightarrow \Gamma_1\times\Gamma_2$,
where $\Gamma_i$ is a finite subgroup
of $\GL_{n_i}(\Bbbk)$ for $i\in\{1,2\}$, and
$n_1\le n_2$ are positive integers such that $n_1+n_2=5$.
One has
\begin{equation*}
\bar{J}(G)\le\bar{J}(\Gamma_1\times\Gamma_2)\le
\bar{J}\big(\GL_{n_1}(\Bbbk)\big)\cdot
\bar{J}\big(\GL_{n_1}(\Bbbk)\big).
\end{equation*}
If
$(n_1,n_2)=(1,4)$, this gives
$\bar{J}(G)\le 960$
by Lemma~\xref{lemma:weak-GL4}.
If
$(n_1,n_2)=(2,3)$, this gives
\begin{equation*}
\bar{J}(G)\le 12\cdot 72=864
\end{equation*}
by Corollary~\xref{corollary:GL2} and Lemma~\xref{lemma:weak-GL3}.

In case~(ii) the group $H$ is an abelian subgroup
of $G$, so that
\begin{equation*}
\bar{J}(G)\le [G:H]\le |\SS_5|=120.
\end{equation*}

In case~(iii) it is easy to check that
$\bar{J}(G)=\bar{J}(\hat{G})\le 960$,
cf. the proof of Lemma~\xref{lemma:weak-GL4}.

In case~(iv) one has
\begin{equation*}
\bar{J}(G)\le
\bar{J}\big(\mathcal{H}_5\rtimes\SL_2(\F_5)\big)
=120.
\end{equation*}

Therefore, we see that
$\bar{J}(G)\le 960$, and thus $\bar{J}\big(\GL_5(\Bbbk)\big)\le 960$.
The inequality
\begin{equation*}
\bar{J}\big(\PGL_5(\Bbbk)\big)\le\bar{J}\big(\GL_5(\Bbbk)\big)
\end{equation*}
holds by Remark~\xref{remark:surjection}.
The value $\bar{J}(\PGL_5(\Bbbk))=960$ is given by the group~\mbox{$\PSp_4(\F_3)\subset\PGL_5(\Bbbk)$},
cf. the proof of Lemma~\xref{lemma:weak-GL4}.
Similarly,
the value~\mbox{$\bar{J}(\GL_5(\Bbbk))=960$} is given by the group
$\PSp_4(\F_3)\subset\GL_5(\Bbbk)$.
\end{proof}

We summarize the main results of \S\S\ref{subsection:GL2}--\ref{subsection:GL5}
in Table~\ref{table:constants}. In the first column we list the dimensions we will need in the sequel.
In the second column we give the values of the weak Jordan constants~\mbox{$\bar{J}(\PGL_n(\Bbbk))$},
and in the third column we give the groups that attain these constants.
Similarly, in the fourth column we give the values of the weak Jordan constants~\mbox{$\bar{J}(\GL_n(\Bbbk))$},
and in the fifth column we give the groups that attain the constants.
In the sixth column we list the actual values of the usual Jordan constants~\mbox{$\bar{J}(\GL_n(\Bbbk))$}
which can be found in~\cite[Proposition~C]{Collins2007}.

\begin{table}[h]
\centering
\begin{tabularx}{\textwidth}{Y|Y|Y|Y|Y|Y}
\Xhline{1\arrayrulewidth}
$n$ & $\bar{J}(\PGL_n(\Bbbk))$ & group & $\bar{J}(\GL_n(\Bbbk))$ & group & $J(\GL_n(\Bbbk))$
\\\Xhline{2\arrayrulewidth}
$2$ & $12$ & $\A_5$ & $12$ & $2.\A_5$ & $60$ \\

$3$ & $40$ & $\A_6$ & $72$ & $3.\A_6$ & $360$ \\

$4$ & $960$ & $\PSp_4(\F_3)$ & $960$ & $\Sp_4(\F_3)$ & $25920$ \\

$5$ & $960$ & $\PSp_4(\F_3)$ & $960$ & $\PSp_4(\F_3)$ & $25920$ \\
\Xhline{1\arrayrulewidth}
\end{tabularx}
\vspace{7pt}
\caption{Jordan constants for linear groups}\label{table:constants}
\end{table}

\subsection{Dimension $7$}

We start with a general observation concerning finite groups
with relatively large abelian subgroups.

\begin{lemma}\label{lemma:isotypical}
Let $G$ be a group,
and $\tilde{\Gamma}\subset G$
be a normal finite abelian subgroup.
Suppose that $\tilde{\Gamma}$ cannot be generated by less than $m$ elements.
Let $V$ be an $N$-dimensional vector space over $\Bbbk$.
Suppose that $V$ is a faithful representation of $G$.
Then there exist positive integers~\mbox{$t$, $m_1,\ldots, m_t$, $d_1,\ldots,d_t$}
such that
\begin{itemize}
\item
$m_1d_1+\ldots+m_td_t=N$;
\item
$m_1+\ldots+m_t\ge m$;
\item
the group $G$ is Jordan with
\begin{equation*}
\bar{J}(G)\le \Big(\prod\limits_{i=1}^t m_i!\Big)
\cdot\Big(\prod\limits_{i=1}^t\bar{J}\big(\GL_{d_i}(\Bbbk)\big)^{m_i}\Big).
\end{equation*}
\end{itemize}
\end{lemma}
\begin{proof}
Let
\begin{equation}\label{eq:splitting}
V=V_1\oplus\ldots\oplus V_s
\end{equation}
be the splitting of $V$ into isotypical
components with respect to $\tilde{\Gamma}$.
Since $V$ is
a faithful representation of $\tilde{\Gamma}$, and $\tilde{\Gamma}$ is an abelian group,
we have an injective homomorphism~\mbox{$\tilde{\Gamma} \hookrightarrow (\Bbbk^*)^s$}.
By assumption one has $s\ge m$.
Suppose that the splitting~\eqref{eq:splitting} contains~$m_1$ summands of dimension $d_1$,
$m_2$ summands of dimension $d_2$, \ldots, and $m_t$ summands of dimension $d_t$.
Then one has $m_1d_1+\ldots+m_td_t=N$. Moreover, the total number of summands
in~\eqref{eq:splitting} equals $m_1+\ldots+m_t=s\ge m$.

Since $\tilde{\Gamma}\subset G$ is a normal subgroup, the group $G$
interchanges the summands in~\eqref{eq:splitting}.
Moreover, $G$ can interchange only those subspaces
$V_i$ and $V_j$ that have the same dimension.
Therefore, we get a homomorphism
\begin{equation*}
\psi\colon G\to\prod\limits_{i=1}^t \SS_{m_i}.
\end{equation*}

Let $\Delta\subset G$ be the kernel of the homomorphism
$\psi$. Then each summand of~\eqref{eq:splitting}
is invariant with respect to $\Delta$. Since $V$ is a
faithful representation of $\Delta$, one has an inclusion
\begin{equation*}
\Delta\hookrightarrow\prod_{j=1}^s \GL(V_j)\cong
\prod_{i=1}^t \big(\GL_{d_i}(\Bbbk)\big)^{m_i}.
\end{equation*}
Note that
\begin{equation*}
[G:\Delta]\le |\prod\limits_{i=1}^t \SS_{m_i}|=\prod\limits_{i=1}^t m_i!.
\end{equation*}
Recall that the groups $\GL_{d_i}(\Bbbk)$ are Jordan by Theorem~\xref{theorem:Jordan}.
Thus the group $G$ is Jordan with
\begin{equation*}
\bar{J}(G)\le [G:\Delta]\cdot\bar{J}(\Delta)\le \Big(\prod\limits_{i=1}^t m_i!\Big)\cdot
\Big(\prod\limits_{i=1}^t\bar{J}\big(\GL_{d_i}(\Bbbk)\big)^{m_i}\Big)
\end{equation*}
by Remark~\xref{remark:Pyber}.
\end{proof}

Lemma~\xref{lemma:isotypical} allows us to provide a bound for Jordan constants
of some subgroups of~\mbox{$\GL_7(\Bbbk)$}.
This bound will be used in the proof of Lemma~\xref{lemma:intersection-of-three-quadrics}.

\begin{lemma}\label{lemma:isotypical-7}
Let $G$ be a group,
and $\tilde{\Gamma}\subset G$ be a normal finite abelian subgroup such that
$\tilde{\Gamma}\cong\mumu_2^m$ with $m\ge 4$.
Suppose that $G$ has a faithful seven-dimensional representation.
Then $G$ is Jordan with
\begin{equation*}
\bar{J}(G)\le 10368.
\end{equation*}
\end{lemma}
\begin{proof}
Since $\tilde{\Gamma}\cong\mumu_2^m$ has a faithful seven-dimensional representation,
we have $m\le 7$.
By Lemma~\xref{lemma:isotypical} there
exist positive integers $t$, $m_1,\ldots, m_t$, $d_1,\ldots,d_t$,
such that
\begin{equation*}
m_1d_1+\ldots+m_td_t=7,
\end{equation*}
while $m_1+\ldots+m_t\ge m$ and
\begin{equation}\label{eq:J-isotypical-7}
\bar{J}(G)\le \Big(\prod\limits_{i=1}^t m_i!\Big)
\cdot\Big(\prod\limits_{i=1}^t\bar{J}\big(\GL_{d_i}(\Bbbk)\big)^{m_i}\Big).
\end{equation}
In particular, one has $4\le m_1+\ldots+m_t\le 7$. Also, we may assume
that $d_1<\ldots<d_t$. We consider several possibilities for $m_1+\ldots+m_t$ case by case.

If $m_1+\ldots+m_t=7$, then $t=1$, $d_1=1$ and $m_1=7$,
so that~\eqref{eq:J-isotypical-7}
gives
\begin{equation*}
\bar{J}(G)\le 7!=5040.
\end{equation*}

If $m_1+\ldots+m_t=6$, then $t=2$, $d_1=1$, $m_1=5$, $d_2=2$, $m_2=1$,
so that~\eqref{eq:J-isotypical-7}
gives
\begin{equation*}
\bar{J}(G)\le 5!\cdot \bar{J}\big(\GL_2(\Bbbk)\big)=120\cdot 12=1440
\end{equation*}
by Corollary~\xref{corollary:GL2}.

If $m_1+\ldots+m_t=5$, then $t=2$, $d_1=1$, and either $m_1=4$, $d_2=3$, $m_2=1$,
or $m_1=3$, $d_2=2$, $m_2=2$.
In the former case~\eqref{eq:J-isotypical-7}
gives
\begin{equation*}
\bar{J}(G)\le 4!\cdot \bar{J}\big(\GL_3(\Bbbk)\big)=24\cdot 72=1728
\end{equation*}
by Lemma~\xref{lemma:weak-GL3}.
In the latter case~\eqref{eq:J-isotypical-7}
gives
\begin{equation*}
\bar{J}(G)\le 3!\cdot 2!\cdot\bar{J}\big(\GL_2(\Bbbk)\big)^2=6\cdot 2\cdot 12^2=1728
\end{equation*}
by Corollary~\xref{corollary:GL2}.

Finally, if $m_1+\ldots+m_t=4$, then either
\begin{equation*}
t=2, \ d_1=1, \ m_1=3, \ d_2=4, \ m_2=1,
\end{equation*}
or
\begin{equation*}
t=2, \ d_1=1, \ m_1=1, \ d_2=2, \ m_2=3,
\end{equation*}
or
\begin{equation*}
t=3, \ d_1=1, \ m_1=2, \ d_2=2, \ m_2=1, \ d_3=3, \ m_3=1.
\end{equation*}
In the first case~\eqref{eq:J-isotypical-7}
gives
\begin{equation*}
\bar{J}(G)\le 3!\cdot \bar{J}\big(\GL_4(\Bbbk)\big)=6\cdot 960=5760
\end{equation*}
by Lemma~\xref{lemma:weak-GL4}.
In the second case~\eqref{eq:J-isotypical-7}
gives
\begin{equation*}
\bar{J}(G)\le 3!\cdot \bar{J}\big(\GL_2(\Bbbk)\big)^3=6\cdot 12^3=10368
\end{equation*}
by Corollary~\xref{corollary:GL2}.
In the third case~\eqref{eq:J-isotypical-7}
gives
\begin{equation*}
\bar{J}(G)\le 2!\cdot \bar{J}\big(\GL_2(\Bbbk)\big)\cdot \bar{J}\big(\GL_3(\Bbbk)\big)=
2\cdot 12\cdot 72=1728
\end{equation*}
by Corollary~\xref{corollary:GL2} and Lemma~\xref{lemma:weak-GL3}.
\end{proof}

\section{Surfaces}
\label{section:dim-2}

The goal of this section is to estimate weak Jordan constants for automorphism
groups of rational surfaces, as well as some other constants of similar nature.
In the sequel for any variety $X$ we will denote by $\Phi(X)$ the
minimal positive integer $m$ such that for any finite group
$G\subset\Aut(X)$ there is a subgroup $F\subset G$ with $[G:F]\le m$ acting
on $X$ with a fixed point. If there does not exist an integer $m$
with the above property, we put $\Phi(X)=+\infty$.
Note that $\Phi(X)$ is bounded by some universal constant for rationally
connected varieties~$X$ of dimension at most $3$ by~\cite[Theorem~4.2]{ProkhorovShramov-RC}.

\subsection{Preliminaries}
\label{subsection:surf-prelim}

We start with the one-dimensional case.

\begin{lemma}
\label{lemma:dim-1}
One has $\Phi(\P^1)=12$.
Moreover, if $T$ is a finite union of rational curves
such that its dual graph $T^{\vee}$ is a tree,
then $\Phi(T)\le 12$.
\end{lemma}
\begin{proof}
The inequality $\Phi(\P^1)\le 12$ is given by
Lemma~\xref{lemma:weak-constant-quadric}\xref{lemma:weak-constant-quadric-i}.
The equality~\mbox{$\Phi(\P^1)=12$} is given by the icosahedral group~\mbox{$\A_5\subset\Aut(\P^1)$}.
Since for any rational curve $C$ one has
\begin{equation*}
\Aut(C)\subset\Bir(C)\cong\Bir(\P^1)=\Aut(\P^1),
\end{equation*}
we also see that $\Phi(C)\le 12$.

Let $T$ be a finite union of rational curves
such that its dual graph $T^{\vee}$ is a tree. Then
there is a natural homomorphism
of $\Aut(T)$ to the finite group $\Aut(T^{\vee})$.
It is easy to show by induction on the number of vertices
that either there is an
edge of $T^{\vee}$ that is invariant under $\Aut(T^{\vee})$,
or there is a vertex of
$T^{\vee}$ that is invariant under $\Aut(T^{\vee})$.
In the former case there is a point $P\in T$ fixed by
$\Aut(T)$, so that $\Phi(T)=1$.
In the latter case there is a
rational curve $C\subset T$
that is invariant under $\Aut(T)$, so that
\begin{equation*}
\Phi(T)\le\Phi(C)\le 12.\qedhere
\end{equation*}
\end{proof}

Now we proceed with the two-dimensional case.
In a sense, we are going to do in a more systematic way
the same things that were done in Lemma~\xref{lemma:weak-constant-quadric}.
For a variety $X$ with an action of a finite group
$G$, we will denote by $\Phi_a(X,G)$ the
minimal positive integer $m$ such that there is an \emph{abelian}
subgroup $A\subset G$ with $[G:A]\le m$ acting
on $X$ with a fixed point.
The main advantage of this definition is the following property.

\begin{lemma}\label{lemma:Phi-a-lift}
Let $X$ and $Y$ be smooth surfaces acted on by a finite group
$G$. Suppose that there
is a $G$-equivariant birational morphism $\pi\colon Y\to X$.
Then $\Phi_a(Y,G)=\Phi_a(X,G)$.
\end{lemma}
\begin{proof}
The assertion is implied by the results of~\cite{Kollar-Szabo-2000}
in arbitrary dimension. We give the proof for dimension $2$ for
the readers convenience.

The inequality $\Phi_a(Y,G)\ge\Phi_a(X,G)$ is obvious.
To prove the opposite inequality
choose an abelian subgroup $A\subset G$ such that
there is a point $P\in X$ fixed by $A$.
We are going to produce a point $Q\in Y$ fixed by $A$ such that
$\pi(Q)=P$.

The birational morphism $\pi$ is a composition
of blow ups of smooth points.
Since $\pi$ is $G$-equivariant and thus $A$-equivariant,
we may replace $X$ by a neighborhood of the point~$P$
and thus suppose that $\pi$ is a sequence of blow ups
of points lying over the point~$P$. If~$\pi$ is an isomorphism,
then there is nothing to prove. Otherwise, by induction
in the number of blow ups, we see that it is enough to
consider the case when $\pi$ is a single blow up
of the point~$P$. In this case the exceptional
divisor $E=\pi^{-1}(P)$ is identified with
the projectivization of the Zariski tangent space
$T_P(X)$, and the action of $A$ on $E$ comes from a linear action of $A$ on~\mbox{$T_P(X)$}.
Since the group~$A$ is abelian, it has a one-dimensional
invariant subspace in $T_P(X)$, which gives an $A$-invariant
point~\mbox{$Q\in E\subset Y$}.
\end{proof}

\subsection{Del Pezzo surfaces}

\begin{lemma}\label{lemma:Phi-a-P2}
Let $G\subset\Aut(\P^2)$ be a finite group.
Then one has $\Phi_a(\P^2, G)\le 72$.
\end{lemma}
\begin{proof}
One has $\Aut(X)\cong\PGL_3(\Bbbk)$.
By the holomorphic Lefschetz fixed-point formula any cyclic group
acting on a rational variety has a fixed point.
Now the required bound is obtained from the classification of finite subgroups
of $\GL_3(\Bbbk)$
(see~\cite[Chapter~XII]{MBD1916}
or~\cite[\S8.5]{Feit},
and also the proof of Lemma~\xref{lemma:weak-GL3}).
\end{proof}

\begin{remark}
Note that the bound given by Lemma~\xref{lemma:Phi-a-P2} is actually attained for the
group~\mbox{$\A_6\subset\PGL_3(\Bbbk)$} whose abelian subgroup of maximal order acting on~$\P^2$
with a fixed point is~$\mumu_5$.
\end{remark}

\begin{lemma}\label{lemma:Phi-a-DP}
Let $X$ be a smooth del Pezzo surface.
Let $G\subset\Aut(X)$ be a finite group.
Then one has
\begin{equation*}
\Phi_a(X, G)\le 288.
\end{equation*}
Moreover, if $X$ is not isomorphic to $\P^1\times\P^1$,
then~\mbox{$\Phi_a(X, G)\le 144$}.
\end{lemma}
\begin{proof}
If $X\cong\P^2$, then $\Phi_a(X, G)\le 72$ by Lemma~\xref{lemma:Phi-a-P2}.

Suppose that $X\cong\P^1\times\P^1$. Then one has
$\Phi_a(X,G)\le 288$ by Lemma~\xref{lemma:weak-constant-quadric}(ii). Note that this value is attained
for the group
\begin{equation*}
G\cong\big(\A_5\times\A_5\big)\rtimes\mumu_2\subset\Aut(\P^1\times\P^1).
\end{equation*}

Suppose that $X$ is a blow up $\pi\colon X\to\P^2$ at one or two
points. Then $\pi$ is an $\Aut(X)$-equivariant birational
morphism, so that $\Phi_a(X,G)\le 72$ by Lemmas~\xref{lemma:Phi-a-P2}
and~\xref{lemma:Phi-a-lift}.

Put $d=K_X^2$. We may assume that $d\le 6$.

Suppose that $d=6$. Then
\begin{equation*}
\Aut(X)\cong\big(\Bbbk^*\times\Bbbk^*)\rtimes\mathrm{D}_{6},
\end{equation*}
where $\mathrm{D}_{6}$ is the dihedral group of order $12$
(see~\cite[Theorem~8.4.2]{Dolgachev-book}). The subgroup~\mbox{$\Bbbk^*\times\Bbbk^*\subset\Aut(X)$}
acts on $X$ with a fixed point by Borel's theorem (see e.\,g.~\mbox{\cite[VIII.21]{Humphreys1975}}).
From this one can easily deduce that~\mbox{$\Phi_a(X, G)\le 12$}
for any finite subgroup~\mbox{$G\subset\Aut(X)$}.

If $d\le 5$, then the group $\Aut(X)$ is finite,
and it is enough to show that~\mbox{$\Phi_a\big(X,\Aut(X)\big)\le 144$}.

Suppose that $d=5$. Then $\Aut(X)\cong\SS_5$
(see~\cite[Theorem~8.5.6]{Dolgachev-book}).
Hence for any subgroup $G\subset\Aut(X)$ one has
\begin{equation*}
\Phi_a\big(X,\Aut(X)\big)\le |\Aut(X)|=120.
\end{equation*}

Suppose that $d=4$.
Then
\begin{equation*}
\Aut(X)\cong\mumu_2^4\rtimes\Gamma,
\end{equation*}
where $|\Gamma|\le 10$
(see~\cite[Theorem~8.6.6]{Dolgachev-book}).
Representing $X$ as an intersection of two quadrics with
equations in diagonal form, one can see that there
is a subgroup~\mbox{$\mumu_2^2\subset\Aut(X)$}
acting on $X$ with a fixed point.
Therefore, one has
\begin{equation*}
\Phi_a\big(X,\Aut(X)\big)\le \frac{|\Aut(X)|}{|\mumu_2^2|}\le
\frac{160}{4}=40.
\end{equation*}

Suppose that $d=3$. Then either
$\Aut(X)\cong \mumu_3^3\rtimes\SS_4$
and $X$ is the Fermat cubic, or~\mbox{$|\Aut(X)|\le 120$}
(see~\cite[Theorem~9.5.6]{Dolgachev-book}).
In the former case it is easy to see that there is a subgroup
$\mumu_3^2\subset\Aut(X)$ acting on $X$ with a fixed point,
so that
\begin{equation*}
\Phi_a\big(X,\Aut(X)\big)\le \frac{|\Aut(X)|}{|\mumu_3^2|}=
\frac{648}{9}=72.
\end{equation*}
In the latter case one has
\begin{equation*}
\Phi_a\big(X,\Aut(X)\big)\le |\Aut(X)|\le 120.
\end{equation*}

Suppose that $d=2$. Then either
$|\Aut(X)|\le 96$,
or $\Aut(X)\cong\mumu_2\times (\mumu_4^2\rtimes\SS_3)$,
or~\mbox{$\Aut(X)\cong\mumu_2\times\PSL_2(\F_7)$}
(see~\cite[Table~8.9]{Dolgachev-book}).
In the latter case one has
\begin{equation*}
\Phi_a\big(X,\Aut(X)\big)\le |\Aut(X)|\le 120.
\end{equation*}
To estimate $\Phi_a\big(X,\Aut(X)\big)$ in the former two cases,
recall that the anticanonical
linear system $|-K_X|$ defines a double cover
\begin{equation*}
\varphi_{|-K_X|}\colon X\to\P^2
\end{equation*}
branched over a smooth quartic curve $C\subset\P^2$.
The subgroup $\mumu_2$ acts by the Galois
involution of the corresponding double cover.
In particular, the curve $\varphi_{|-K_X|}^{-1}(C)$ consists
of $\mumu_2$-fixed points. If $\Aut(X)\cong\mumu_2\times (\mumu_4^2\rtimes\SS_3)$,
this gives
\begin{equation*}
\Phi_a\big(X,\Aut(X)\big)\le\frac{|\Aut(X)|}{|\mumu_2|}=\frac{192}{2}=96.
\end{equation*}
If $\Aut(X)\cong\mumu_2\times\PSL_2(\F_7)$, then the group
$\PSL_2(\F_7)\subset\Aut(X)$ contains a subgroup~$\mumu_7$,
and $\mumu_7$ acts on the curve $\varphi_{|-K_X|}^{-1}(C)\cong C$ with a fixed point
(this can be easily seen, for example, from the Riemann--Hurwitz formula since $C$ is a smooth curve of genus $3$).
Thus
\begin{equation*}
\Phi_a\big(X,\Aut(X)\big)\le\frac{|\Aut(X)|}{|\mumu_2\times\mumu_7|}=\frac{336}{14}=24.
\end{equation*}

Finally, suppose that $d=1$. Then
\begin{equation*}
\Phi_a\big(X, \Aut(X)\big)\le |\Aut(X)|\le 144
\end{equation*}
(see~\cite[Table~8.14]{Dolgachev-book}).
\end{proof}

\begin{remark}\label{remark:pohuj}
In several cases
(say, for a del Pezzo surface of degree $d=5$)
one can produce better upper bounds for~\mbox{$\Phi_a(X, G)$}
than those given in the proof of
Lemma~\xref{lemma:Phi-a-DP}, but we do not pursue
this goal.
\end{remark}

Lemma~\xref{lemma:Phi-a-DP} immediately implies the following.

\begin{corollary}[{cf. Lemma~\xref{lemma:weak-constant-quadric}(iii)}]
\label{corollary:barJ-for-DP}
Let $X$ be a smooth del Pezzo surface.
Then one has $\bar{J}\big(\Aut(X)\big)\le 288$.
Moreover, if $X$ is not isomorphic to $\P^1\times\P^1$,
then~\mbox{$\bar{J}\big(\Aut(X)\big)\le 144$}.
\end{corollary}

\subsection{Rational surfaces}

Now we pass to the case of arbitrary rational surfaces.

\begin{lemma}\label{lemma:dim-2-constants}
Let $X$ be a smooth rational surface, and $G\subset\Aut(X)$ be a finite subgroup.
Then there exists an abelian subgroup $H\subset G$
of index $[G:H]\le 288$ that acts on $X$ with a fixed point.
\end{lemma}
\begin{proof}
Let $Y$ be a smooth projective rational
surface, and $G\subset\Aut(Y)$
be a finite group. Let
$\pi\colon Y\to X$ be a result of a $G$-Minimal Model Program
ran on $Y$. One has
\begin{equation*}
\Phi_a(Y,G)=\Phi_a(X,G)
\end{equation*}
by Lemma~\xref{lemma:Phi-a-lift}. Moreover, $X$ is either a del Pezzo surface,
or there is a
$G$-equivariant conic bundle structure on~$X$
(see~\cite[Theorem~1G]{Iskovskikh-1979s-e}).
If $X$ is a del Pezzo surface, then~\mbox{$\Phi_a(X,G)\le 288$} by Lemma~\xref{lemma:Phi-a-DP},
so that $\Phi_a(Y,G)\le 288$.

Therefore, we assume that there is a $G$-equivariant conic bundle structure
\begin{equation*}
\phi\colon X\to B\cong \P^1.
\end{equation*}
There is an exact sequence of groups
\begin{equation*}
1\to G_{\phi}\longrightarrow G\stackrel{u}\longrightarrow
G_{B}\to 1,
\end{equation*}
where $G_{\phi}$ acts by fiberwise automorphisms with respect to
$\phi$, and $G_{B}\subset\Aut(\P^1)$.
By Lemma~\xref{lemma:dim-1} we find a subgroup
$G_{B}'\subset G_{B}$ of index $[G_{B}:G_{B}']\le 12$
acting on $\P^1$ with a fixed point $P\in\P^1$.
The group
\begin{equation*}
G'=u^{-1}(G_{B}')\subset G
\end{equation*}
acts by automorphisms of the fiber $C=\phi^{-1}(P)$.
Note that $C$ is a reduced conic, i.e.
it is either isomorphic to $\P^1$, or is a union
of two copies of $\P^1$ meeting at one point.

Suppose that $C\cong \P^1$. Then there is a point $Q\in C$ that is invariant
with respect to some subgroup~\mbox{$G''\subset G'$} of index
$[G':G'']\le 12$ by Lemma~\xref{lemma:dim-1}.
The morphism $\phi\colon X\to B$ is smooth at $Q$.
Hence the map $d\phi\colon T_Q(X)\to T_P(B)$ is surjective.
By Lemma~\xref{lemma:Aut-P}
the group $G''$ acts faithfully
on the Zariski tangent space $T_Q(X)$, and
the group $G'_{B}$ acts faithfully
on the Zariski tangent space $T_P(B)$.
The map $d\phi$ is $G''$-equivariant and so $G''$ has one-dimensional
invariant
subspace $\Ker (d\phi)\subset T_Q(X)\cong \Bbbk^2$.
In this case $G''$ must be abelian with $[G:G'']\le 12\cdot 12=144$.

Now consider the case when $C$ is a reducible conic, i.e.
it is a union
of two copies of $\P^1$ meeting at one point, say $Q$.
Then $Q$ is $G'$-invariant. There exists a subgroup $G''\subset G'$ of index
$[G':G'']\le 2$ such that both irreducible components $C_1,\, C_2\subset C$ are invariant with
respect to~$G''$. In this case subspaces $T_Q(C_i)\subset T_Q(X)$ are $G''$-invariant
and as above $G''$ is abelian with $[G:G'']\le 12\cdot 2=24$.

Therefore, one has
\begin{equation*}
\Phi_a(Y,G)=\Phi_a(X,G)\le [G:G'']\le 144.\qedhere
\end{equation*}
\end{proof}

\begin{corollary}\label{corollary:barJ-for-surfaces}
Let $X$ be a smooth rational surface.
Then one has $\bar{J}\big(\Aut(X)\big)\le 288$.
\end{corollary}

\begin{corollary}\label{corollary:Cr-2}
One has $\bar{J}\big(\Cr_2(\Bbbk)\big)=288$.
\end{corollary}
\begin{proof}
Let $G\subset\Cr_2(\Bbbk)$ be a finite group.
It is enough to study the weak Jordan constant $\bar{J}(G)$.
Regularizing the action of $G$ and taking an equivariant
desingularization (see e.\,g.~\mbox{\cite[Lemma-Definition~3.1]{Prokhorov-Shramov-J}}),
we may assume that~\mbox{$G\subset\Aut(X)$} for a smooth rational surface~$X$.
Now the bound $\bar{J}\big(\Cr_2(\Bbbk)\big)\le 288$ follows from Corollary~\xref{corollary:barJ-for-surfaces}.
The equality is due to Lemma~\ref{lemma:weak-constant-quadric}\ref{lemma:weak-constant-quadric-iii}.
\end{proof}

A direct consequence of Corollary~\ref{corollary:Cr-2} is that the weak Jordan constant of
the Cremona group of rank~$2$ is bounded by~$288$ for an arbitrary (not necessarily algebraically closed)
base field. Together with Remark~\ref{remark:Pyber}
this gives a proof of Proposition~\ref{proposition:Cr-2}.

\subsection{Non-rational surfaces}

We conclude this section by three easy observations concerning automorphism groups
of certain non-rational surfaces.

\begin{lemma}\label{lemma:ruled-surface}
Let $C$ be a smooth curve of genus $g\ge 2$, and let $S$ be a
ruled surface over~$C$.
Then the group $\Aut(S)$ is Jordan with $\bar{J}\big(\Aut(S)\big)\le 1008(g-1)$.
\end{lemma}
\begin{proof}
Let $G\subset\Aut(S)$ be a finite group. It is enough to prove the corresponding
bound for~$\bar{J}(G)$.
There is an exact sequence of groups
\begin{equation*}
1\to G_{\phi}\longrightarrow G\longrightarrow G_C\to 1,
\end{equation*}
where $G_{\phi}$ acts by fiberwise automorphisms with
respect to $\phi$, and $G_C\subset\Aut(C)$.
One has
\begin{equation*}
|G_C|\le 84(g-1)
\end{equation*}
by the Hurwitz bound.
On the other hand, the group $G_{\phi}$ is a subgroup
of~\mbox{$\Aut(\P^1)\cong\PGL_2(\Bbbk)$},
so that $G_{\phi}$ contains an abelian subgroup
$H$ of index
\begin{equation*}
[G_{\phi}:H]\le 12
\end{equation*}
by Corollary~\xref{corollary:GL2}.
Thus one has
\begin{equation*}
\bar{J}(G)\le [G:H]=[G:G_{\phi}]\cdot [G_{\phi}:H]=
[G:G_{\phi}]\cdot |G_C|\le 12\cdot 84\cdot (g-1)=1008(g-1).\qedhere
\end{equation*}
\end{proof}

\begin{lemma}[{cf. \cite[Corollary~2.15]{Prokhorov-Shramov-J}}]
\label{lemma:abelian-surface}
Let $S$ be an abelian surface.
Then the group~\mbox{$\Aut(S)$} is Jordan with $\bar{J}\big(\Aut(S)\big)\le 5760$.
\end{lemma}
\begin{proof}
One has $\Aut(S)\cong A\rtimes\Gamma$, where $A$ is an abelian group
(that is identified with the group of points on $S$), and
$\Gamma$ is a subgroup of $\GL_4(\Z)$.
Thus $\Aut(S)$ is Jordan with
\begin{equation*}
\bar{J}\big(\Aut(S)\big)\le [\Aut(S):A]=|\Gamma|\le 5760
\end{equation*}
by the Minkowski bound for $\GL_4(\Z)$ (see e.\,g.~\cite[\S1.1]{Serre2007}).
\end{proof}

To obtain a bound for a weak Jordan constant in the last case we will use some
purely group-theoretic facts.

\begin{proposition}[{see Corollary~2 of Theorem~1.17 in Chapter~2 of~\cite{Suzuki82}}]
\label{proposition:Suzuki}
Let $p$ be a prime number, $G$ be a group of order $p^n$, and $A\subset G$ be an abelian normal subgroup of maximal possible order~$p^a$. Then $2n\le a(a+1)$.
\end{proposition}

\begin{lemma}\label{lemma:group-theory}
Let $G$ be a finite group with $|G|\le 79380$.
Then
\begin{equation*}
\bar{J}(G)\le 9922.
\end{equation*}
\end{lemma}
\begin{proof}
Suppose that $|G|$ is divisible by a
prime number $p$. Then $G$ contains a cyclic subgroup of order $p$, so that
\begin{equation*}
\bar{J}(G)\le \frac{|G|}{p}.
\end{equation*}
In particular, if $|G|$ is divisible by a prime $p\ge 11$, then
\begin{equation*}
\bar{J}(G)\le \frac{|G|}{11}<7217.
\end{equation*}

Similarly, suppose that $p$ is a prime such that $|G|$ is divisible by $p^2$.
Let $G_p\subset G$ be a Sylow $p$-subgroup.
Then $|G_p|\ge p^2$. If $|G_p|=p^2$, then $G_p$ is abelian, so that
\begin{equation*}
\bar{J}(G)\le [G:G_p]=\frac{|G|}{p^2}.
\end{equation*}
If $|G_p|\ge p^3$, then $G_p$ contains an abelian subgroup
$A$ of order $|A|\ge p^2$ by Proposition~\ref{proposition:Suzuki}, and we again have
\begin{equation*}
\bar{J}(G)\le [G:A]\le \frac{|G|}{p^2}.
\end{equation*}
In particular, if there is a prime $p\ge 3$ such that $|G|$ is divisible
by $p^2$, then
\begin{equation*}
\bar{J}(G)\le \frac{|G|}{p^2}\le \frac{|G|}{9}\le 8820.
\end{equation*}

Now suppose that $|G|$ is not divisible by any prime greater than $7$,
and $|G|$ is not divisible by a square of any prime greater than $2$. This means
that
\begin{equation*}
|G|=2^{\alpha}\cdot 3^{\beta}\cdot 5^{\gamma}\cdot 7^{\delta},
\end{equation*}
where $\beta,\gamma,\delta\in\{0,1\}$. If $\alpha\le 3$, then
\begin{equation*}
\bar{J}(G)\le |G|\le 2^3\cdot 3\cdot 5\cdot 7=840.
\end{equation*}
Thus we assume that $\alpha\ge 4$. Let $G_2\subset G$ be a Sylow $2$-subgroup.
Applying Proposition~\ref{proposition:Suzuki} once again, we see that $G_2$ contains an abelian
subgroup~$A$ of order $|A|\ge 8$. Hence one has
\begin{equation*}
\bar{J}(G)\le [G:A]\le \frac{|G|}{8}<9923.\qedhere
\end{equation*}
\end{proof}

Now we are ready to bound a weak Jordan constant for automorphism groups
of surfaces of general type of low degree.

\begin{lemma}
\label{lemma:general-type-45}
Let $S$ be a smooth minimal surface of general type of degree
$K_S^2\le 45$. Then the group~\mbox{$\Aut(S)$} is Jordan with
$\bar{J}\big(\Aut(S)\big)\le 9922$.
\end{lemma}
\begin{proof}
By
\cite{Xiao1995}
one has
\begin{equation*}
|\Aut(S)|\le 42^2\cdot K_S^2\le 79380.
\end{equation*}
Thus the group~\mbox{$\Aut(S)$} is Jordan with
\begin{equation*}
\bar{J}\big(\Aut(S)\big)\le 9922
\end{equation*}
by Lemma~\xref{lemma:group-theory}.
\end{proof}

\section{Terminal singularities}
\label{section:terminal}

In this section we study Jordan property for automorphism groups of
germs of three-dimensional terminal singularities,
and derive some conclusions about automorphism groups
of non-Gorenstein terminal Fano threefolds.

\subsection{Local case}

Recall from~\S\xref{subsection:linear-prelim} that for an arbitrary variety $U$ and a point $P\in U$
we denote by $\Aut_P(U)$ the stabilizer of $P$ in $\Aut(U)$.
Now we are going to estimate a weak Jordan constant of a
group~\mbox{$\Aut_P(U)$}, where~\mbox{$P\in U$} is a three-dimensional
terminal singularity.

\begin{lemma}\label{lemma:dim-3-terminal}
Let $U$ be a threefold, and $P\in U$ be a terminal singular point of~$U$.
Let~\mbox{$G\subset\Aut_P(U)$} be a finite subgroup.
Then for some positive integer $r$ there is an extension
\begin{equation}
 \label{equation-terminal-singular-point-extension}
1\longrightarrow \mumu_r
\longrightarrow \tilde{G} \longrightarrow G
\longrightarrow 1
\end{equation}
such that the following assertions hold.
\begin{enumerate}
\item There is an embedding $\tilde{G}\subset\GL_4(\Bbbk)$,
and the group $\tilde{G}$ has a semi-invariant of degree~$2$.
\item If $(U,P)$ is a cyclic
quotient singularity, then there is an embedding $\tilde{G}\subset\GL_3(\Bbbk)$.
\item Let $D$ be a $G$-invariant boundary on $X$ such that the log pair $(U,D)$
is log canonical and such that there is a minimal center~$C$ of log canonical singularities
is a $G$-invariant curve containing~$P$
\textup(see \cite[Proposition~1.5]{Kawamata1997}\textup).
Then $\tilde{G}\subset\Bbbk^*\times\GL_3(\Bbbk)$.
\end{enumerate}
\end{lemma}
\begin{proof}
Let $r\ge 1$ be the index of $U\ni P$, i.\,e.
$r$ equals the minimal positive
integer~$t$ such that~$tK_U$ is Cartier at~$P$.
Replacing $U$ by a smaller $G$-invariant neighborhood of $P$ if necessary,
we may assume that $rK_U\sim 0$. Consider the index-one cover
\begin{equation*}
\pi\colon (U^\sharp,P^\sharp)\to (U,P)
\end{equation*}
(see \cite[Proposition~3.6]{Reid-YPG1987}).
Then $U^\sharp\ni P^\sharp$ is a terminal singularity of index~$1$,
and~\mbox{$U\cong U^\sharp /\mumu_r$}.
Note that $U^\sharp\ni P^\sharp$ is a hypersurface singularity,
\mbox{i.\,e.~$\dim T_{P^\sharp}(U^\sharp)\le 4$} (see~\mbox{\cite[Corollary~3.12(i)]{Reid-YPG1987}}).
Moreover, $U^\sharp$ is smooth at $P^\sharp$ if~\mbox{$(U,P)$}
is a cyclic quotient singularity.

By construction of the index one cover
every element of $\Aut_P(U)$ admits~$r$
lifts to~\mbox{$\Aut(U^\sharp, P^\sharp)$}. Thus
we have a natural exact sequence~\eqref{equation-terminal-singular-point-extension},
where $\tilde{G}$ is some subgroup of $\Aut_{P^\sharp}(U^\sharp)$.
Furthermore, by Lemma~\xref{lemma:Aut-P} we know that $\tilde{G}\subset\GL_3(\Bbbk)$
if $U^\sharp$ is smooth at $P^\sharp$. This gives assertion~(ii).

Now suppose that $\dim T_{P^\sharp}(U^\sharp)= 4$.
By Lemma~\xref{lemma:Aut-P} one has an embedding~\mbox{$\tilde{G}\subset\GL_4(\Bbbk)$}.
Moreover, $U^\sharp\ni P^\sharp$ is a hypersurface singularity
of multiplicity $2$ by~\mbox{\cite[Corollary~5.38]{Kollar-Mori-1988}}.
This means that the kernel of the natural map
\[
\Sym^2\big(\mathfrak m_{P^{\sharp}, U^{\sharp}}/\mathfrak m_{P^{\sharp}, U^{\sharp}}^2\big)
\longrightarrow
\mathfrak m_{P^{\sharp}, U^{\sharp}}^2/\mathfrak m_{P^{\sharp}, U^{\sharp}}^3
\]
is generated by an element of degree $2$.
Therefore,
the group $\tilde{G}$ has a semi-invariant polynomial of degree $2$.
This completes the proof of assertion~(i).

Finally, let $C$, $D$, and $G$ be as in assertion~(iii).
Put $D^\sharp=\pi^*D$ and $C^\sharp=\pi^{-1}(C)$.
One can show that $C^\sharp$ is again a minimal center of log canonical singularities of
$(U^\sharp, D^\sharp)$ (cf. \cite[Proposition~5.20]{Kollar-Mori-1988}).
In particular, $C^\sharp$ is smooth (see~\cite[Theorem~1.6]{Kawamata1997}).
As above, one has an embedding $\tilde{G}\subset\GL\big(T_{P^\sharp}(U^\sharp)\big)$.
Moreover, since $C^{\sharp}$ is $\tilde{G}$-invariant,
we have a decomposition
of $\tilde{G}$-representations
\begin{equation*}
T_{P^\sharp}(U^\sharp)=T_1\oplus T_3,
\end{equation*}
where $T_1=T_{P^\sharp}(C^\sharp)\cong\Bbbk$
and $\dim T_3=3$.
Hence, one has
\begin{equation*}
\tilde{G}\subset \GL(T_1)\times \GL(T_3)\cong\Bbbk^*\times\GL_3(\Bbbk),
\end{equation*}
which proves assertion~(iii).
\end{proof}

\begin{corollary}\label{corollary:dim-3-terminal}
Let $U$ be a threefold, and $P\in U$ be a terminal singularity.
Then the following assertions hold.
\begin{enumerate}
\item The group
$\Aut_P(U)$ is Jordan with
\begin{equation*}
\bar{J}(\Aut_P(U))\le 288.
\end{equation*}
\item If $(U,P)$ is a cyclic
quotient singularity, then $\Aut_P(U)$ is Jordan with
\begin{equation*}
\bar{J}(\Aut_P(U))\le 72.
\end{equation*}
\item Let $C\ni P$ be a curve contained in $U$
and $\Gamma\subset\Aut_P(U)$
be a subgroup such that $C$ is $\Gamma$-invariant.
Assume that $C$ is a minimal center of log canonical singularities
of the log pair $(U,D)$ for some $\Gamma$-invariant boundary $D$.
Then $\Gamma$ is Jordan with
\begin{equation*}
\bar{J}(\Gamma)\le 72.
\end{equation*}
\end{enumerate}
\end{corollary}
\begin{proof}
Suppose that $G\subset\Aut_P(U)$ is a finite subgroup.
It is enough to prove the corresponding bounds for the constant $\bar{J}(G)$.
One has $\bar{J}(G)\le\bar{J}(\tilde{G})$, where $\tilde{G}$ is the extension
of $G$ given by Lemma~\xref{lemma:dim-3-terminal}.
Thus, assertion~(i) follows from Lemma~\xref{lemma:dim-3-terminal}(i) and
Lemma~\xref{lemma:GL4-quadric}, while
assertion~(ii) follows from Lemma~\xref{lemma:dim-3-terminal}(ii) and
Lemma~\xref{lemma:weak-GL3}.

Suppose that $\Gamma$ is as in assertion~(iii), and $G\subset\Gamma$.
Then
\[
\bar{J}(G)\le\bar{J}(\tilde{G})\le
\bar{J}\big(\Bbbk^*\times\GL_3(\Bbbk)\big)=
\bar{J}\big(\GL_3(\Bbbk)\big)
\]
by Lemma~\xref{lemma:dim-3-terminal}(iii). Therefore, assertion~(iii) follows from
Lemma~\xref{lemma:weak-GL3}.
\end{proof}

\subsection{Non-Gorenstein Fano threefolds}

Now we will use Corollary~\xref{corollary:dim-3-terminal} to study automorphism
groups of non-Gorenstein terminal Fano threefolds.

\begin{lemma}\label{lemma:non-Gorenstein-Fano-3-fold}
Let $X$ be a Fano
threefold with terminal singularities.
Suppose that $X$ has a non-Gorenstein singular point.
Then the group~\mbox{$\Aut(X)$} is Jordan with
\begin{equation*}
\bar J\big(\Aut(X)\big)\le 4608.
\end{equation*}
\end{lemma}
\begin{proof}
We use the methods of \cite[\S6]{Prokhorov2009e}. Let $P_1$ be a non-Gorenstein point
and~\mbox{$P_1,\ldots, P_N\in X$} be its $\Aut(X)$-orbit.
Let $r$ be the index of points $P_1,\ldots, P_N\in X$.
By the orbifold Riemann--Roch theorem and
Bogomolov--Miyaoka inequality
we have
\begin{equation*}
\frac{3}{2} N\le \Big( r-\frac{1}{r}\Big)N\le 24
\end{equation*}
(see \cite{Kawamata-1992bF}, \cite{KMMT-2000}).
This immediately implies that $N\le 16$.

The subgroup $\Aut_{P_1}(X)\subset\Aut(X)$ stabilizing the point
$P_1$ has index
\begin{equation*}
[\Aut(X):\Aut_{P_1}(X)]\le N.
\end{equation*}
Thus we have
\begin{equation*}
\bar{J}\big(\Aut(X)\big)\le
[\Aut(X):\Aut_{P_1}(X)]\cdot\bar{J}\big(\Aut_{P_1}(X)\big)\le
N\cdot \bar{J}\big(\Aut_{P_1}(X)\big)\le 16\cdot 288=4608
\end{equation*}
by Corollary~\xref{corollary:dim-3-terminal}(i).
\end{proof}

\begin{remark}
It is known that terminal non-Gorenstein Fano threefolds are bounded,
i.e. they belong to an algebraic family (see \cite{Kawamata-1992bF}, \cite{KMMT-2000}).
However it is expected that the class of these varieties is huge \cite{GRD}.
There are only few results related to some special types of these Fanos
(see e.g. \cite{Brown-Suzuki-2007mm}, \cite{Prokhorov-e-QFano7}).
\end{remark}

\section{Mori fiber spaces}
\label{section:Mfs}
Recall that a $G$-equivariant morphism $\phi\colon X\to S$ of normal
varieties acted on by a finite group
$G$ is a \emph{$G$-Mori fiber space},
if $X$ has terminal $G\Q$-factorial singularities, one has~\mbox{$\dim(S)<\dim(X)$},
the fibers of $\phi$ are connected, the anticanonical divisor~$-K_X$
is $\phi$-ample, and the relative $G$-invariant Picard number
$\rho^G(X/S)$ equals~$1$.
If the dimension of~$X$ equals~$3$, there are three cases:
\begin{itemize}
\item
$S$ is a point, $-K_X$ is ample; in this case $X$ is said to be a $G\Q$-Fano threefold,
and $X$ is a $G$-Fano threefold provided that the singularities of $X$ are Gorenstein;
\item
$S$ is a curve, a general fiber of $\phi$ is a del Pezzo surface; in this case $X$ is said to be a $G\Q$-del Pezzo fibration;
\item
$S$ is a surface, a general fiber of $\phi$ is a rational curve; in this case $X$ is said to be a $G\Q$-conic bundle.
\end{itemize}
The goal of this section is to estimate weak Jordan constants for
the automorphism groups of varieties of $G\Q$-conic bundles and $G\Q$-del Pezzo fibrations.

\subsection{Conic bundles}

We start with automorphism groups of $G\Q$-conic bundles.

\begin{lemma}\label{lemma:conic-bundle}
Let $G$ be a finite group, and $\phi\colon X\to S$ be a
three-dimensional $G$-equivariant fibration into rational curves
over a rational surface $S$.
Then $\bar{J}(G)\le 3456$.
\end{lemma}
\begin{proof}
By~\cite{Avilov2014}
we may assume that
$X$ and $S$ are smooth, and any fiber of $\phi$ is a (possibly reducible or non-reduced) conic.
There is an exact sequence of groups
\begin{equation*}
1\to G_{\phi}\longrightarrow G\stackrel{\gamma}\longrightarrow G_S\to 1,
\end{equation*}
where $G_{\phi}$ acts by fiberwise automorphisms with
respect to $\phi$, and $G_S\subset\Aut(S)$.
By Lemma~\xref{lemma:dim-2-constants} there is an abelian subgroup $G_S'\subset G_S$
of index
\begin{equation*}
[G_S:G_S']\le 288
\end{equation*}
such that $G_S'$
acts on $S$ with a fixed point. Let $P\in S$ be one
of the fixed points of $G_S'$, and let
\begin{equation*}
C=\phi^{-1}(P)\subset X
\end{equation*}
be the fiber of $\phi$ over the point $P$. Put $G'=\gamma^{-1}(G_S')$.
Then $G'$ is a subgroup of $G$ of index
\begin{equation*}
[G:G']=[G_S:G_S']\le 288,
\end{equation*}
and the fiber $C$ is $G'$-invariant.

The fiber $C$ is a reduced conic, so that it is either isomorphic to $\P^1$, or is a union
of two copies of $\P^1$ meeting at one point.

In the former case there is a point $Q\in C$ that is invariant
with respect to some subgroup~\mbox{$G''\subset G'$} of index
\begin{equation*}
[G':G'']\le 12
\end{equation*}
by Lemma~\xref{lemma:dim-1}. In the latter case the
intersection point $Q$ of the irreducible components~$C_1$ and~$C_2$
of~$C$ is invariant with respect to the group $G'$,
and there exists a subgroup~\mbox{$G''\subset G'$} of index
$[G':G'']\le 2$
such that~$C_1$ and~$C_2$ are invariant with
respect to~$G''$.

By Lemma~\xref{lemma:Aut-P} the group $G''$ acts faithfully
on the Zariski tangent space $T_Q(X)$, and
the group $G'_S$ acts faithfully
on the Zariski tangent space $T_P(S)$. As we have seen,
the group $G''$ preserves
the point $Q$ and a tangent direction
\begin{equation*}
v\in T_Q(X)\cong\Bbbk^3
\end{equation*}
that lies in the kernel of the natural projection
$T_Q(X)\to T_P(S)$.
Moreover, there is an embedding
\begin{equation*}
G''\hookrightarrow\Gamma_1\times\Gamma_2,
\end{equation*}
where $\Gamma_1\subset\Bbbk^*$, and $\Gamma_2\subset G_S'$.
Since $G_S'$ and $\Bbbk^*$ are abelian groups, we conclude
that so is~$G''$. Therefore, one has
\begin{equation*}
\bar{J}(G)\le [G:G'']=[G:G']\cdot [G':G'']\le 288\cdot 12=3456.\qedhere
\end{equation*}
\end{proof}

\subsection{Del Pezzo fibrations}

Before we pass to the case of $G\Q$-del Pezzo fibrations we will
establish some auxiliary results.
Recall \cite[Definition~3.7]{Kollar-ShB-1988} that a surface singularity
is said to be \textit{of type $T$}
if it is a quotient singularity and admits a $\Q$-Gorenstein one-parameter smoothing.

\begin{lemma}\label{lemma-T-singularities}
Let $X$ be a normal threefold with at worst isolated singularities
and let~\mbox{$S\subset X$} be an effective Cartier divisor
such that the log pair $(X,S)$ is purely log terminal
(see \cite[\S2.3]{Kollar-Mori-1988}). Then $S$ has only singularities of type $T$.
\end{lemma}
\begin{proof}
Regard $X$ as the total space of a deformation of $S$. By our assumptions
divisors~\mbox{$K_X+S$} and~$S$ are $\Q$-Cartier. Hence
$X$ is $\Q$-Gorenstein.
By the inversion of adjunction (see \cite[Theorem~5.20]{Kollar-Mori-1988})
the surface $S$ has only Kawamata log terminal (i.\,e. quotient) singularities
(see \cite[Theorem~5.50]{Kollar-Mori-1988}).
Hence the singularities of $S$ are of type~$T$.
\end{proof}

\begin{lemma}\label{lemma:P1xP1-degeneration-1}
Let $S$ be a singular del Pezzo surface with $T$-singularities.
Assume that~$S$ has at least one non-Gorenstein point.
Then $\Aut(S)$ has an orbit of length at most~$2$ on~$S$.
\end{lemma}
\begin{proof}
Assume that $\Aut(S)$ has no orbits of length at most $2$
on $S$. By \cite[Proposition~2.6]{Hacking-Prokhorov-2010} one has
\begin{equation*}
\dim |-K_S|=K_S^2\ge 1.
\end{equation*}
Write $|-K_S|=F+|M|$, where $|M|$ is a linear system without fixed components
and $F$ is the fixed part of $|-K_S|$, so that
\begin{equation*}
\dim |M|=\dim |-K_S|=K_S^2.
\end{equation*}
By \cite[Theorem~4.2]{Prokhorov-degenerations-del-Pezzo} the log pair $(S, M+F)$ is
log canonical for a general member~\mbox{$M\in |M|$}.
In particular, $F$ is reduced.
Let $\Sing'(S)$ be the set
of non-Du Val points of~$S$.
By our assumptions $\Sing'(S)\neq \varnothing$.
Clearly,
any member of $|-K_S|$ contains~\mbox{$\Sing'(S)$}; otherwise~$-K_S$ is Cartier
at some point of $\Sing'(S)$, so that this point is Du Val on~$S$.
Since the log pair $(S,F+M)$ is
log canonical and $K_S+F+M$ is Cartier, by the classification of two-dimensional log canonical singularities
(\cite[Theorem~4.15]{Kollar-Mori-1988}) the divisor~\mbox{$F+M$} has two analytic branches
at each point of $\Sing(X) \cap \Supp(F+M)$.
In particular, we have
\begin{equation*}
\Sing'(S)\subset \Sing(F+M).
\end{equation*}
Thus by our assumption $\Sing(F+M)$ contains
at least three points.
Furthermore, since the support of $F+M$ is connected, by adjunction one has~\mbox{$p_a(F+M)=1$},
all irreducible components of $F+M$
are smooth rational curves and
the corresponding dual graph is a combinatorial cycle. Moreover,
the number of these irreducible components is at least $3$.

First assume that $F\neq 0$.
By Shokurov's connectedness theorem (see e.g
\cite[Theorem~5.48]{Kollar-Mori-1988})
we know that $F$ is connected.
Hence $F$ is a connected chain of rational curves.
In this situation $\Aut(S)$ acts on $F$ so that there exists either
a fixed point~\mbox{$P\in \Sing(F)$} or an invariant irreducible component $F_1\subset F$
(cf. the proof of Lemma~\xref{lemma:dim-1}).
In the first case we have a contradiction with our assumption
and in the second case $\Aut(S)$ permutes two points of intersection
of $F_1$ with~\mbox{$\Supp(F-F_1)$}, again a contradiction.

Thus $F=0$ and so $\Sing'(S)\subset \Bs|M|$ and $\Sing'(S)\subset \Sing(M)$.
Since $\Sing'(S)$ contains at least three points and $p_a(M)=1$, the divisor~$M$ is reducible.
By Bertini's theorem the linear system $|M|$ is composed of a pencil, which means
that there is a pencil~$|L|$ such that~\mbox{$|M|=n|L|$} for some $n\ge 2$, and $\Sing'(S)\subset \Bs|L|$.
Since the log pair~\mbox{$(S,M)$} is log canonical, there are exactly two irreducible
components of $M$ passing through
any point $P\in \Sing'(S)$, see \cite[Theorem~4.15]{Kollar-Mori-1988}.
Since $\Sing'(S)$ contains at least three points,
the dual graph of $M$ cannot be a combinatorial cycle,
a contradiction.
\end{proof}

\begin{lemma}\label{lemma:DP-smooth-fiber}
Let $X$ be a threefold, and $G\subset\Aut(X)$
be a finite subgroup.
Suppose that there is a $G$-invariant smooth del Pezzo surface
$S$ contained in the smooth locus of $X$.
Then~\mbox{$\bar{J}(G)\le 288$}.
\end{lemma}
\begin{proof}
There is an exact sequence of groups
\begin{equation*}
1\to K\longrightarrow G\stackrel{\beta}\longrightarrow H\to 1,
\end{equation*}
where $K$ acts on $S$ trivially, and
$H\subset\Aut(S)$.
By Lemma~\xref{lemma:Phi-a-DP} there is a point
$Q\in S$ fixed by an abelian subgroup $H_Q\subset H$ of
index $[H:H_Q]\le 288$.
Put $G_Q=\beta^{-1}(H_Q)$. Then $G_Q\subset G$ is a subgroup
that fixes the point $Q$,
such that the index
\begin{equation*}
[G:G_Q]=[H:H_Q]\le 288.
\end{equation*}
By Lemma~\xref{lemma:Aut-P} the group $G_Q$ acts
faithfully on the Zariski tangent space
$T_Q(X)\cong\Bbbk^3$. The two-dimensional Zariski tangent space
$T_Q(S)\subset T_Q(X)$ is $G_Q$-invariant, and thus
$G_Q$ is contained in a subgroup
\begin{equation*}
\Bbbk^*\times\GL\big(T_Q(S)\big)\cong
\Bbbk^*\times\GL_2(\Bbbk)
\subset\GL_3(\Bbbk)\cong\GL\big(T_Q(X)\big).
\end{equation*}
Hence $G_Q\subset A\times H_Q$, where $A\subset\Bbbk^*$ is some
cyclic group.
Therefore, the group $G_Q$ is abelian,
so that one has
\begin{equation*}
\bar{J}(G)\le [G:G_Q]\le 288.\qedhere
\end{equation*}
\end{proof}

\begin{remark}\label{remark:invariant-set}
Let $G\subset\Aut(X)$ be a finite subgroup,
and $\Sigma\subset X$ be a non-empty finite subset.
Then a stabilizer $G_P\subset G$ of a point $P\in\Sigma$ has
index $[G:G_P]\le |\Sigma|$, so that by Remark~\xref{remark:Pyber} one has
\begin{equation*}
\bar{J}(G)\le |\Sigma|\cdot\bar{J}(G_P)\le |\Sigma|\cdot\bar{J}
\big(\Aut_P(X)\big).
\end{equation*}
\end{remark}

Now we are ready to finish with weak Jordan constants of
rationally connected three-dimensional $G\Q$-del Pezzo fibrations.

\begin{lemma}\label{lemma:dP-fibration}
Let $G$ be a finite group, and $\phi\colon X\to B\cong\P^1$ be a
three-dimensional $G\Q$-del Pezzo fibration.
Then $\bar{J}(G)\le 10368$.
\end{lemma}
\begin{proof}
There is an exact sequence of groups
\begin{equation*}
1\to G_{\phi}\longrightarrow G\stackrel{\alpha}\longrightarrow G_{B}\to 1,
\end{equation*}
where $G_{\phi}$ acts by fiberwise automorphisms with
respect to $\phi$, and
\begin{equation*}
G_{B}\subset\Aut(B)\cong\PGL_2(\Bbbk).
\end{equation*}
By Lemma~\xref{lemma:dim-1} there is a subgroup $G_{B}'\subset G_{B}$
of index $[G_{B}:G_{B}']\le 12$ such that~$G_{B}'$
acts on $B$ with a fixed point.

Let $P\in B$ be one
of the fixed points of $G_{B}'$, let $F= \phi^*(P)$
be the scheme fiber over~$P$, and let $S=F_{\red}$.
Put $G'=\alpha^{-1}(G_{B}')$.
Then $G'$ is a subgroup of $G$ of index
\begin{equation*}
[G:G']=[G_{B}:G_{B}']\le 12,
\end{equation*}
and the fiber $S$ is $G'$-invariant. In particular, one has
\begin{equation*}
\bar{J}(G)\le [G:G']\cdot\bar{J}(G').
\end{equation*}

Suppose that $F$ is a multiple fiber of $\phi$, i.e. $S\neq F$.
Then by~\cite{MoriProkhorov-multfibers}
there is a $G$-invariant set $\Sigma\subset S$ of singular
points of $X$ such that either $|\Sigma|\le 3$, or $|\Sigma|=4$ and
$\Sigma$ consists of cyclic quotient singularities.
In the former case Remark~\xref{remark:invariant-set}
and Corollary~\xref{corollary:dim-3-terminal}(i) imply that
\begin{equation*}
\bar{J}(G)\le 12\cdot 3\cdot 288=10368.
\end{equation*}
In the latter case Remark~\xref{remark:invariant-set}
and Corollary~\xref{corollary:dim-3-terminal}(ii) imply that
\begin{equation*}
\bar{J}(G)\le 12\cdot 4\cdot 72=3456.
\end{equation*}

Therefore, we can assume that $S$ is not a multiple fiber of $\phi$.
In particular, $S=F$ is a Cartier divisor on $X$.

Suppose that the log pair $(X, S)$ is not purely log terminal
(see \cite[\S2.3]{Kollar-Mori-1988}).
Let $c$ be the log canonical threshold
of the log pair $(X,S)$
(cf. the proof of~\cite[Lemma~3.4]{ProkhorovShramov-RC}).
Let~\mbox{$Z_1\subset S$} be a minimal center of log
canonical singularities of the log pair~\mbox{$(X, cS)$},
see \cite[Proposition~1.5]{Kawamata1997}.
Since $(X,S)$ is not purely log terminal, we conclude
that $c<1$, so that $\dim(Z)\le 1$.
It follows from~\cite[Lemma~2.5]{ProkhorovShramov-RC}
that $Z$ is $G'$-invariant. If~$Z$ is a point, then
\begin{equation*}
\bar{J}(G)\le [G:G']\cdot\bar{J}(G')\le
12\cdot 288=3456
\end{equation*}
by Corollary~\xref{corollary:dim-3-terminal}(i).
Thus we assume that $Z$ is a curve.
Using~\cite[Lemma~2.5]{ProkhorovShramov-RC} once again,
we see that $Z$ is smooth and rational.
By Lemma~\xref{lemma:dim-1} there is a subgroup
$G''\subset G'$ of index $[G':G'']\le 12$ such that
$G''$ has a fixed point on $Z$.
Hence
\begin{equation*}
\bar{J}(G)\le [G:G'']\cdot\bar{J}(G'')\le
144\cdot 72=10368
\end{equation*}
by Corollary~\xref{corollary:dim-3-terminal}(iii).

Therefore, we may assume that the log pair
$(X,S)$ is purely log terminal.
Then by~\mbox{\cite[Theorem 5.50]{Kollar-Mori-1988}} the surface $S$
is a del Pezzo surface with only Kawamata log terminal
singularities. Moreover, the singularities of $S$ are of type $T$
(see Lemma \xref{lemma-T-singularities}).
If~$K_S$ is not Cartier, Lemma~\xref{lemma:P1xP1-degeneration-1}
implies
that there is a $G'$-orbit of length at most~$2$
contained in~$S$. In this case we have
\begin{equation*}
\bar{J}(G)\le [G:G']\cdot\bar{J}(G')\le
12\cdot 2\cdot 288=6912
\end{equation*}
by Remark~\xref{remark:invariant-set}
and Corollary~\xref{corollary:dim-3-terminal}(i).

Therefore, we may assume that $K_S$ is Cartier and so $S$ has at worst Du Val
singularities. Denote their number by $m(S)$.
Then by Noether formula applied to the minimal resolution
we have
\begin{equation*}
m(S)\le 9-K_S^2\le 8.
\end{equation*}
Thus by Remark~\xref{remark:invariant-set} and Corollary~\xref{corollary:dim-3-terminal}(i)
we have
\begin{equation*}
\bar{J}(G)\le [G:G']\cdot\bar{J}(G')\le
2\cdot 9\cdot 288=5184.
\end{equation*}

Therefore, we are left with the case when $S$ is smooth.
Now Lemma~\xref{lemma:DP-smooth-fiber}
implies that
\begin{equation*}
\bar{J}(G)\le [G:G']\cdot\bar{J}(G')\le
12\cdot 288=3456
\end{equation*}
and completes the proof.
\end{proof}

\section{Gorenstein Fano threefolds}
\label{section:Fano}

Let $X$ be a Fano threefold
with at worst terminal Gorenstein singularities.
In this case, the number
\begin{equation*}
\g(X)=\frac{1}{2}(-K_X)^3+1
\end{equation*}
is called the \textit{genus} of $X$.
By Riemann--Roch theorem and Kawamata--Viehweg vanishing one has
\begin{equation*}
\dim |-K_X|=\g(X)+1
\end{equation*}
(see e.\,g. \cite[2.1.14]{Iskovskikh-Prokhorov-1999}).
In particular, $\g(X)$ is an integer, and $\g(X)\ge 2$.
The maximal number $\iota=\iota(X)$
such that $-K_X$ is divisible by $\iota$ in
$\Pic(X)$ is called the \textit{Fano index}, or sometimes just \emph{index}, of~$X$.
Recall that~\mbox{$\Pic(X)$} is a finitely generated torsion free abelian group,
see e.g.~\mbox{\cite[Proposition 2.1.2]{Iskovskikh-Prokhorov-1999}}.
The rank $\rho(X)$ of the free abelian group
$\Pic(X)$ is called the \emph{Picard rank} of~$X$.
Let $H$ be a divisor class such that
$-K_X\sim\iota(X) H$.
The class $H$ is unique since $\Pic(X)$ is torsion free.
Define the \textit{degree} of $X$ as~\mbox{$\dd(X)=H^3$}.
The goal of this section is to bound weak Jordan constants for
automorphism groups of singular terminal Gorenstein Fano threefolds.

\subsection{Low degree}
\label{subsection:low-degree}

We start with the case of small anticanonical degree.
We will use notation and results of~\S\ref{subsection:WCI}.

\begin{proposition}[{cf. \cite[Lemma~4.4.1]{Kuznetsov-Prokhorov-Shramov}}]
\label{proposition:double-cover}
Let $X$ be a Fano threefold with terminal
Gorenstein singularities such that $\rho(X)=1$.
Suppose that $H$ is not very ample,
i.\,e. one of the following
possibilities holds
\textup(see \cite[Theorem~0.6]{Shin1989}, \cite[Corollary~0.8]{Shin1989},
\cite[Theorem 1.1]{Jahnke-Radloff-2006},
\cite[Theorem 1.4]{Przhiyalkovskij-Cheltsov-Shramov-2005en}\textup):
\begin{enumerate}
\renewcommand\labelenumi{(\roman{enumi})}
\renewcommand\theenumi{(\roman{enumi})}
\item\label{enumerate-(ii)}
$\iota(X)=2$ and $\dd(X)=1$;
\item\label{enumerate-(iii)}
$\iota(X)=2$ and $\dd(X)=2$;
\item\label{enumerate-(i)}
$\iota(X)=1$ and $\g(X)=2$;
\item\label{enumerate-(iv)}
$\iota(X)=1$, $\g(X)=3$, and $X$ is a double cover of a three-dimensional quadric.
\end{enumerate}
Suppose that $G\subset\Aut(X)$ is a finite group.
Then for some positive integer $r$ there is a central extension
\begin{equation*}
1\to \mumu_r\to\tilde{G}\to G\to 1
\end{equation*}
such that one has an embedding
$\tilde{G}\subset\GL_3(\Bbbk)\times\Bbbk^*$ in case~\xref{enumerate-(ii)},
an embedding $\tilde{G}\subset\GL_4(\Bbbk)$
in cases~\xref{enumerate-(iii)} and~\xref{enumerate-(i)}, and an embedding~\mbox{$\tilde{G}\subset\GL_5(\Bbbk)$}
in case~\xref{enumerate-(iv)}.
\end{proposition}
\begin{proof}
According to \cite[Corollary 0.8]{Shin1989} and
\cite[Theorem~1.5]{Przhiyalkovskij-Cheltsov-Shramov-2005en},
in cases \xref{enumerate-(ii)}, \xref{enumerate-(iii)} and~\xref{enumerate-(i)}
our Fano variety $X$ is naturally embedded as a weighted hypersurface in
the weighted projective space
$\P=\P(a_0,\ldots,a_4)$,
where
\begin{equation*}
(a_0,\ldots,a_4)=(1^3,2,3), (1^4,2), (1^4,3),
\end{equation*}
respectively.
In case \xref{enumerate-(iv)} our $X$ is naturally embedded as a weighted complete intersection of multidegree $(2,4)$
in $\P=\P(1^5,2)$.
Let $\O_X(1)$ be the restriction of the (non-invertible) divisorial sheaf
$\O_{\P}(1)$ to $X$ (see~\mbox{\cite[1.4.1]{Dolgachev-1982}}).
Since $X$ is Gorenstein, in all cases it is contained in the smooth
locus of~$\P$, and thus $\O_X(1)$ is an invertible divisorial sheaf on~$X$.
Moreover, under the above embeddings we have
\begin{equation*}
\O_X(1)=\O_X(-K_X)
\end{equation*}
in cases \xref{enumerate-(i)} and \xref{enumerate-(iv)}, while
\begin{equation*}
\O_X(1)=\O_X(-\textstyle{\frac{1}{2}}K_X)
\end{equation*}
in cases \xref{enumerate-(ii)} and \xref{enumerate-(iii)}.
Since the group $\Pic(X)$ has no torsion, in all cases
the class of $\O_X(1)$ in $\Pic(X)$
is invariant with respect to the whole automorphism group~$\Aut(X)$.
Also, the line bundle $\O_X(1)$ is ample, so that the algebra $R(X,\O_X(1))$
is finitely generated.
Therefore, by Lemma~\xref{lemma:action-on-algebra-and-WPS} for any finite subgroup $\Gamma\subset\Aut(X)$
the action of $\Gamma$ on $X$ is induced by its action on
\begin{equation*}
\P\cong \operatorname{Proj} R\big(X,\O_X(1)\big).
\end{equation*}
Thus the assertion follows from Lemma~\xref{lemma:WPS}.
\end{proof}

\begin{remark}
Assume the setup of Proposition~\xref{proposition:double-cover}.
Then using the notation of the proof of Lemma~\xref{lemma:action-on-algebra-and-WPS} one can argue
that a central extension of the group $G$ acts on the vector space
\begin{equation*}
V=\bigoplus\limits_{m=1}^N V_m,
\end{equation*}
which immediately gives its embedding into $\GL_{k_1+\ldots+k_N}(\Bbbk)$.
This would allow to avoid using Lemma~\xref{lemma:WPS}, but would give a slightly
weaker result.
\end{remark}

Using a more explicit geometric approach, one can strengthen the assertion of
Proposition~\xref{proposition:double-cover}\xref{enumerate-(ii)}.

\begin{corollary}
In the assumptions of Proposition~\xref{proposition:double-cover}\xref{enumerate-(ii)}
one has~\mbox{$G\subset\GL_3(\Bbbk)$}.
\end{corollary}
\begin{proof}
The base locus of the linear system $|H|$ is a single point $P$ which is contained in
the smooth part of $X$ (see e.g. \cite[Theorem 0.6]{Shin1989}). Clearly,
the point $P$ is $\Aut(X)$-invariant.
Therefore, Lemma~\xref{lemma:Aut-P} implies that
$G\subset\GL_3(\Bbbk)$.
\end{proof}

\begin{lemma}\label{lemma:double-cover}
Let $X$ be a Fano threefold with Gorenstein terminal singularities.
Suppose that $\rho(X)=1$, and one of the following
possibilities holds:
\begin{enumerate}
\renewcommand\labelenumi{(\roman{enumi})}
\renewcommand\theenumi{(\roman{enumi})}
\item\label{lemma:double-cover-enumerate-(ii)}
$\iota(X)=2$ and $\dd(X)=1$;
\item\label{lemma:double-cover-enumerate-(iii)}
$\iota(X)=2$ and $\dd(X)=2$;
\item\label{lemma:double-cover-enumerate-(i)}
$\iota(X)=1$ and $\g(X)=2$;
\item\label{lemma:double-cover-enumerate-(iv)}
$\iota(X)=1$, $\g(X)=3$, and $X$ is a double cover of a three-dimensional quadric.
\end{enumerate}
Then the group $\Aut(X)$ is Jordan with
$\bar{J}\big(\Aut(X)\big)\le 960$.
\end{lemma}
\begin{proof}
Apply Proposition~\xref{proposition:double-cover} together with Lemma~\xref{lemma:weak-GL5}.
\end{proof}

\begin{lemma}\label{lemma:hypersurface}
Let $X\subset\P^4$ be a
hypersurface of degree at least~$2$.
Then the group $\Aut(X)$ is Jordan with
$\bar{J}\big(\Aut(X)\big)\le 960$.
\end{lemma}
\begin{proof}
There is an embedding
$\Aut(X)\subset\PGL_5(\Bbbk)$,
see e.g. \cite[Corollary~3.1.4]{Kuznetsov-Prokhorov-Shramov}. Thus the assertion follows
from Lemma~\xref{lemma:weak-GL5}.
\end{proof}

\subsection{Complete intersection of a quadric and a cubic}
\label{subsection:X23}

Now we will describe some properties of finite subgroups of automorphisms
of a complete intersection of a quadric and a cubic in~$\P^5$.

\begin{lemma}\label{lemma:X23}
Let $X\subset\P^5$ be a
Fano threefold with terminal Gorenstein
singularities
such that $\rho(X)=1$, $\iota(X)=1$, and $\g(X)=4$,
i.\,e. $X$ is a complete intersection of a quadric and a cubic
in $\P^5$ \textup(see \cite[Proposition~IV.1.4]{Iskovskikh-1980-Anticanonical},
\cite[Theorem 1.6 or Remark~4.2]{Przhiyalkovskij-Cheltsov-Shramov-2005en}\textup).
Let~\mbox{$Q\subset\P^5$} be the \textup(unique\textup) quadric passing through $X$.
Then one of the following possibilities occurs:
\begin{enumerate}
\item
the quadric $Q$ is smooth; in this case
there is a subgroup~\mbox{$\Aut'(X)\subset\Aut(X)$}
of index at most $2$ such that~\mbox{$\Aut'(X)\subset\PGL_4(\Bbbk)$};
\item
the quadric $Q$ is a cone with an isolated singularity; in this case
for any finite subgroup
$G\subset\Aut(X)$ there is an embedding
\begin{equation*}
G\subset\SO_5(\Bbbk)\times\Bbbk^*\subset\GL_5(\Bbbk);
\end{equation*}
\item the quadric $Q$ is a cone whose singular locus is a line; in this case
for any finite subgroup
$G\subset\Aut(X)$ there
is a subgroup $F\subset G$ of index~\mbox{$[G:F]\le 3$}
such that there is an embedding
\begin{equation*}
F\subset\Bbbk^*\times\left(\SO_4(\Bbbk)\times\Bbbk^*/\mumu_2\right)\subset\Bbbk^*\times\GL_4(\Bbbk).
\end{equation*}
\end{enumerate}
\end{lemma}
\begin{proof}
The embedding $X\hookrightarrow\P^5$ is given by the anticanonical linear system
on~$X$. Hence there is an action of the group $\Aut(X)$ on $\P^5$ that
agrees with the action of $\Aut(X)$ on~$X$, see e.g.~\cite[Lemma~3.1.2]{Kuznetsov-Prokhorov-Shramov}.
The quadric $Q$ is $\Aut(X)$-invariant, and
the action of~\mbox{$\Aut(X)$} on $Q$ is faithful.
Since the singularities of $X$ are terminal and thus isolated, we see that
the singular locus of $Q$ is at most one-dimensional.

Suppose that $Q$ is non-singular. Then $Q$ is isomorphic to
the Grassmannian $\Gr(2,4)$, so that
\begin{equation*}
\Aut(Q)\cong\PGL_4(\Bbbk)\rtimes\mumu_2,
\end{equation*}
which gives
case~(i).

Therefore, we may assume that $Q$ is singular.
Then $\Sing(Q)$ is a linear subspace of~$\P^5$
of dimension $\delta\le 1$.

Suppose that $\delta=0$, so that $\Sing(Q)$ is a single point $P$.
Then the point $P$ is $\Aut(Q)$-invariant,
and thus also $\Aut(X)$-invariant. Let $G\subset\Aut(X)$
be a finite subgroup.
By Lemma~\xref{lemma:Aut-P} there is an embedding
\begin{equation*}
G\subset\GL\big(T_P(Q)\big)=\GL\big(T_P(\P^5)\big)\cong\GL_5(\Bbbk).
\end{equation*}
Moreover, the group
$G$ acts by a character on a quadratic polynomial on $T_P(\P^5)$
that corresponds to the quadric $Q$. Hence $G$ is contained in
the subgroup
\begin{equation*}
\pi^{-1}\left(\PSO_5(\Bbbk)\right)\subset \GL_5(\Bbbk),
\end{equation*}
where $\pi\colon\GL_5(\Bbbk)\to\PGL_5(\Bbbk)$ is the natural projection.
This gives case~(ii).

Finally, suppose that $\delta=1$.
Let $L\cong\P^1$ be the vertex of $Q$. Then $L$ is $\Aut(Q)$-invariant,
and thus also $\Aut(X)$-invariant. Let $G\subset\Aut(X)$ be a finite
subgroup.
Note that $X\cap L$ is non-empty and consists of at most three points.
Hence there
is a subgroup~\mbox{$F\subset G$} of index $[G:F]\le 3$
such that $F$ has a fixed point on $L$. Denote this point
by $P$. By Lemma~\xref{lemma:Aut-P} there is an embedding
\begin{equation*}
F\hookrightarrow \GL\big(T_P(\P^5)\big)\cong\GL_5(\Bbbk).
\end{equation*}
Moreover, the representation of $F$ in
$T_P(\P^5)$ splits as a sum of a one-dimensional and a four-dimensional
representations since $F$ preserves the tangent
direction $T_P(L)$ to $L$. Put
\begin{equation*}
V=T_P(\P^5)/T_P(L).
\end{equation*}
Then there is an embedding
$F\hookrightarrow F_1\times F_2$, where $F_1$ is a finite cyclic
group, and $F_2$ is a finite subgroup of $\GL(V)\cong\GL_4(\Bbbk)$.
The last thing we need to observe is that
$F_2$ preserves a quadric cone in $\P(V)$ corresponding to an intersection of the tangent cone to~$Q$
at $P$ with the subspace $V\hookrightarrow T_P(\P^5)$.
Therefore, $F_2$ is contained in
the subgroup
\begin{equation*}
\pi^{-1}\left(\PSO_4(\Bbbk)\right)\subset \GL_4(\Bbbk),
\end{equation*}
where $\pi\colon\GL_4(\Bbbk)\to\PGL_4(\Bbbk)$ is the natural projection.
Since
\begin{equation*}
\pi^{-1}\left(\PSO_4(\Bbbk)\right)\cong \SO_4(\Bbbk)\times\Bbbk^*/\mumu_2,
\end{equation*}
this gives case~(iii) and completes the proof of the lemma.
\end{proof}

\begin{corollary}\label{corollary:X23}
Let $X$ be a Fano threefold with Gorenstein terminal singularities.
Suppose that $\rho(X)=1$, $\iota(X)=1$, and $\g(X)=4$.
Then the group $\Aut(X)$ is Jordan with
\begin{equation*}
\bar{J}\big(\Aut(X)\big)\le 1920.
\end{equation*}
\end{corollary}
\begin{proof}
By Lemma~\xref{lemma:X23} one of the following possibilities holds:
\begin{enumerate}
\item
there is a subgroup $\Aut'(X)\subset\Aut(X)$
of index at most $2$ such that~\mbox{$\Aut'(X)\subset\PGL_4(\Bbbk)$};
\item
for any finite subgroup
$G\subset\Aut(X)$ there is an embedding $G\subset\GL_5(\Bbbk)$;
\item for any finite subgroup
$G\subset\Aut(X)$ there
is a subgroup $F\subset G$ of index~\mbox{$[G:F]\le 3$}
such that there is an embedding
\begin{equation*}
F\subset \Bbbk^*\times\left(\SO_4(\Bbbk)\times\Bbbk^*/\mumu_2\right).
\end{equation*}
\end{enumerate}
In particular, the group $\Aut(X)$ is Jordan.
In case~(i) one has
\begin{equation*}
\bar{J}\big(\Aut(X)\big)\le 2\cdot\bar{J}\big(\Aut'(X)\big)\le
2\cdot\bar{J}\big(\PGL_4(\Bbbk)\big)=2\cdot 960=1920
\end{equation*}
by Lemma~\xref{lemma:weak-GL4}.
In case~(ii) one has
\begin{equation*}
\bar{J}\big(\Aut(X)\big)\le \bar{J}\big(\PGL_5(\Bbbk)\big)= 960
\end{equation*}
by Lemma~\xref{lemma:weak-GL5}.
In case~(iii) one has
\begin{equation*}
\bar{J}\big(\Aut(X)\big)\le 3\cdot\bar{J}\left(\Bbbk^*\times\left(\SO_4(\Bbbk)\times\Bbbk^*/\mumu_2\right)\right)=
3\cdot \bar{J}\big(\SO_4(\Bbbk)\big)\le
3\cdot 288=864
\end{equation*}
by Lemma~\xref{lemma:GL4-quadric}.
\end{proof}

\subsection{General case}

The results of \S\xref{subsection:low-degree}
and~\S\xref{subsection:X23} imply the following

\begin{corollary}\label{corollary:Fano-3-fold-small-g}
Let $X$ be a Fano threefold with Gorenstein terminal singularities.
Suppose that $\rho(X)=1$, $\iota(X)=1$, and $\g(X)\le 4$.
Then $\Aut(X)$ is Jordan with
\begin{equation*}
\bar{J}\big(\Aut(X)\big)\le 1920.
\end{equation*}
\end{corollary}
\begin{proof}
Recall that $\g(X)\ge 2$.
If $\g(X)=2$, then $\Aut(X)$ is Jordan with
$\bar{J}\big(\Aut(X)\big)\le 960$ by Lemma~\xref{lemma:double-cover}.
If $\g(X)=3$ and $-K_X$ is not very ample,
then $\Aut(X)$ is also Jordan with
$\bar{J}\big(\Aut(X)\big)\le 960$ by Lemma~\xref{lemma:double-cover}.
If $\g(X)=3$ and $-K_X$ is very ample,
then~$X$ is a smooth quartic in $\P^4$ (because $\dim |-K_X|=4$ and $-K_X^3=4$),
so that~\mbox{$\Aut(X)$} is Jordan with
$\bar{J}\big(\Aut(X)\big)\le 960$ by Lemma~\xref{lemma:hypersurface}.
Finally, if $\g(X)=4$, then the group~\mbox{$\Aut(X)$}
is Jordan with $\bar{J}\big(\Aut(X)\big)\le 1920$ by Corollary~\xref{corollary:X23}.
\end{proof}

Now we are ready to study automorphism groups of arbitrary \emph{singular}
Gorenstein $G$-Fano threefolds.

\begin{lemma}\label{lemma:Gorenstein-G-Fano-3-fold}
Let $G$ be a finite group, and let $X$ be a singular Gorenstein
$G$-Fano threefold.
Then the group $\Aut(X)$ is Jordan with
\begin{equation*}
\bar{J}\big(\Aut(X)\big)\le 9504.
\end{equation*}
\end{lemma}
\begin{proof}
Let $P_1,\ldots, P_N\in X$ be all singular points of $X$.
The group $\Aut(X)$ acts on the set $\{P_1,\dots,P_N\}$.
The subgroup $\Aut_{P_1}(X)\subset\Aut(X)$ stabilizing the point
$P_1$ has index
\begin{equation*}
[\Aut(X):\Aut_{P_1}(X)]\le N.
\end{equation*}
We have
\begin{equation*}
\bar{J}\big(\Aut(X)\big)\le
N\cdot\bar{J}\big(\Aut_{P_1}(X)\big).
\end{equation*}
According to~\cite{Namikawa-1997} there exists a \textit{smoothing} of $X$, that is
a one-parameter deformation
\begin{equation*}
\mathfrak{X}\to B\ni 0
\end{equation*}
such that
a general fiber $\mathfrak{X}_b$ is smooth and the central fiber~$\mathfrak{X}_0$
is isomorphic to~$X$.
One has
\begin{equation}\label{eq:Namikawa}
N\le 21-\frac 12 \chi_{\operatorname{top}}(\mathfrak{X}_b) = 20-\rho(\mathfrak{X}_b) + h^{1,2}(\mathfrak{X}_b)
\end{equation}
by \cite[Theorem~13]{Namikawa-1997}. Moreover, there is an identification
$\Pic(\mathfrak{X}_b)\cong\Pic(X)$, see~\mbox{\cite[Theorem~1.4]{Jahnke2011}}.

Suppose that $\rho(X)\ge 2$.
Smooth Fano threefolds $V$ whose Picard group admits an action of a finite group $G$ such
that $\rho(V)^G=1$ and $\rho(V)>1$ are classified in \cite{Prokhorov-GFano-2}.
Applying this classification to $V=\mathfrak{X}_b$
we obtain $h^{1,2}(\mathfrak{X}_b)\le 9$.

Suppose that $\rho(X)=1$.
If $\iota(X)=2$ and $\dd(X)\le 2$, then the group $\Aut(X)$
is Jordan with $\bar{J}\big(\Aut(X)\big)\le 960$
by Lemma~\xref{lemma:double-cover}.
If $\iota(X)=1$ and $\g(X)\le 4$, then $\Aut(X)$
is Jordan with $\bar{J}\big(\Aut(X)\big)\le 1920$ by
Corollary~\xref{corollary:Fano-3-fold-small-g}.
In all other cases by the classification of
smooth Fano threefolds (see~\cite[\S12.2]{Iskovskikh-Prokhorov-1999})
we have $h^{1,2}(\mathfrak{X}_b)\le 14$.

Therefore, we are left with several possibilities with $h^{1,2}(\mathfrak{X}_b)\le 14$.
In this case~\eqref{eq:Namikawa} implies that
$N\le 33$ (and in some cases this bound can be significantly improved, see~\cite{Prokhorov-factorial-Fano-e}).
Now Corollary~\xref{corollary:dim-3-terminal}(i)
implies that $\Aut(X)$ is Jordan with
\begin{equation*}
\bar{J}\big(\Aut(X)\big)\le 33\cdot 288=9504.\qedhere
\end{equation*}
\end{proof}

\section{Smooth Fano threefolds}
\label{section:smooth-Fano}

In this section we bound weak Jordan constants for automorphism
groups of smooth Fano threefolds.

\subsection{Complete intersections of quadrics}

It appears that
we can get a reasonable bound for a weak Jordan constant
of an automorphism group of a smooth complete intersection of two quadrics
of arbitrary dimension.
Here we will use the results of~\S\ref{subsection:int-two-quadrics-finite}.

\begin{lemma}\label{lemma:int-2-quadrics}
Let $X\subset\P^n$, $n\ge 4$, be a smooth complete intersection of $2$ quadrics.
Then the group $\Aut(X)$ is Jordan with
\begin{equation*}
\bar{J}\big(\Aut(X)\big)\le (n+1)!
\end{equation*}
\end{lemma}
\begin{proof}
By Proposition~\xref{proposition:int-2-quadrics} there is an exact sequence
\[
1\longrightarrow \Gamma\longrightarrow \Aut(X)\longrightarrow G_{\mathcal{P}}\longrightarrow 1
\]
where $\Gamma\cong\mumu_2^n$
and $G_{\mathcal{P}}\subset\SS_{n+1}$.
Therefore, the group $\Aut(X)$ is Jordan with
\begin{equation*}
\bar{J}\big(\Aut(X)\big)\le [\Aut(X):\Gamma]\le |\SS_{n+1}|=(n+1)!\qedhere
\end{equation*}
\end{proof}

In dimension $3$ we can also bound weak Jordan constants for automorphism
groups of smooth complete intersections of three
quadrics.

\begin{lemma}\label{lemma:intersection-of-three-quadrics}
Let $X\subset\P^6$ be a smooth complete intersection of $3$ quadrics.
Then
the group $\Aut(X)$ is Jordan with
\begin{equation*}
\bar{J}\big(\Aut(X)\big)\le 10368.
\end{equation*}
\end{lemma}
\begin{proof}
There is an exact sequence
\begin{equation*}
1\to \Gamma\to\Aut(X)\to\PGL_{3}(\Bbbk),
\end{equation*}
where $\Gamma\cong\mumu_2^m$ with $m\le 6$,
see~\eqref{eq:sequence-aut-quadrics} and
Corollary~\xref{corollary:gamma-finite}.
If $m\le 2$, then~\mbox{$\bar{J}\big(\Aut(X)\big)\le 2304$} by Lemma~\xref{lemma:4-PGL3}.
Therefore, we assume that $m\ge 3$.

Put
\begin{equation*}
V=H^0\big(\P^6, \mathcal{O}_{\P^6}(H)\big)^{\vee},
\end{equation*}
so that $\P^6$ is identified with $\P(V)$.
Since the anticanonical class of $X$ is linearly equivalent
to a hyperplane section of $X$ in $\P^6$, the group $\Aut(X)$ acts on~$V$,
see e.g.~\cite[Corollary~3.1.3]{Kuznetsov-Prokhorov-Shramov}. Thus we may assume that $\Aut(X)\subset\GL(V)$.

Let $-\Id\in\GL(V)$ be the scalar matrix $\diag(-1,\ldots,-1)$.
Let $\tilde{\Gamma}\subset\GL(V)$ be a group generated by $\Gamma$ and $-\Id$,
and let $G\subset\GL(V)$ be a group generated by $\Aut(X)$ and $-\Id$.
Since $\Aut(X)\subset\GL(V)$ acts faithfully on $\P(V)$ and thus does not contain scalar matrices,
we see that
\begin{equation*}
\tilde{\Gamma}\cong\mumu_2\times\Gamma\cong \mumu_2^{m'}
\end{equation*}
with $m'=m+1\ge 4$.
We conclude that $\Aut(X)$ is Jordan with
\begin{equation*}
\bar{J}\big(\Aut(X)\big)\le\bar{J}(G)\le 10368
\end{equation*}
by Lemma~\xref{lemma:isotypical-7}.
\end{proof}

\begin{remark}\label{remark:int-3-quadrics-nonrational}
Let $X\subset\P^6$ be a smooth complete intersection of $3$ quadrics.
Then $X$ is non-rational, see \cite[Theorem~5.6]{Beauville1977}.
Therefore, automorphism groups of varieties of this type cannot
provide examples of subgroups in~$\Cr_3(\Bbbk)$ whose Jordan constants
attain the bounds given by Theorem~\xref{theorem:constant}, cf. Remark~\xref{remark-P1-P1-P1} below.
\end{remark}

\subsection{Fano threefolds of genus $6$}
\label{subsection:Gushel}

Recall that a smooth Fano threefold $X$ with~\mbox{$\rho(X)=1$}, $\iota(X)=1$, and $\g(X)=6$
may be either an intersection of the Grassmannian~\mbox{$\Gr(2,5)\subset \P^9$}
with a quadric and two hyperplanes,
or a double cover of a smooth Fano threefold
\begin{equation*}
Y=\Gr(2,5)\cap \P^6\subset \P^9
\end{equation*}
with the branch divisor $B\in |-K_Y|$ (see \cite{Gushelcprime1982}).
We will refer to the former varieties as \emph{Fano
threefolds of genus $6$ of the first type}, and to the latter varieties as \emph{Fano
threefolds of genus $6$ of the second type}.

\begin{remark}
In~\cite{Debarre-Kuznetsov2015}
these were called ordinary and special varieties, respectively.
\end{remark}

\begin{lemma}[{cf.~\cite[Corollary~4.2]{DIM12}, \cite[Proposition~3.12]{Debarre-Kuznetsov2015}}]
\label{lemma:smooth-Fano-3-fold-g-6}
Let $X$ be a smooth Fano threefold with~\mbox{$\rho(X)=1$}, $\iota(X)=1$, and $\g(X)=6$.
If $X$ is of the first type,
then there is an embedding
\begin{equation*}
\Aut(X)\hookrightarrow\Aut\big(\Gr(2,5)\big) \cong\PGL_5(\Bbbk).
\end{equation*}
If $X$ is of the second type,
then there is a normal subgroup $\Gamma\subset\Aut(X)$ such that
$\Gamma\cong\mumu_2$ and there is an exact sequence
\begin{equation*}
1\to\Gamma\to\Aut(X)\to\PGL_2(\Bbbk).
\end{equation*}
\end{lemma}
\begin{proof}
By definition, we have a natural morphism $\gamma\colon X \to \Gr(2,5)$.
By~\cite[Theorem~2.9]{Debarre-Kuznetsov2015} the morphism $\gamma$ is functorial.
Note that $\gamma$ is completely determined by what is called GM data in~\cite{Debarre-Kuznetsov2015},
in particular it is equivariant with respect to the action of the group $\Aut(X)$.
Consider the corresponding map
\begin{equation*}
\theta \colon \Aut(X) \to \Aut(\Gr(2,5)) \cong \PGL_5(\Bbbk).
\end{equation*}

Suppose that $X$ is a Fano threefold of genus $6$ of the first type.
Then functoriality of~$\gamma$ implies that $\theta$ is an embedding.
This proves the first assertion of the lemma.

Now suppose that $X$ is a Fano threefold of genus $6$ of the second type.
Then the morphism $\gamma$ is a double cover, and its image is a Fano threefold $Y$
with $\rho(Y)=1$, $\iota(Y)=2$, and~\mbox{$\dd(Y)=5$},
see~\cite[Proposition~2.20]{Debarre-Kuznetsov2015}.
Let $\Gamma\subset\Aut(X)$ be the subgroup generated by the Galois
involution of the double cover~\mbox{$\gamma\colon X\to Y$}.
Then $\Gamma\cong\mumu_2$ is a normal subgroup
of $\Aut(X)$, and $\Aut(X)/\Gamma$ embeds into~\mbox{$\Aut(Y)$}.
On the other hand, one has $\Aut(Y)\cong\PGL_2(\Bbbk)$,
see e.g.~\cite[Proposition~4.4]{Mukai-1988}
or~\cite[Proposition~7.1.10]{CheltsovShramov2016}.
This gives the second assertion of the lemma.
\end{proof}

\begin{corollary}\label{corollary:genus-6}
Let $X$ be a smooth Fano threefold with
$\rho(X)=1$, $\iota(X)=1$ and~\mbox{$\g(X)=6$}.
Then the group $\Aut(X)$ is Jordan with
\begin{equation*}
\bar{J}\big(\Aut(X)\big)\le 960.
\end{equation*}
\end{corollary}
\begin{proof}
Suppose that $X$ is a Fano threefold of genus $6$ of the first type.
Then there is an embedding $\Aut(X)\subset\PGL_5(\Bbbk)$
by Lemma~\xref{lemma:smooth-Fano-3-fold-g-6}, so that $\Aut(X)$ is Jordan with~\mbox{$\bar{J}(\Aut(X))\le 960$}
by Lemma~\xref{lemma:weak-GL5}.

Now suppose that $X$ is a Fano threefold of genus $6$ of the second type.
Then there is an exact sequence
\begin{equation*}
1\to\Gamma\to\Aut(X)\to\PGL_2(\Bbbk)
\end{equation*}
by Lemma~\xref{lemma:smooth-Fano-3-fold-g-6}.
Therefore, $\Aut(X)$ is Jordan with~\mbox{$\bar{J}(G)\le 12$} by Lemma~\xref{lemma:2-PGL2}.
\end{proof}

\subsection{Large degree and index}

Now we consider the cases with large anticanonical degree
and large index.

\begin{lemma}\label{lemma:high-genus-Fanos-Jordan}
Let $X$ be a smooth Fano threefold with $\iota(X)=1$ and
$\g(X)\ge 7$. Then the group $\Aut(X)$ is Jordan with
\begin{enumerate}
\item $\bar{J}\big(\Aut(X)\big)\le 504$
if $\g(X)=7$;
\item $\bar{J}\big(\Aut(X)\big)\le 9922$
if $\g(X)=8$;
\item $\bar{J}\big(\Aut(X)\big)\le 2016$
if $\g(X)=9$;
\item $\bar{J}\big(\Aut(X)\big)\le 5760$
if $\g(X)=10$;
\item $\bar{J}\big(\Aut(X)\big)\le 40$
if $\g(X)=12$.
\end{enumerate}
\end{lemma}
\begin{proof}
Assertion~(i) follows from~\cite[Corollary~4.3.5(i)]{Kuznetsov-Prokhorov-Shramov}
and Remark~\xref{remark:elliptic}.
Assertion~(ii) follows from~\cite[Corollary~4.3.5(ii)]{Kuznetsov-Prokhorov-Shramov}
and Lemma~\xref{lemma:general-type-45}.
Assertion~(iii) follows from~\cite[Corollary~4.3.5(iii)]{Kuznetsov-Prokhorov-Shramov}
and Lemma~\xref{lemma:ruled-surface}.
Assertion~(iv) follows from~\cite[Corollary~4.3.5(iv)]{Kuznetsov-Prokhorov-Shramov}
and Lemma~\xref{lemma:abelian-surface}. Finally,
assertion~(v) follows from~\cite[Corollary~4.3.5(v)]{Kuznetsov-Prokhorov-Shramov}
and Lemma~\xref{lemma:weak-GL3}.
\end{proof}

\begin{lemma}\label{lemma:rho-1-iota-large}
Let $G$ be a finite group, and $X$ be a smooth Fano threefold.
Suppose that~\mbox{$\rho(X)=1$} and $\iota(X)>1$. Then the group $\Aut(X)$ is Jordan with
\begin{equation*}
\bar{J}\big(\Aut(X)\big)\le 960.
\end{equation*}
\end{lemma}
\begin{proof}
It is known that $\iota(X)\le 4$. Moreover, $\iota(X)= 4$
if and only if $X\cong \P^3$, and~\mbox{$\iota(X)=3$}
if and only if $X$ is a quadric in $\P^4$
(see e.\,g. \cite[3.1.15]{Iskovskikh-Prokhorov-1999}).
In the former case one has~\mbox{$\Aut(X)\cong\PGL_4(\Bbbk)$}, so that
the group $\Aut(X)$ is Jordan with
$\bar{J}\big(\Aut(X)\big)=960$ by Lemma~\xref{lemma:weak-GL4}.
In the latter case the group $\Aut(X)$ is Jordan with
$\bar{J}\big(\Aut(X)\big)\le 960$ by Lemma~\xref{lemma:hypersurface}.

Thus we may assume that $\iota(X)=2$. Recall that $1\le \dd(X)\le 5$ (see e.\,g. \cite[\S12.2]{Iskovskikh-Prokhorov-1999}).

If $\dd(X)=5$, then $X$ is isomorphic
to a linear section
of the Grassmannian~\mbox{$\operatorname{Gr}(2,5)\subset\P^9$} by a
subspace~\mbox{$\P^6\subset\P^9$}, see \cite[\S12.2]{Iskovskikh-Prokhorov-1999}.
In this case one has
\begin{equation*}
\Aut(X)\cong \PGL_2(\Bbbk),
\end{equation*}
see~\mbox{\cite[Proposition~4.4]{Mukai-1988}} or~\cite[Proposition~7.1.10]{CheltsovShramov2016}.
So, the group $\Aut(X)$ is Jordan with~\mbox{$\bar{J}\big(\Aut(X)\big)=12$} by Corollary~\xref{corollary:GL2}.

If $\dd(X)=4$, then
$X$ is a complete intersection of two quadrics in $\P^5$
(see \cite[\S12.2]{Iskovskikh-Prokhorov-1999}).
Thus $\Aut(X)$ is Jordan with
$\bar{J}\big(\Aut(X)\big)\le 720$ by Lemma~\xref{lemma:int-2-quadrics}.

If $\dd(X)=3$, then $X\cong X_3\subset \P^4$ is a cubic threefold
(see \cite[\S12.2]{Iskovskikh-Prokhorov-1999}).
Thus $\Aut(X)$ is Jordan with
$\bar{J}\big(\Aut(X)\big)\le 960$ by Lemma~\xref{lemma:hypersurface}.

Finally, if $\dd(X)=2$ or $\dd(X)=1$, then $\Aut(X)$ is Jordan with
$\bar{J}\big(\Aut(X)\big)\le 960$ by Lemma~\xref{lemma:double-cover}.
\end{proof}

\subsection{Large Picard rank}

Finally, we deal with smooth $G$-Fano threefolds with Picard rank
greater than~$1$. We denote by $W_6$ a smooth divisor of bidegree $(1,1)$ in~\mbox{$\P^2\times \P^2$}.
Clearly, $W_6$ is a Fano threefold with $\iota(W_6)=2$ and $\rho(W_6)=2$.

\begin{lemma}\label{lemma:rho-large}
Let $G$ be a finite group, and $X$ be a smooth $G$-Fano threefold.
Suppose that~\mbox{$\rho(X)>1$}. Then $\Aut(X)$ is Jordan with
$\bar{J}\big(\Aut(X)\big)\le 10368$.
\end{lemma}
\begin{proof}
By \cite{Prokhorov-GFano-2} we have the following possibilities.

\begin{enumerate}
\item $\rho(X)=2$, $\iota(X)=2$, and $X\cong W_6$;
\item $\rho(X)=3$, $\iota(X)=2$, and $X\cong\P^1\times\P^1\times\P^1$;
\item $\rho(X)=2$, $\iota(X)=1$, and $X\subset \P^2\times \P^2$ is a divisor of
bidegree $(2,2)$;
\item $\rho(X)=2$, $\iota(X)=1$, and $X$ is a double cover of $W_6$
whose branch divisor $S\subset W_6$ is a member of the linear system~\mbox{$|-K_{W_6}|$};
\item $\rho(X)=2$, $\iota(X)=1$, and $X$ is the blow up of $\P^3$ along a curve $C\subset \P^3$
of degree~$6$ and genus $3$;
\item $\rho(X)=2$, $\iota(X)=1$, and $X$ is the blow up of a smooth quadric $Q\subset \P^4$
along a rational twisted quartic curve
$C\subset Q$;
\item $\rho(X)=3$, $\iota(X)=1$, and $X$ is a double cover of $\P^1\times \P^1\times \P^1$
whose branch divisor~$S$ is a member of the linear system~\mbox{$|-K_{\P^1\times \P^1\times \P^1}|$};
\item $\rho(X)=3$, $\iota(X)=1$, and $X$ is the blow up of $W_6$ along a rational
curve $C\subset W_6$ of bidegree $(2,2)$;
\item $\rho(X)=4$, $\iota(X)=1$, and $X\subset \P^1\times \P^1\times \P^1\times \P^1$ is a divisor of
multi\-degree~\mbox{$(1,1,1,1)$}; in this case each of four projections
$\pi_i\colon X\to \P^1\times \P^1\times \P^1$ is the blow up along
an elliptic curve $C$ which is an intersection of two
members of the linear system~\mbox{$|-\frac{1}{2}K_{\P^1\times \P^1\times \P^1}|$}.
\end{enumerate}

In case (i) one has
\begin{equation*}
\Aut(X)\cong\PGL_3(\Bbbk)\rtimes\mumu_2,
\end{equation*}
so that $\Aut(X)$ is Jordan with
\begin{equation*}
\bar{J}\big(\Aut(X)\big)\le |\mumu_2|\cdot\bar{J}\big(\PGL_3(\Bbbk)\big)=2\cdot 40=80
\end{equation*}
by Lemma~\xref{lemma:weak-GL3}.

In case (ii) one has
\begin{equation*}
\Aut(X)\cong\big(\PGL_2(\Bbbk)\times\PGL_2(\Bbbk)\times\PGL_2(\Bbbk)\big)
\rtimes\SS_3,
\end{equation*}
so that $\Aut(X)$ is Jordan with
\begin{equation*}
\bar{J}\big(\Aut(X)\big)\le |\SS_3|\cdot \bar{J}\big(\PGL_2(\Bbbk)\big)=6\cdot 12^3=10368
\end{equation*}
by Corollary~\xref{corollary:GL2}.

In case (iii) one has $\rho(X)=2$, so that the projections
\begin{equation*}
p_i\colon X \hookrightarrow \P^2\times\P^2\to \P^2,\quad i=1,2,
\end{equation*}
are all possible Mori contractions from $X$.
Hence
the action of $\Aut(X)$ on $X$ lifts to the action on $\P^2\times\P^2$
and the embedding
\begin{equation*}
p_1\times p_2\colon X \hookrightarrow \P^2\times\P^2
\end{equation*}
is $\Aut(X)$-equivariant. Thus
\begin{equation*}
\Aut(X)\subset\Aut(\P^2\times\P^2)\cong\big(\PGL_3(\Bbbk)\times\PGL_3(\Bbbk)\big)\rtimes\mumu_2,
\end{equation*}
so that $\Aut(X)$ is Jordan with
\begin{equation*}
\bar{J}\big(\Aut(X)\big)\le 2\cdot\bar J\big(\PGL_3(\Bbbk)\big)^2=2\cdot 40^2=3200
\end{equation*}
by Lemma~\xref{lemma:weak-GL3}.

In case (iv) one has $\rho(X)=2$, so that
two conic bundles
\begin{equation*}
\pi_i\colon X\to W_6\to\P^2,\quad i=1,2,
\end{equation*}
are
all possible Mori contractions from $X$. Thus there is a
subgroup~\mbox{$\Aut'(X)\subset\Aut(X)$} of index at most $2$ such that
the conic bundle $\pi_1\colon X\to\P^2$ is $\Aut'(X)$-equivariant.
Let~\mbox{$G\subset\Aut'(X)$} be a finite subgroup. Then one has
\begin{equation*}
\rho(X/\P^2)^{G}=\rho(X/\P^2)=1,
\end{equation*}
so that
$\pi_1\colon X\to\P^2$ is a $G$-equivariant conic bundle. Thus $\Aut(X)$ is Jordan
with
\begin{equation*}
\bar{J}\big(\Aut(X)\big)\le [\Aut(X):\Aut'(X)]\cdot \bar{J}\big(\Aut'(X)\big)\le 2\cdot 3456=
6912
\end{equation*}
by Lemma~\xref{lemma:conic-bundle}.

In case (v)
one has $\rho(X)=2$, so that the contraction~\mbox{$\pi\colon X\to\P^3$}
is one of the two possible Mori contractions from $X$. Hence
there is a subgroup $\Aut'(X)$ of index at most~$2$ such that
$\pi$ is $\Aut'(X)$-equivariant. In particular,
$\Aut'(X)$ acts on $\P^3$ faithfully, and since the curve $C\subset\P^3$
is not contained in any plane, $\Aut'(X)$ acts
faithfully on $C$ as well. Therefore,~\mbox{$\Aut(X)$} is Jordan with
\begin{equation*}
\bar{J}\big(\Aut(X)\big)\le [\Aut(X):\Aut'(X)]\cdot\bar{J}\big(\Aut'(X)\big)\le
2\cdot\bar{J}\big(\Aut(C)\big)\le 2\cdot 168=336
\end{equation*}
by Remark~\xref{remark:elliptic}.

In case (vi)
one has $\rho(X)=2$, so that the contraction~\mbox{$\pi\colon X\to Q$}
is one of the two possible Mori contractions from $X$. Hence
there is a subgroup $\Aut'(X)$ of index at most~$2$ such that
$\pi$ is $\Aut'(X)$-equivariant. In particular,
$\Aut'(X)$ acts on $Q$ faithfully. Since all automorphisms of $Q$ are linear,
and the curve $C\subset Q\subset\P^4$
is not contained in any hyperplane, $\Aut'(X)$ acts
faithfully on $C$ as well.
Therefore,~\mbox{$\Aut(X)$} is Jordan with
\begin{equation*}
\bar{J}\big(\Aut(X)\big)\le [\Aut(X):\Aut'(X)]\cdot\bar{J}\big(\Aut'(X)\big)\le
2\cdot\bar{J}\big(\PGL_2(\Bbbk)\big)=24
\end{equation*}
by Corollary~\xref{corollary:GL2}.

In case (vii) one has $\rho(X)=3$, and the map $X\to \P^1\times\P^1\times\P^1\hookrightarrow \P^8$
is given by the anticanonical linear system.
Three projections $\P^1\times\P^1\times\P^1\to\P^1\times\P^1$ give us three conic bundle
structures
\begin{equation*}
\pi_i\colon X\to\P^1\times\P^1\times\P^1\to\P^1\times\P^1,\quad i=1,2,3,
\end{equation*}
on $X$ and these projections are permuted by the automorphism group
$\Aut(X)$, because the morphism $X\to\P^1\times\P^1\times\P^1$ is $\Aut(X)$-equivariant.
Thus there is a subgroup~\mbox{$\Aut'(X)\subset\Aut(X)$} of index at most $3$ such that
the conic bundle~\mbox{$\pi_1\colon X\to\P^1\times\P^1$} is $\Aut'(X)$-equivariant.
Let $G\subset\Aut'(X)$ be a finite subgroup. Then one has
\begin{equation*}
\rho(X/\P^1\times\P^1)^{G}=\rho(X/\P^1\times\P^1)=1,
\end{equation*}
so that
$\pi_1\colon X\to\P^1\times\P^1$ is a $G$-equivariant conic bundle. Thus $\Aut(X)$ is Jordan
with
\begin{equation*}
\bar{J}\big(\Aut(X)\big)\le [\Aut(X):\Aut'(X)]\cdot \bar{J}\big(\Aut'(X)\big)\le 3\cdot 3456=
10368
\end{equation*}
by Lemma~\xref{lemma:conic-bundle}.

In case (viii)
one has $\rho(X)=3$, and three divisorial contractions
\begin{equation*}
\pi_i\colon X\to W_6,\quad i=1,2,3,
\end{equation*}
are all possible birational
Mori contractions
from $X$ (see~\cite[Table~3, no.~13]{Mori1981-82}).
Thus there is a subgroup
$\Aut'(X)$ of index at most $3$ such that
$\pi_1$ is $\Aut'(X)$-equivariant. In particular,~\mbox{$\Aut'(X)$} acts on $W_6$ faithfully.
The morphism $\pi_1$ is a blow up
of a rational curve~\mbox{$C_1\subset W_6$} of bi-degree $(2,2)$.
Since the images of $C_1$ under both projections
\begin{equation*}
C_1\hookrightarrow W_6\to \P^2
\end{equation*}
span~$\P^2$, we see that $\Aut'(X)$ acts on $C_1$ faithfully as well.
Therefore,~\mbox{$\Aut(X)$} is Jordan with
\begin{equation*}
\bar{J}\big(\Aut(X)\big)\le [\Aut(X):\Aut'(X)]\cdot\bar{J}\big(\Aut'(X)\big)\le
3\cdot\bar{J}\big(\PGL_2(\Bbbk)\big)=36
\end{equation*}
by Corollary~\xref{corollary:GL2}.

In case (ix)
one has $\rho(X)=4$, and four projections
\begin{equation*}
\pi_i\colon X\hookrightarrow\P^1\times\P^1\times\P^1\times\P^1\to\P^1\times\P^1\times\P^1,
\quad i=1,2,3,4,
\end{equation*}
are all possible birational Mori contractions from $X$ (see~\cite[Table~4, no.~1]{Mori1981-82}).
Thus there is a subgroup
$\Aut'(X)\subset\Aut(X)$ of index at most $4$ such that
the divisorial contraction
\begin{equation*}
\pi_1\colon X\to\P^1\times\P^1\times\P^1
\end{equation*}
is $\Aut'(X)$-equivariant.
The morphism $\pi_1$ is a blow up
of an elliptic curve
\begin{equation*}
C_1\subset\P^1\times\P^1\times\P^1
\end{equation*}
of tri-degree $(1,1,1)$.
Since all three projections
\begin{equation*}
C_1\hookrightarrow \P^1\times\P^1\times\P^1\to \P^1
\end{equation*}
are dominant, one can see that $\Aut'(X)$ acts on $C_1$ faithfully as well.
Therefore,~\mbox{$\Aut(X)$} is Jordan with
\begin{equation*}
\bar{J}\big(\Aut(X)\big)\le [\Aut(X):\Aut'(X)]\cdot\bar{J}\big(\Aut'(X)\big)\le
4\cdot\bar{J}\big(\Aut(C)\big)\le 24
\end{equation*}
by Remark~\xref{remark:elliptic}.
\end{proof}

\begin{remark}[{cf. Remark~\xref{remark:int-3-quadrics-nonrational}}]
Let $X$ be a smooth $G$-Fano threefold with~\mbox{$\rho(X)>1$}, and assume the notation of
the proof of Lemma~\xref{lemma:rho-large}. Then one has $\bar{J}\big(\Aut(X)\big)<10368$
with an exception of case~(ii), and with a possible exception
of case~(vii). However, if $X$ is like in case~(vii), then
it is non-rational, see~\cite{Alzati1992}.
Therefore, automorphism groups of varieties of this type cannot
provide examples of subgroups in~$\Cr_3(\Bbbk)$ whose Jordan constants
attain the bounds given by Theorem~\xref{theorem:constant}, cf. Remark~\xref{remark-P1-P1-P1} below.
\end{remark}

\begin{remark}
In general, studying Fano varieties with large automorphism
groups is an interesting problem on its own. In many cases such varieties exhibit
intriguing birational properties, see e.g.~\cite{CheltsovShramov11},
\cite{CheltsovShramov2016}, \cite{PrzyjalkowskiShramov2016}.
\end{remark}

\section{Proof of the main theorem}
\label{section:proof}

In this section we complete the proof of Theorem~\xref{theorem:constant}.

\subsection{Summary for Fano threefolds}

We summarize the results of \S\xref{section:Fano}
and~\S\xref{section:smooth-Fano} as follows.

\begin{proposition}\label{proposition:Gorenstein-Fano-Jordan}
Let $X$ be a Fano threefold with terminal Gorenstein singularities.
Suppose that $\rho(X)=1$.
Then the group $\Aut(X)$ is Jordan with
\begin{equation*}
\bar{J}\big(\Aut(X)\big)\le 10368.
\end{equation*}
\end{proposition}
\begin{proof}
If $X$ is singular, the group $\Aut(X)$ is Jordan with
\begin{equation*}
\bar{J}\big(\Aut(X)\big)\le 9504
\end{equation*}
by Lemma~\xref{lemma:Gorenstein-G-Fano-3-fold}.
Therefore, we assume that $X$ is smooth.
If $\iota(X)>1$, then the group~\mbox{$\Aut(X)$} is Jordan with
\begin{equation*}
\bar{J}\big(\Aut(X)\big)\le 960
\end{equation*}
by Lemma~\xref{lemma:rho-1-iota-large}.

It remains to consider the case when $X$ is a smooth Fano threefold
with~\mbox{$\Pic(X)=\Z\cdot K_X$}.
According to the classification
(see e.\,g. \cite[\S12.2]{Iskovskikh-Prokhorov-1999}),
one has either $2\le \g(X) \le 10$, or $\g(X)=12$.
If $\g(X)\le 4$, then the group $\Aut(X)$ is Jordan with
\begin{equation*}
\bar{J}\big(\Aut(X)\big)\le 1920
\end{equation*}
by Corollary~\xref{corollary:Fano-3-fold-small-g}.
If $\g(X)=5$, then the variety $X$ is an intersection
of three quadrics in~$\P^6$ (see \cite[\S12.2]{Iskovskikh-Prokhorov-1999}),
so that the group $\Aut(X)$ is Jordan with
\begin{equation*}
\bar{J}\big(\Aut(X)\big)\le 10368
\end{equation*}
by Lemma~\xref{lemma:intersection-of-three-quadrics}.
If $\g(X)=6$, then the group $\Aut(X)$ is Jordan with
\begin{equation*}
\bar{J}\big(\Aut(X)\big)\le 960
\end{equation*}
by Corollary~\xref{corollary:genus-6}.
Finally, if $\g(X)\ge 7$, then the group $\Aut(X)$ is Jordan with
\begin{equation*}
\bar{J}\big(\Aut(X)\big)\le 9922
\end{equation*}
by Lemma~\xref{lemma:high-genus-Fanos-Jordan}.
\end{proof}

\begin{corollary}\label{corollary:G-Fano-Jordan}
Let $G$ be a finite group, and
$X$ be a $G\Q$-Fano threefold.
Then the group~\mbox{$\Aut(X)$} is Jordan with
\begin{equation*}
\bar{J}\big(G\big)\le 10368.
\end{equation*}
\end{corollary}
\begin{proof}
If $X$ has a non-Gorenstein singular point, then the group $\Aut(X)$ is Jordan with
\begin{equation*}
\bar{J}\big(\Aut(X)\big)\le 4608
\end{equation*}
by Lemma~\ref{lemma:non-Gorenstein-Fano-3-fold}. Therefore, we may assume that $X$ is a (Gorenstein) $G$-Fano threefold.
If $X$ is singular, then the group $\Aut(X)$ is Jordan with
\begin{equation*}
\bar{J}\big(\Aut(X)\big)\le 9504
\end{equation*}
by Lemma~\xref{lemma:Gorenstein-G-Fano-3-fold}.
If $X$ is smooth and $\rho(X)>1$, then
the group $\Aut(X)$ is Jordan with
\begin{equation*}
\bar{J}\big(\Aut(X)\big)\le 10368
\end{equation*}
by Lemma~\xref{lemma:rho-large}.
Therefore, we may assume that $\rho(X)=1$, so that
the group $\Aut(X)$ is Jordan with
\begin{equation*}
\bar{J}\big(\Aut(X)\big)\le 10368
\end{equation*}
by Proposition~\xref{proposition:Gorenstein-Fano-Jordan}.
\end{proof}

\begin{remark}[{cf. Remark~\xref{remark:pohuj}}]
In several cases
one can produce better bounds for weak Jordan constants
of certain Fano threefolds applying a bit more effort.
We did not pursue this goal since the current estimates are already
enough to prove our main results.
\end{remark}

\subsection{Proof and concluding remarks}

Now we are ready to prove Theorem~\xref{theorem:constant}.

\begin{proof}[{Proof of Theorem~\xref{theorem:constant}}]
Let $X$ be a rationally connected threefold over an arbitrary field~$\Bbbk$
of characteristic~$0$,
and let $G\subset\Bir(X)$ be a finite group. It is enough to establish the upper
bounds for $\bar{J}(G)$ and~\mbox{$J(G)$}. Moreover, to prove the bounds we may assume
that~$\Bbbk$ is algebraically closed.

Regularizing the action of $G$ and taking an equivariant
desingularization (see e.\,g.~\mbox{\cite[Lemma-Definition~3.1]{Prokhorov-Shramov-J}}),
we may assume that $X$ is smooth and~\mbox{$G\subset\Aut(X)$}.
Applying $G$-equivariant Minimal Model Program to $X$
(which is possible due to an equivariant version of~\cite[Corollary 1.3.3]{BCHM} and
\cite[Theorem~1]{MiyaokaMori}, since rational connectedness
implies uniruledness), we may assume that either there is a $G\Q$-conic bundle
structure~\mbox{$\phi\colon X\to S$} for some rational surface $S$, or there is a
$G\Q$-del Pezzo fibration~\mbox{$\phi\colon X\to\P^1$}, or~$X$ is a $G\Q$-Fano threefold.
Therefore, we have
\begin{equation*}
\bar{J}(G)\le 10368
\end{equation*}
by Lemmas~\xref{lemma:conic-bundle} and~\xref{lemma:dP-fibration}
and Corollary~\xref{corollary:G-Fano-Jordan}.
Applying Remark~\xref{remark:Pyber}, we obtain the inequality
\begin{equation*}
J(G)\le 10368^2=107495424.
\end{equation*}

If $\Bbbk$ is algebraically closed, then the group
$\Cr_3(\Bbbk)$ contains a group
\[
\Aut\big(\P^1\times\P^1\times\P^1\big)\supset
\big(\A_5\times\A_5\times\A_5\big)\rtimes\SS_3,
\]
and the largest abelian subgroup of the latter finite group
has order $125$.
Therefore, one has
\begin{equation*}
\bar{J}\big(\Cr_3(\Bbbk)\big)=10368.
\end{equation*}
\end{proof}

\begin{remark}\label{remark-P1-P1-P1}
We do not know whether the bound for the
(usual) Jordan constant for the group $\Cr_3(\Bbbk)$ over an algebraically closed field~$\Bbbk$
of characteristic~$0$ provided by
Theorem~\xref{theorem:constant} is sharp or not. The Jordan constant
of the group~\mbox{$\Aut(\P^1\times\P^1\times\P^1)$} is smaller than that,
but there may be other automorphism groups of rational varieties
providing this value, cf. Lemma~\xref{lemma:dP-fibration}.
We also do not know the actual value of $J\big(\Cr_2(\Bbbk)\big)$, but we
believe that it can be found by a thorough (and maybe a little bit boring)
analysis of automorphism groups of del Pezzo surfaces and two-dimensional conic bundles,
since in dimension~$2$ much more precise classification results are available.
\end{remark}

\begin{remark}\label{remark:dim-4-fail}
In dimension $4$ and higher we cannot hope
(at least on our current level of
understanding the problem) to obtain results
similar to Theorem~\xref{theorem:constant}.
The first reason is that in dimension~$3$
we have a partial classification of Fano varieties, which gives a
much more detailed information than the boundedness proved in~\cite{KMMT-2000}
and~\cite{Birkar}; this gives us a possibility to (more or less)
establish an alternative proof of Theorem~\ref{theorem:RC-Jordan}
by repeating the same steps as in~\cite{ProkhorovShramov-RC}
and using this information instead of boundedness.
Another (and actually more serious)
reason is that
we use a classification
of three-dimensional terminal singularities to obtain bounds for Jordan constants
of automorphism groups of terminal Fano varieties and Mori fiber spaces.
The result of~\cite[Theorem~1]{Kollar11} shows that a ``nice'' classification
of higher dimensional terminal singularities is impossible,
at least in the setup we used in Lemma~\xref{lemma:dim-3-terminal} and
Corollary~\xref{corollary:dim-3-terminal}, due to unboundedness
of the dimensions of Zariski tangent spaces of their index one
covers.
\end{remark}

\appendix

\section{Automorphisms of some complete intersections}
\label{section:appendix}

In this section we collect some (well-known) results about automorphisms
of complete intersections of quadrics,
and complete intersections in weighted projective spaces.

\subsection{Complete intersections of quadrics}
\label{subsection:int-two-quadrics-finite}

Let $X \subset \P^n = \P(V)$ be a smooth complete intersection of $r$ quadrics.
Let $I_X$ be the ideal sheaf of $X$, so that
\begin{equation*}
W = H^0(\P(V),I_X(2))
\end{equation*}
is the $r$-dimensional vector space of quadrics passing through $X$. Let
\begin{equation*}
q\colon W \hookrightarrow \Sym^2V^\vee
\end{equation*}
be the natural embedding.

\begin{lemma}\label{lemma:CIQ-basic}
Suppose that $n\ge 2r$.
Then any automorphism of $X$ is induced by an automorphism of $\P(V)$ and
induces an automorphism of $\P(W)$.
\end{lemma}
\begin{proof}
By adjunction formula one has $-K_X\sim (n+1-2r)H$,
where $H$ is the class of a hyperplane section of~$X$.
Thus $X$ is Fano, and in particular there is no torsion in
the Picard group of~$X$.
Therefore, the class of $H$ in $\Pic(X)$ is $\Aut(X)$-invariant, and
there is a natural embedding
\begin{equation*}
\Aut(X) \hookrightarrow \PGL(V).
\end{equation*}
Furthermore, the twisted ideal sheaf $I_X(2)$ is invariant, hence the subspace
\begin{equation*}
\P(W) \subset \P(\Sym^2V^\vee)
\end{equation*}
is invariant
under the action of $\Aut(X)$, and so we also have a map
\begin{equation*}
\Aut(X) \to \PGL(W).\qedhere
\end{equation*}
\end{proof}

In the remaining part of this section
we denote by $\Aut_W(X)$ the image of the morphism $\Aut(X) \to \PGL(W)$ constructed in Lemma~\xref{lemma:CIQ-basic},
and by $\Gamma(X)$ its kernel. Thus we have an exact sequence
\begin{equation}\label{eq:sequence-aut-quadrics}
1 \to \Gamma(X) \to \Aut(X) \to \Aut_W(X) \to 1.
\end{equation}
Note that the group $\Gamma(X)$ is the subgroup of $\Aut(X)$ which preserves every quadric in the linear system
of quadrics passing through $X$.
In what follows we discuss what one can say about the groups $\Gamma(X)$ and $\Aut_W(X)$ in some special cases.

First, if $r = 1$ then $W = \Bbbk$, so $\Aut_W(X) = 1$ and
\begin{equation*}
\Gamma(X) = \Aut(X) \cong \PSO_{n+1}(\Bbbk).
\end{equation*}

Now assume that $r = 2$.

\begin{proposition}\label{proposition:int-2-quadrics}
Let $X\subset\P^n$, $n\ge 4$,
be a smooth complete intersection of two quadrics. Then $\Gamma(X) \cong \mumu_2^n$ and
\begin{equation*}
\Aut_W(X) \subset \PGL_2(\Bbbk).
\end{equation*}
Moreover, there is also an embedding
\begin{equation*}
\Aut_W(X) \subset \SS_{n+1}.
\end{equation*}
\end{proposition}
\begin{proof}
In some homogeneous coordinates $x_0,\ldots,x_n$ on $\P^n$
the variety $X$ is given by equations $f_1=f_2=0$, where
\begin{equation}\label{eq:two-quadrics}
f_1=x_0^2+\ldots+x_n^2,
\quad
f_2=\lambda_0 x_0^2+\ldots+\lambda_n x_n^2
\end{equation}
for some pairwise distinct numbers $\lambda_i$
(see e.\,g.~\cite[Proposition~2.1]{Reid1972}).
It is easy to see that the singular quadrics in the pencil generated by $f_1$ and $f_2$
are given by equations
\begin{equation*}
f_2-\lambda_if_1=0,
\end{equation*}
and their singular loci are
the points~\mbox{$(1:0:\ldots:0), \ldots, (0:\ldots:0:1)$}.

Since the subgroup $\Gamma(X) \subset \Aut(X)$ preserves every quadric in the pencil, it also preserves
the singular loci of the singular quadrics. Therefore, it is a subgroup in the standard torus
(formed by the diagonal matrices) in $\PGL(V)$. Since it also fixes the quadric $f_1 = 0$, it follows
that all diagonal entries of the matrix that represents an element of $\Gamma(X)$ have the same square. So, rescaling the matrices if necessary, we may assume that all diagonal entries are $\pm 1$. Therefore,
one has
\begin{equation}\label{eq:Gamma-mu-2-n}
\Gamma(X) \cong \mumu_2^{n+1} / \mumu_2 \cong \mumu_2^n,
\end{equation}
because we have to quotient out by the transformation $\pm \operatorname{Id}_V$
acting trivially on~$\P(W)$.

Since $\dim(W) = 2$, we have $\Aut_W(X) \subset \PGL(W)\cong \PGL_2(\Bbbk)$.
On the other hand, the group $\Aut_W(X)$ permutes the points of $\P(W)$ corresponding to singular quadrics,
hence there is a homomorphism $\Aut_W(X) \to \SS_{n+1}$.
It is an embedding since
any automorphism of the projective line~\mbox{$\P(W) \cong \P^1$} that preserves~\mbox{$n+1\ge 5$} points is trivial.
\end{proof}

\begin{remark}\label{remark:Aut-on-JX}
Assume the notation of Proposition~\xref{proposition:int-2-quadrics}.
Then the group $\Gamma(X)$ is generated by the $n+1$ reflections $\gamma_i\colon x_i \mapsto -x_i$
which satisfy the obvious relation
\begin{equation*}
\gamma_1 \circ \cdots\circ \gamma_{n+1}=1,
\end{equation*}
see~\eqref{eq:Gamma-mu-2-n}.

Now consider the case $n=5$, i.e. $X\subset \P^5$ is a smooth three-dimensional
intersection of two quadrics.
Let~$B$ be the corresponding hyperelliptic curve, see \cite[Proposition~2.1]{Reid1972},
and also \cite[Remark~2.2.11]{Kuznetsov-Prokhorov-Shramov}.
Let $\Sigma(X)$ be the Hilbert scheme of lines on~$X$ (see~\cite[\S2.1]{Kuznetsov-Prokhorov-Shramov}).
Then~\mbox{$\Sigma(X)$} is an abelian surface isomorphic to the Jacobian of~$B$,
and also isomorphic to the intermediate Jacobian of~$X$ (see~\cite[\S6]{Griffiths-Harris-1978}).
The group~\mbox{$\Gamma(X)$} consists of elements $\gamma$ of the types listed in Table~\xref{table:Fix}.
Here by~\mbox{$\Fix(\gamma,Z)$} we denote the locus of fixed points of
an automorphism $\gamma$ of a variety~$Z$.
\begin{table}[h]
\caption{Elements of $\Gamma(X)$}\label{table:Fix}
\begin{tabular}{|l|c|l|l|}
\hline
$\gamma$ & $\#$ & $\Fix(\gamma, X)$ &$\Fix(\gamma, \Sigma(X))$
\\[5pt]
\hline
\hline
$1$ & $1$ & $X$ & $\Sigma(X)$
\\
\hline
$\gamma_i$ & $6$ & del Pezzo surface of degree $4$ & $16$ points
\\
\hline
$\gamma_i\circ \gamma_j$, $i\neq j$ & $15$ & elliptic curve & $\varnothing$
\\
\hline
\mbox{$\gamma_i\circ \gamma_j\circ \gamma_k$},\ $i\neq j\neq k\neq i$ & $10$ & $8$ points & $16$ points\\
\hline
\end{tabular}
\end{table}

Put
\begin{equation*}
\Gamma_0(X)=\{1,\, \gamma_i\circ \gamma_j \mid i\neq j \}.
\end{equation*}
Then $\Gamma_0(X)$ is a subgroup of index $2$ in $\Gamma(X)$.
Let $\sigma$ be the homomorphism
\begin{equation*}
\sigma\colon \Gamma(X)\to \Gamma(X)/\Gamma_0(X)\cong \{\pm 1\}\cong\mumu_2.
\end{equation*}
Since $\Sigma(X)$ is an abelian surface, one has $H^2(\Sigma(X),\Z)\cong\Lambda^2H^1(\Sigma(X),\Z)$.
We see from the topological Lefschetz fixed point formula and Table~\xref{table:Fix} that
$\Gamma(X)$ acts on~\mbox{$H^3(X,\Z)$} and $H^1(\Sigma(X),\Z)$ via
\begin{equation*}
\gamma (x)=\sigma(\gamma)\cdot x.
\end{equation*}
Thus we have the following commutative diagram
\[
\xymatrix{
\Aut(X)\ar@{^{(}->}[rrr]\ar[dd]^{\nu_1}\ar[dr]^{\nu} &&&\Aut(\Sigma(X)) \ar[dd]^{\nu_2}
\\
&\Aut(B)\ar[r]^-{\cong} & \Aut(\Jac(B))\ar[dr]^{\cong}&
\\
\Aut(J(X))\ar[rrr]^{\cong}&&& \Aut(\operatorname{Alb}(\Sigma(X)))
}
\]
where $J(X)$ is the intermediate Jacobian of $X$.
The elements of the kernel of
\begin{equation*}
\nu\colon \Aut(X)\to \Aut(B)
\end{equation*}
act trivially on our pencil of quadrics, and each $\gamma_i$ switches
two families of planes on smooth quadrics.
Therefore, one has $\Ker \nu= \Gamma_0(X)$, and $\nu(\gamma_i)$ is the
hyperelliptic involution of the curve~$B$. Furthermore, $\Ker\nu_2$ is the subgroup of translations and
\begin{equation*}
\Aut(X)\cap \Ker\nu_2=\Ker\nu=\Gamma_0(X).
\end{equation*}
This shows
that the kernel of the homomorphism
\begin{equation*}
\nu_1\colon \Aut(X)\to \Aut(J(X))
\end{equation*}
coincides with $\Gamma_0(X)$.
See also \cite{XPan-2015}, \cite{ChenPanZhang-2015}.
\end{remark}

An easy consequence of Proposition~\xref{proposition:int-2-quadrics} is the following result.

\begin{corollary}\label{corollary:gamma-finite}
For any $2 \le r \le n/2$ we have $\Gamma(X) \cong \mumu_2^m$ for some $m \le n$.
\end{corollary}

\begin{proof}
Let $Q$ and $Q'$ be two general quadrics passing through $X$.
Then $Y=Q\cap Q'$ is a complete intersection of $Q$ and $Q'$. Moreover,
$Y$ is non-singular outside $X$ by Bertini's theorem,
and $Y$ is non-singular at the points of $X$ since $X$ is a complete intersection
of~\mbox{$Q$, $Q'$}, and several other quadrics.

Since $\Gamma(X)$ preserves every quadric passing through $X$, it also preserves the quadrics in the pencil generated by $Q$ and $Q'$,
hence $\Gamma(X) \subset \Gamma(Y)$. So, the claim follows from Proposition~\xref{proposition:int-2-quadrics}.
\end{proof}

We can describe the cases when $\Gamma(X) \ne 1$. This, in fact, is equivalent to
``strict semistability'' of~$q$, i.e. to the situation when
\begin{equation*}
V = V_0 \oplus V_1 \oplus \ldots \oplus V_m
\end{equation*}
and $q = q_0 + q_1 + \ldots + q_m$, where
\begin{equation*}
q_i:W \to \Sym^2 V_i^\vee.
\end{equation*}
In the example below all $V_i$
are one-dimensional.

\begin{example}
Let $X\subset\P^n$ be given by $r\le n$ equations
\begin{equation*}
\begin{array}{llll}
\lambda_{10}x_0^2+\lambda_{11}x_1^2+\ldots+\lambda_{1n}x_n^2=0,
&\ldots,
&\lambda_{r0}x_0^2+\lambda_{r1}x_1^2+\ldots+\lambda_{rn}x_n^2=0,
\end{array}
\end{equation*}
where $\lambda_{ij}\in\Bbbk$ are sufficiently general.
Then $X$ is a smooth complete intersection of $r$ quadrics,
and clearly all diagonal matrices with entries $\pm 1$ preserve each of the quadrics.
Therefore, in this case one has $\Gamma(X) \cong \mumu_2^n$.
This shows that the group $\Gamma(X)$ may be nontrivial for any $r$.
\end{example}

Now we will consider intersections of three quadrics. Let
$\Delta$ be a reduced connected curve.
Recall that $\Delta$ is said to be \emph{stable} if its singularities are nodes,
and $\Delta$ has no infinitesimal automorphisms.
The automorphism group of a stable curve is finite~\mbox{\cite[Theorem~1.11]{DeligneMumford1969}}.
Note also that any nodal plane curve of degree at least $4$ is stable (see e.g. \cite[Proposition~2.1]{Hassett1999a}).

\begin{lemma}\label{lemma:CI-3-quadrics}
Let $X\subset\P^n$, $n\ge 6$,
be a smooth complete intersection of three quadrics.
Then $\Aut_W(X)$ acts faithfully on a stable curve. In particular, the group $\Aut(X)$ is finite.
\end{lemma}
\begin{proof}
Let $\Delta \subset \P^2=\P(W)$ be the curve that parameterizes degenerate quadrics passing through $X$.
This curve is usually called the \emph{Hesse curve} of $X$ (see \cite[\S2.2]{Tyurin1975}).
One has
\begin{equation*}
\deg\Delta=n+1\ge 7.
\end{equation*}
The curve $\Delta$ is $\Aut(X)$-invariant.
Since it is not a line, we conclude that $\Aut_W(X)$ acts faithfully on $\Delta$.
It is well known that the curve $\Delta$ is nodal;
this follows, for example, from~\cite[Proposition~1.2(iii)]{Beauville1977}
applied to the quadric bundle over~$\P^2$ that is obtained by blowing up $\P^n$ along $X$. Thus, the curve $\Delta$ is stable.

As we noticed above, stability of $\Delta$ implies finiteness of $\Aut_W(X)$. On the other hand, $\Gamma(X)$ is finite
by Corollary~\xref{corollary:gamma-finite}. So, finiteness of $\Aut(X)$ follows from exact sequence~\eqref{eq:sequence-aut-quadrics}.
\end{proof}

\subsection{Complete intersections in weighted projective spaces}
\label{subsection:WCI}

In this section we discuss the automorphism groups of complete intersections in weighted projective spaces.
Recall that a weighted projective space $\P(a_0,\ldots,a_n)$, $n\ge 1$,
is defined as
\begin{equation*}
\P(a_0,\ldots,a_n)=\operatorname{Proj} \Bbbk\left[x_0,\ldots,x_n\right],
\end{equation*}
where the variables $x_0,\ldots,x_n$ have (positive integer)
weights $a_0,\ldots,a_n$, respectively (see~\cite{Dolgachev-1982}).
Also, it can be constructed as a quotient
\begin{equation*}
\P(a_0,\ldots,a_n)=\big(\mathbb{A}^n\setminus\{0\}\big)/\Bbbk^*, \quad
\lambda\colon (x_0,\ldots,x_n)\mapsto (\lambda^{a_0}x_0,\ldots,\lambda^{a_n}x_n).
\end{equation*}
In a standard way, a weighted projective space $\P=\P(a_0,\ldots,a_n)$
is equipped with
rank $1$ coherent sheaves
$\O_{\P}(m)$, $m\in\Z$. These sheaves are \emph{divisorial} but non-invertible in general
(see \cite[\S1]{Dolgachev-1982}).

Any weighted projective space is isomorphic to a \emph{well-formed}
weighted projective space, i.e.
a weighted projective space $\P(a_0,\ldots,a_n)$ such that
the greatest common divisor of any $n$ among the $n+1$ weights $a_0,\ldots,a_n$ equals~$1$
(see \cite[1.3.1]{Dolgachev-1982} for details).

Let $\Cox(\P)$ be the Cox ring of $\P$, see
\cite[\S5.1 and \S5.2]{CoxLittleShenck} or~\cite[\S\,I.4.1]{Arzhantsev-et-al}
for a definition.

\begin{lemma}\label{lemma:Cox-vs-weights}
Suppose that the weighted projective space $\P=\P(a_0,\ldots,a_n)$
is well-formed. Then the following assertions hold.
\begin{itemize}
\item[(i)] The group $\Cl(\P)\cong\Z$
is generated by the class of~$\O_{\P}(1)$.

\item[(ii)] One has a canonical isomorphism of $\Z$-graded rings
\begin{equation*}
\Cox(\P)\cong \bigoplus_{m\ge 0} H^0(\P, \O_{\P}(m))\cong\Bbbk[x_0,\ldots,x_n],
\end{equation*}
where the weight of the variable $x_i$ is defined to be $a_i$.
\end{itemize}
\end{lemma}

\begin{proof}
We have the standard exact sequence
\begin{equation}
\label{equation-sequence-P-U}
0  \longrightarrow \Z\cdot D_i \longrightarrow \Cl(\P) \longrightarrow \Cl(U_0) \longrightarrow 0,
\end{equation}
where
$$
U_0=\P\setminus D_0\cong \mathbb A^n/\mumu_{a_0}
$$
and the action of $\mumu_{a_0}$ on $\mathbb A^n$
is diagonal with weights~\mbox{$(a_1,\dots,a_n)$}.
By our well-formedness
assumption $\gcd(a_1,\dots,a_n)=1$, i.e. the action of $\mumu_{a_0}$ on $\mathbb A^n$
is free in codimension $1$.
Therefore, $\Cl(U_0)\cong \Z/a_0\Z$ and $\Cl(\P)\cong \Z\oplus T$,
where $T$ is a finite cyclic group whose order divides $a_0$.
By symmetry the order of $T$ divides $a_i$ for all $i$ and again by our well-formedness
assumption $T=0$. Thus $\Cl(\P)\cong \Z$.
Let $D$ be the positive generator of~$\Cl(\P)$ and
let $D_i$ be the effective Weil divisor given by $x_i=0$.
Since $\Cl(U_0)\cong \Z/a_0\Z$, the sequence \eqref{equation-sequence-P-U}
shows $D_0\sim a_0D$ and similarly $D_i\sim a_iD$ for all $i$.
By definition of sheaves $\O_{\P}(m)$ we have $\O_{\P}(a_i)\cong \O_{\P}(D_i)$.
This proves assertion~(i).

Now one can show that  the Serre homomorphism
\[
\Bbbk[x_0,\ldots,x_n]_m \longrightarrow H^0(\P, \O_{\P}(m))
\]
is an isomorphism (see e.g. \cite[1.4.1]{Dolgachev-1982}).
Keeping in mind the definition of the grading on~\mbox{$\Cox(\P)$}, see~\cite[\S5.2]{CoxLittleShenck},
we get the required isomorphism of $\Z$-graded rings, which proves
assertion~(ii).
\end{proof}

\begin{remark}
Lemma~\ref{lemma:Cox-vs-weights}(ii) is given
as \cite[Exercise~5.2.2]{CoxLittleShenck};
the proof is based on~\mbox{\cite[Example~5.1.14]{CoxLittleShenck}} and
\cite[Exercises~4.1.5 and~4.2.11]{CoxLittleShenck}.
Lemma~\ref{lemma:Cox-vs-weights}(i)
is~\cite[Exercise~4.1.5]{CoxLittleShenck}. In both cases
one can find some details clarified
in~\cite{CoxLittleShenck-erratum}.
The proof relies on the description of~$\P$ as a toric variety
given in \cite[Example~3.1.17]{CoxLittleShenck}, which uses
a well-formedness assumption. All the rest is a standard techniques
of working with the divisor class group of a toric variety
based on~\cite[Theorem~4.1.3]{CoxLittleShenck}.
\end{remark}

Consider the polynomial ring $R=\Bbbk[x_0,\dots, x_n]$
as a graded $\Bbbk$-algebra $R=\bigoplus_{i\ge 0} R_i$ with grading given by
$\deg x_i=a_i>0$. In particular, one has $R_0=\Bbbk$.
Denote by $R_{\le m}$ the graded vector subspace
$\bigoplus_{i\le m} R_i\subset R$.

\begin{lemma}
\label{lemma:Levi}
Let $U_m\subset R_m$ be the intersection
of $R_m$ with the subalgebra of~$R$ generated by $R_{\le m-1}$, and
put
\begin{equation*}
k_m=\dim R_m-\dim U_m.
\end{equation*}
Suppose that $R$ is finitely generated, so that
there is a positive $N$ such that $k_m=0$ for~\mbox{$m>N$}.
Put
$$
\Gamma=\GL_{k_1}(\Bbbk)\times\ldots\times\GL_{k_N}.
$$
Then $\Aut(R)$, regarded as the automorphism group
of the \emph{graded} algebra $R$, contains the group $\Gamma$, and any
reductive subgroup of $\Aut(R)$ is isomorphic to a subgroup of $\Gamma$.
\end{lemma}

\begin{proof}
The group $\Aut(R)$ acts on every
vector space $R_m$ so that the subspace $U_m$ is $\Aut(R)$-invariant.
Choose $V_m\subset R_m$ to be a
vector subspace such that
\begin{equation*}
U_m\oplus V_m= R_m.
\end{equation*}
One has $k_m=\dim V_m$.
This gives an obvious action of $\Gamma$ on $R$.

Now let $G\subset\Aut(R)$ be a reductive subgroup.
Then one can choose a $G$-invariant vector subspace
$V_m'\subset R_m$ such that
\begin{equation*}
U_m\oplus V_m'= R_m.
\end{equation*}
Moreover, the action of $G$ on $R$ is recovered from its
action on $\bigoplus V_m'$. Since $V_m'\cong V_m$,
this gives the second assertion of the lemma.
\end{proof}

We will use the abbreviation
\begin{equation*}
(a_1^{k_1},\ldots,a_N^{k_N})=
(\underbrace{a_1,\ldots,a_1}_{k_1\ \text{times}},\ldots,\underbrace{a_N,\ldots,a_N}_{k_N\ \text{times}}),
\end{equation*}
where $k_1,\ldots,k_N$ are allowed to be any non-negative integers.

\begin{proposition}\label{proposition:well-formed}
Suppose that the weighted projective space $\P=\P(a_1^{k_1},\ldots,a_N^{k_N})$
is well-formed.
Let $R_U$ be the unipotent radical of the group $\Aut(\P)$, so that
the quotient
\begin{equation*}
\Aut_{\mathrm{red}}(\P)=\Aut(\P)/R_U
\end{equation*}
is reductive.
Then
\begin{equation*}
\Aut_{\mathrm{red}}\big(\P\big) \cong \big(\GL_{k_1}(\Bbbk) \times \ldots \times \GL_{k_N}(\Bbbk)\big)/\Bbbk^*,
\end{equation*}
where $\Bbbk^*$ embeds into the above product by
\begin{equation}\label{eq:Gm-embedding}
t \mapsto (t^{a_1}{\mathrm{Id}}_{k_1},\ldots,t^{a_N}{\mathrm{Id}}_{k_N}).
\end{equation}
Here $\mathrm{Id}_k$ denotes the identity $k\times k$-matrix.
\end{proposition}
\begin{proof}
Since $\P$ is well-formed, by Lemma~\ref{lemma:Cox-vs-weights}(ii)
one has an isomorphism of $\mathbb{Z}$-graded rings
\begin{equation*}
\Cox(\P)\cong\Bbbk\left[x_1^{(1)},\ldots,x_1^{(k_1)},\ldots,x_N^{(1)},\ldots,x_N^{(k_N)}\right],
\end{equation*}
where the weight of the variable $x_i^{(j)}$ is defined to be $a_i$.
The action of $\Bbbk^*$
on $\Cox(\P)$ defined by~\eqref{eq:Gm-embedding} agrees with this grading.
One has
$$
\mathrm{Spec}\Cox(\P)\cong\mathbb{A}^{k_1+\ldots+k_N}.
$$
Let $\widetilde{\Aut}(\P)$ be the normalizer of $\Bbbk^*$ in
the group
\begin{equation*}
\Gamma=\Aut\big(\mathbb{A}^{k_1+\ldots+k_N}\setminus\{0\}\big)\cong\Aut_0\big(\mathbb{A}^{k_1+\ldots+k_N}\big).
\end{equation*}
Then $\Gamma$ naturally acts on $\Cox(\P)$, that is identified with the ring of regular
functions on~\mbox{$\mathbb{A}^{k_1+\ldots+k_N}\setminus\{0\}$}.
Moreover, $\widetilde{\Aut}(\P)$ is actually a centralizer of $\Bbbk^*$ in $\Gamma$,
since all weights of the action of $\Bbbk^*$ on $\Cox(\P)$ are positive
(and $\Cox(\P)$ splits into a sum of eigen-spaces of~$\Bbbk^*$).
Thus $\widetilde{\Aut}(\P)$ is isomorphic to the group
of graded automorphisms of the ring~\mbox{$\Cox(\P)$} by \cite[Theorem~4.2(iii)]{Cox1995-we}.

According to the Levi decomposition there exists a reductive subgroup
$\tilde L\subset \widetilde{\Aut}(\P)$ such that
$\tilde L\cong \widetilde{\Aut}(\P)/\tilde{R}_U$,
where $\tilde{R}_U$ is the unipotent radical of $\widetilde{\Aut}(\P)$.
By Lemma~\ref{lemma:Levi} one has
\begin{equation*}
\widetilde{\Aut}(\P)/\tilde{R}_U\cong \tilde L\cong
\GL_{k_1}(\Bbbk) \times \ldots \times \GL_{k_N}(\Bbbk).
\end{equation*}
On the other hand, by \cite[Theorem~4.2(ii)]{Cox1995-we}
one has
\begin{equation*}
1\to\Bbbk^*\to \widetilde{\Aut}(\P)\to\Aut(\P)\to 1,
\end{equation*}
and the assertion follows.
\end{proof}

\begin{remark}
The assertion of Proposition~\xref{proposition:well-formed} fails without
the assumption that the weighted projective space $\P$ is well-formed.
One can take the weighted projective line
\begin{equation*}
\P(1,2)\cong\P^1
\end{equation*}
as a counterexample.
\end{remark}

\begin{lemma}\label{lemma:WPS}
Let $\P=\P(a_1^{k_1},\ldots,a_N^{k_N})$ be a well-formed
weighted projective space
such that $a_1 < \ldots < a_{N-1} < a_N$.
Let $G\subset\Aut(\P)$ be a finite subgroup.
Assume that $k_N=1$.
Then for some positive integer $r$ there is a central extension
\begin{equation}\label{equation-extension-WPS}
1\to\mumu_{r}\to \tilde{G}\to G\to 1
\end{equation}
with
\begin{equation*}
\tilde{G}\subset\GL_{k_1}(\Bbbk) \times \ldots \times \GL_{k_{N-1}}(\Bbbk).
\end{equation*}
\end{lemma}
\begin{proof}
Let $R_U \subset \Aut(\P)$
be the unipotent radical of $\Aut(\P)$.
Since any nontrivial element
of a unipotent group has infinite order, the intersection $G \cap R_U$ is trivial,
hence $G$ embeds into the reductive quotient
\begin{equation*}
\Aut_{\mathrm{red}}(\P)=\Aut(\P)/R_U.
\end{equation*}

Consider the embedding
\begin{equation*}
\Bbbk^*\hookrightarrow \GL_{k_1}(\Bbbk) \times \ldots \times \GL_{k_N}(\Bbbk)
\end{equation*}
given by formula~\eqref{eq:Gm-embedding}.
By Proposition~\xref{proposition:well-formed} one has
\begin{equation*}
\Aut_{\mathrm{red}}\big(\P\big) \cong \big(\GL_{k_1}(\Bbbk) \times \ldots \times \GL_{k_N}(\Bbbk)\big)/\Bbbk^*,
\end{equation*}

Consider the subgroup
\begin{equation*}
\SL_{k_1,\ldots,k_N}(\Bbbk)=
\Bigl
\{ (g_1,\ldots,g_N) \in
\GL_{k_1}(\Bbbk) \times \ldots\times\GL_{k_N}(\Bbbk)\ \big|\ \prod_{i=1}^N \det(g_i) = 1 \Bigr\}.
\end{equation*}
in $\GL_{k_1}(\Bbbk) \times \ldots\times\GL_{k_N}(\Bbbk)$.
This group intersects with the above $\Bbbk^*$ along
$\mumu_r \subset \Bbbk^*$, where~\mbox{$r=\sum_{i=1}^N a_ik_i$}.
Moreover, we have
\begin{equation*}
\Aut_{\mathrm{red}}(\P) \cong \SL_{k_1,\ldots,k_N}(\Bbbk)/\mumu_r
\end{equation*}
Denote the preimage of $G$ in $\SL_{k_1,\ldots,k_N}(\Bbbk)$ by $\tilde{G}$.
Then there is a central extension~\eqref{equation-extension-WPS}.
So, it remains to notice that as~\mbox{$k_N=1$}, we have
\begin{equation*}
\SL_{k_1,\ldots,k_{N-1},k_N}(\Bbbk) \cong
\GL_{k_1}(\Bbbk) \times \ldots \times \GL_{k_{N-1}}(\Bbbk).\qedhere
\end{equation*}
\end{proof}

Let $Y$ be a normal projective variety and let $A$ be a Weil divisor on $Y$.
Put
\begin{equation*}
R_m(Y,A)=H^0\big(Y, \CO_Y(mA)\big).
\end{equation*}
Then
\[
R(Y,A)=\bigoplus\limits_{m=0}^{\infty} R_m(Y,A)
\]
has a natural structure of a graded $\Bbbk$-algebra.

\begin{remark}\label{remark:ample}
If the divisor $A$ is ample, then the algebra $R(Y,A)$ is finitely generated,
and~\mbox{$Y\cong\operatorname{Proj}(R(Y,A))$}.
\end{remark}

As before, define a (graded) vector subspace
\begin{equation*}
R_{\le N}(Y,A)=\bigoplus\limits_{m\le N} R_m(Y,A)\subset R(Y,A).
\end{equation*}

\begin{lemma}\label{lemma:action-on-algebra}
Let $Y$ be a normal projective variety with an action of
a group~$\Gamma$, and~$A$ be an ample Weil divisor on $Y$.
Suppose that the class of $A$ in $\Cl(Y)$ is $\Gamma$-invariant.
Then for some positive integer $r$ there is a central extension
\begin{equation}\label{eq:central-ext-Gamma}
1\to \mumu_r\to\tilde{\Gamma}\to \Gamma\to 1
\end{equation}
such that $\tilde{\Gamma}$ acts on the algebra $R(Y,A)$,
and this action induces the initial action of $\Gamma$ on~$Y$.
\end{lemma}
\begin{proof}
One has
\begin{equation*}
Y\cong\operatorname{Proj}\big(R(Y,A)\big).
\end{equation*}
The algebra $R(Y,A)$ is generated by its vector subspace
$R_{\le N}(Y,A)$ for some $N$.
Now it suffices to define an action of an appropriate central
extension~\eqref{eq:central-ext-Gamma}
on $R_m(Y,A)$ for each $1\le m\le N$ (see e.g.~\cite[\S3.1]{Kuznetsov-Prokhorov-Shramov}).
Taking $r$ to be sufficiently divisible, we may assume that $\tilde{\Gamma}$
acts on the whole vector space $R_{\le N}(Y,A)$, which gives the desired action on
the algebra~\mbox{$R(Y,A)$}.
\end{proof}

\begin{lemma}[{cf. Lemma~\ref{lemma:Levi}}]
\label{lemma:action-on-algebra-and-WPS}
Let $Y$ be a normal projective variety with an action of
a finite group~$\Gamma$, and $A$ be an ample Weil divisor on $Y$.
Suppose that the class of $A$ in $\Cl(Y)$ is $\Gamma$-invariant
and $R(Y,A)$ is generated by $R_{\le N}(Y,A)$.
For $1\le m\le N$ let $U_m\subset R_m(Y,A)$ be the intersection
of $R_m(Y,A)$ with the subalgebra of~\mbox{$R(Y,A)$} generated by $R_{\le m-1}(Y,A)$, and
put
\begin{equation*}
k_m=\dim R_m(Y,A)-\dim U_m.
\end{equation*}
Then there is a natural embedding
\begin{equation*}
Y\hookrightarrow\P=\P\big(1^{k_1},\ldots,N^{k_N}\big)
\end{equation*}
and an action of $\Gamma$ on $\P$ that induces the initial
action of $\Gamma$ on~$Y$.
\end{lemma}
\begin{proof}
By Lemma~\xref{lemma:action-on-algebra} there is an action of a finite central extension
$\tilde{\Gamma}$ of $\Gamma$ on $R(Y,A)$ that induces the initial action of
$\Gamma$ on~$Y$. In particular, the group $\tilde{\Gamma}$ acts on every
vector space $R_m(Y,A)$.
Obviously, the subspace $U_m$ is $\tilde{\Gamma}$-invariant.
Choose $V_m\subset R_m(Y,A)$ to be a $\tilde{\Gamma}$-invariant
vector subspace such that
\begin{equation*}
U_m\oplus V_m= R_m(Y,A).
\end{equation*}
One has $k_m=\dim V_m$. Let $x_m^{(1)},\ldots,x_m^{(k_m)}$
be a basis in~$V_m$. Then there is a natural surjection
\begin{equation*}
\Bbbk\left[x_1^{(1)},\ldots,x_1^{(k_1)},\ldots,x_N^{(1)},\ldots,x_N^{(k_N)}\right]\to R(Y,A)
\end{equation*}
that induces an embedding
\begin{equation*}
Y\cong\operatorname{Proj}\big(R(Y,A)\big)\hookrightarrow
\operatorname{Proj}\left(\Bbbk\left[x_1^{(1)},\ldots,x_1^{(k_1)},\ldots,x_N^{(1)},\ldots,x_N^{(k_N)}\right]\right)=\P.
\end{equation*}
Note that the action of $\tilde{\Gamma}$ on $\P$ factors through
the action of $\Gamma$ on $\P$, and this action clearly induces
the initial action of $\Gamma$ on~$Y$.
\end{proof}


\newcommand{\etalchar}[1]{$^{#1}$}
\def\cprime{$'$}

\end{document}